\documentclass[11pt]{article}
\usepackage[utf8]{inputenc}
\usepackage{cite}

\title{{\Large \vspace{-0.5cm} 
On symplectic aspects of $SU(2)$ character varieties for punctured surfaces  }}
\author{Aliakbar Daemi and Christopher Scaduto} 
\date{}

\usepackage[a4paper, total={6in, 8.5in}]{geometry}
\usepackage{graphicx}
\usepackage{times}
\usepackage{amssymb}
\usepackage{tikz-cd}
\usepackage{titling}
\usepackage{mathrsfs} 
\usepackage{booktabs}
\usepackage[all]{xy}
\usepackage{amsthm}
\usepackage{diagbox}
\usepackage{tabularx}
\usepackage{amscd}
\usepackage{caption}
\usepackage{nicefrac}
\usepackage{amsmath}
\usepackage{mathtools}
\usepackage{tikz}
\usepackage{mathabx}
\usepackage{dsfont}
\usepackage{lipsum}
\usepackage{mwe}
\usepackage{slashed}
\usepackage{rotating}
\usepackage{subcaption}
\usepackage[colorlinks,pagebackref,hypertexnames=false]{hyperref} \usepackage[alphabetic,backrefs,msc-links]{amsrefs}
\usepackage{epstopdf,pinlabel}
\usepackage{enumitem} 
\definecolor{mint}{HTML}{239B56}
\hypersetup{
    colorlinks=true,
    linkcolor=blue,
    filecolor=magenta,      
    urlcolor=cyan,
    citecolor=mint
}
\usepackage{pgfplots}
\pgfplotsset{compat=newest}
\usetikzlibrary{shapes.geometric}

\definecolor{greenish}{rgb}{0.01, 0.75, 0.24}
\definecolor{blueish}{rgb}{0.0, 0.72, 0.92}
\definecolor{orangeish}{rgb}{1.0, 0.55, 0.0}

\newcolumntype{Y}{>{\centering\arraybackslash}X}

\usepackage{sectsty}

\sectionfont{\fontsize{12}{15}\selectfont}

\newcommand{\R}{\mathbb{R}}
\newcommand{\C}{\mathbb{C}}
\newcommand{\Z}{\mathbb{Z}}
\newcommand{\F}{\mathbb{F}}
\newcommand{\Q}{\mathbb{Q}}

\newcommand{\bA}{{\bf A}}

\newtheorem{theorem}{Theorem}[section]
\newtheorem{prop}[theorem]{Proposition}
\newtheorem{lemma}[theorem]{Lemma}

\newtheorem{corollary}[theorem]{Corollary}
\newtheorem{remark}[theorem]{Remark}

\newcommand{\Addresses}{{
 \bigskip
 \footnotesize
 Aliakbar Daemi, \textsc{Department of Mathematics, Washington University in St. Louis, One Brookings drive, Room 212,
 St. Louis, MO 63130}\par\nopagebreak
 \textit{E-mail address}: \texttt{adaemi@wustl.edu}
 \vspace{.4cm}

Christopher Scaduto, \textsc{Department of Mathematics, University of Miami, 1365 Memorial Dr 515, Coral Gables, FL 33124}\par\nopagebreak
 \textit{E-mail address}: \texttt{cscaduto@miami.edu}
}}

\begin{document}

\maketitle

\vspace{-0.75cm}

\begin{abstract}{
For a surface with an odd number of punctures, the moduli space of flat $SU(2)$ connections with traceless holonomy around each puncture is a symplectic manifold. When the moduli space is nonempty, there is a natural homomorphism from the mapping class group of the punctured surface to the symplectic mapping class group of this moduli space. It is shown that this homomorphism is injective if and only if the dimension of the moduli space is greater than $2$. This generalizes work of Seidel and Wehrheim--Woodward. Also given is a complete classification of Lagrangian spheres in the projective plane blown up at $5$ points with its monotone symplectic structure, which is the moduli space for the 5-punctured sphere. Furthermore, it is determined when two such Lagrangian spheres can be displaced by a symplectic isotopy. Results are also obtained regarding Lagrangian spheres in the intersection of two quadrics in $\mathbb{C}\mathbb{P}^5$. The proofs involve instanton Floer theory and results on Heegaard splittings. A main technical result establishes the approximation of any Hamiltonian isotopy of the $SU(2)$ moduli space by holonomy perturbations which are used in instanton homology. 
}
\end{abstract}

\vspace{.5cm}
\hypersetup{linkcolor=black}
\tableofcontents
\hypersetup{linkcolor=blue}

%!TEX root = main.tex

\section{Introduction}

Let $\Sigma_{g,n}$ be a genus $g$ Riemann surface with $n$ marked points, where $n$ is odd. The moduli space $M_{g,n}$ of flat $SU(2)$ connections on the punctured surface, with traceless holonomy around each puncture, is a smooth monotone symplectic manifold of dimension $6g+2n-6$. A concrete description of $M_{g,n}$ is as the space of conjugacy classes of homomorphisms from the fundamental group of the punctured surface to $SU(2)$ for which each loop encircling a puncture is sent to a traceless element. The moduli space $M_{g,n}$ may also be identified with a moduli space of rank $2$ stable parabolic bundles on $\Sigma_{g,n}$, see for example \cite{mehta-seshadri}. In this paper, we obtain information on the symplectic mapping class group of $M_{g,n}$ and about Lagrangians in $M_{g,n}$ using a combination of tools from instanton Floer theory and 3-manifold topology. Results are also obtained for a variation of $M_{g,n}$ given by the {\emph{odd character variety}} of $\Sigma_g$, the moduli space $M_g^{\text{odd}}$ of flat $SU(2)$ connections on $\Sigma_{g,n}$ with $-1$ holonomy around each puncture. It is a smooth monotone symplectic manifold of dimension $6g-6$ and does not depend on $n$, apart from $n$ being odd.

\subsubsection*{The symplectic mapping class group of $M_{g,n}$}

A diffeomorphism $\phi$ of $\Sigma_{g,n}$ induces a symplectomorphism of $M_{g,n}$ which only depends on the class of $\phi$ in the mapping class group $\text{Mod}(\Sigma_{g,n})$. Consider the induced homomorphism
\begin{equation}\label{eq:homfrommappingclassgroup}
	\text{Mod}(\Sigma_{g,n}) \to \pi_0 \text{Symp}(M_{g,n})
\end{equation}
to the symplectic mapping class group of $M_{g,n}$. In fact, there is a group extension
\begin{equation*}
\begin{tikzcd}
1   \arrow[r] &  H^1(\Sigma_{g,n};\Z/2) \arrow[r] &   \Gamma_{g,n}  \arrow[r]  & \text{Mod}(\Sigma_{g,n})   \arrow[r]  & 1
\end{tikzcd}\label{eq:defnofgamma}
\end{equation*}
where $\Sigma_{g,n}$ in the first term is viewed as a punctured surface and $ \Gamma_{g,n}$ is the semi-direct product associated to the action of  $\text{Mod}(\Sigma_{g,n})$ on $H^1(\Sigma_{g,n};\Z/2)$. Then \eqref{eq:homfrommappingclassgroup} extends to a homomorphism
\[
	\rho_{g,n}: \Gamma_{g,n} \to \pi_0 \text{Symp}(M_{g,n}).
\]

\begin{theorem}\label{thm:mappingclassgroupaction}
		For $n$ odd and $3g+n>4$, the homomorphism $\rho_{g,n}$ is injective.
\end{theorem}

If $n$ is odd and $3g+n\leq 4$, then $(g,n)$ is one of $(0,1)$, $(0,3)$, or $(1,1)$. The moduli space $M_{0,1}$ is empty, and so $\rho_{0,1}$ is not even defined.  In the other cases, $M_{0,3}$ is a point and $M_{1,1}$ is a $2$-sphere, so that the associated symplectic mapping class groups are trivial, while $\text{Mod}(\Sigma_{0,3})\cong S_3$ and $\text{Mod}(\Sigma_{1,1})\cong SL(2,\Z)$ are nontrivial. Thus in these cases $\rho_{g,n}$ is not injective.

In previous work \cite{dsodd}, the authors proved a version of Theorem \ref{thm:mappingclassgroupaction} for $M_g^{\text{odd}}$, answering a question of Dostoglou and Salamon \cite{dostoglou-salamon}. The proof of Theorem \ref{thm:mappingclassgroupaction} also adapts to provide an alternative and more self-contained proof of the results in \cite{dsodd}, as discussed below.

Theorem \ref{thm:mappingclassgroupaction} is a generalization of a result due to Seidel \cite{seidel-4d}, who proved injectivity in the case $(g,n)=(0,5)$, that of the $5$-punctured sphere. In fact, in this case $\rho_{0,5}$ is also surjective and hence an isomorphism, a result that follows from the work of Evans \cite{evans-stein}; see also \S \ref{sec:lagspherecomp}. This case is exceptional in that $M_{0,5}$ may be identified as a symplectic manifold with the del Pezzo surface given by the projective plane blown up at $5$ generic points:
\[
	M_{0,5} = \text{Bl}_5\mathbb {CP}^2
\]
Here the symplectic structure on $\text{Bl}_5\mathbb {CP}^2$ is monotone, given by an anticanonical K\"{a}hler form, which is unique up to symplectomorphism. In another direction, building on Seidel's work, Wehrheim and Woodward \cite{ww-floerfield} previously exhibited many elements of $\Gamma_{0,n}$ not in the kernel of $\rho_{0,n}$.

\subsubsection*{Lagrangian spheres in the $5$-point blowup of $\mathbb{CP}^2$}

 The methods of this paper also yield information about Lagrangians in $M_{g,n}$. The information in the case $M_{0,5}=\text{Bl}_5\mathbb {CP}^2$ is rather definitive. To set the stage, for a closed symplectic manifold $M$ define the {\emph{Lagrangian sphere complex}} to be the simplicial complex $\mathcal{L}(M)$ whose vertices are Hamiltonian isotopy classes of embedded Lagrangian spheres in $M$; a collection of distinct vertices span a simplex if and only if they can be represented by pairwise disjoint Lagrangians. Note that Hamiltonian isotopy and symplectic isotopy are equivalent notions for $M_{g,n}$ as it is simply-connected.
 
Define the {\emph{signed curve complex}} $\widetilde{\mathcal{C}}(\Sigma_{g,n})$ to be the simplicial complex whose vertices are pairs $(\gamma,\epsilon)$ where $\gamma$ is the isotopy class of an essential and non-peripheral simple closed curve on the punctured surface $\Sigma_{g,n}$ and $\epsilon\in \{-,+\}$; a collection of distinct vertices span a simplex if and only if the isotopy classes can be represented by pairwise disjoint curves, a condition independent of the signs. The same definition without the signs $\epsilon$ gives the ordinary curve complex $\mathcal{C}(\Sigma_{g,n})$. Note $\dim \mathcal{C}(\Sigma_{g,n})=3g +n -4$ and $\dim \widetilde{\mathcal{C}}(\Sigma_{g,n}) = 2\dim \mathcal{C}(\Sigma_{g,n}) + 1$. See Figure \ref{fig:maxsimplex5puncturedsphere}.

 \begin{theorem}\label{thm:lagrangianspheres}
		There is a simplicial isomorphism between the signed curve complex of the sphere with $5$ marked points and the Lagrangian sphere complex of the del Pezzo surface $\textup{Bl}_5\mathbb {CP}^2$:
		\[
			 \widetilde{\mathcal{C}}(\Sigma_{0,5}) \cong \mathcal{L}(\textup{Bl}_5\mathbb {CP}^2)
		\]
\end{theorem}
 
 \noindent  In fact, there is an action of $\pi_0\text{Symp}(\textup{Bl}_5\mathbb {CP}^2)$ on 
 $\mathcal{L}(\textup{Bl}_5\mathbb {CP}^2 )$ and one of $\Gamma_{0,5}$ on $\widetilde{\mathcal{C}}(\Sigma_{0,5})$, and these actions are intertwined under the simplicial equivalence of Theorem \ref{thm:lagrangianspheres} by the isomorphism $\rho_{0,5}$. Given a vertex $(\gamma,\epsilon)$ of the signed curve complex where $\epsilon=+$ (resp. $\epsilon=-$), the associated Lagrangian $\Lambda_\gamma^\epsilon$ in $M_{0,5}=\text{Bl}_5\mathbb {CP}^2$ is defined using the space of flat $SU(2)$ connections on the $5$-punctured sphere that extend to a flat $SO(3)$ connection on the trivial (resp. non-trivial) $SO(3)$-bundle over the $3$-manifold obtained by attaching a $2$-handle along $\gamma$.
 
 A key ingredient in the proof of Theorem \ref{thm:lagrangianspheres} is that  $\pi_0\text{Symp}(\textup{Bl}_5\mathbb {CP}^2)$ acts transitively on the symplectic isotopy classes of Lagrangian spheres, a result due to Borman--Li--Wu \cite[Corollary 1.2]{borman-li-wu}. From this and the fact that $\rho_{0,5}$ is an isomorphism, it can be seen that every Lagrangian sphere in $M_{0,5}$ is symplectic isotopic to one of the Lagrangians $\Lambda_{\gamma}^\epsilon$. The new input of the current paper which leads to Theorem \ref{thm:lagrangianspheres} is a method that determines exactly when two such Lagrangians are symplectic isotopic and when they can be displaced by a symplectic isotopy.
 
As a point of comparison,
for any symplectic structure on either $S^2\times S^2$ or $\text{Bl}_n \mathbb {CP}^2 $ with $n\leq 4$, there exist only finitely many Lagrangian spheres up to symplectic isotopy, as follows from \cite[Corollary 1.5]{li-li-wu}; see also the earlier works \cite[Theorem 1.4]{evans-delpezzo} and \cite{hind}.
As far as the authors are aware, Theorem \ref{thm:lagrangianspheres} gives the first classification of Lagrangian spheres, up to symplectic isotopy, in a symplectic manifold that has dimension bigger than $2$ and where the number of Lagrangian spheres up to symplectic isotopy is infinite.  
 
 \begin{figure}[t]
\centering
\labellist
  \pinlabel {${\color{red}(\gamma_1,+)}$} [r] at 110 290
    \pinlabel {${\color{blue}(\gamma_1,-)}$} [r] at 178 340
      \pinlabel {${\color{red}(\gamma_2,+)}$} [l] at 336 290
    \pinlabel {${\color{blue}(\gamma_2,-)}$} [l] at 270 340
\endlabellist
\includegraphics[scale=0.45]{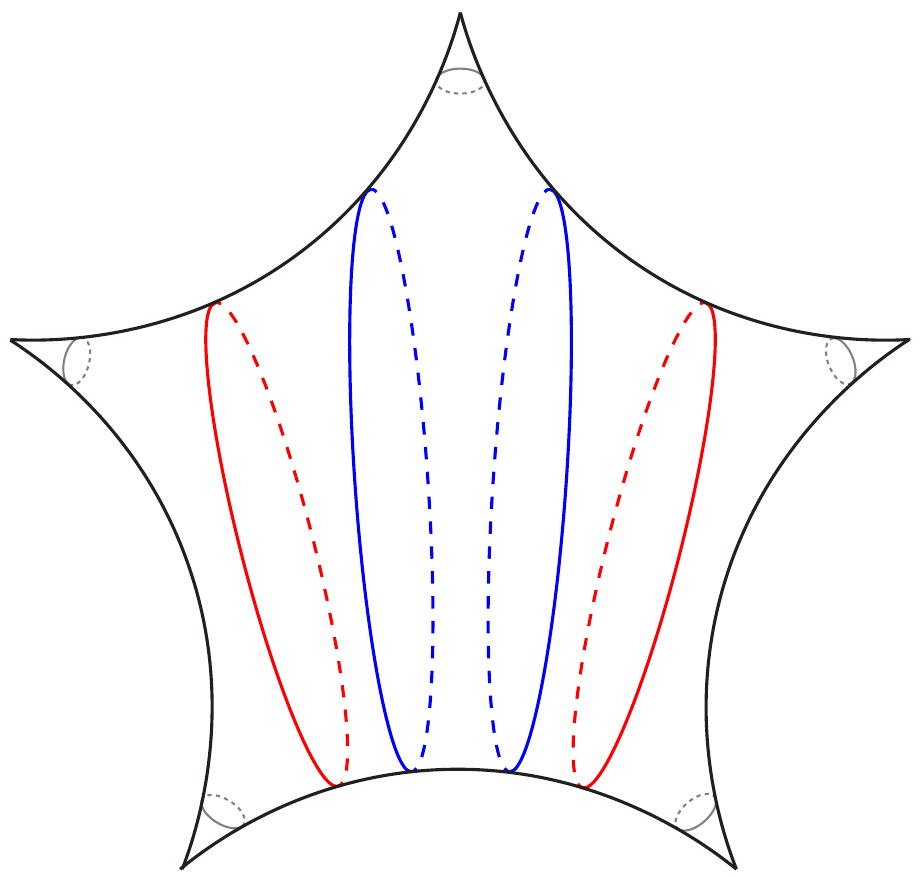}
\caption{\small The signed curve complex $\dim \widetilde{\mathcal{C}}(\Sigma_{0,5})$ has dimension $3$. A maximal simplex is determined by four vertices $(\gamma_1,\pm )$, $(\gamma_2,\pm )$ where $\gamma_1$ and $\gamma_2$ are represented by disjoint essential simple closed curves. }
\label{fig:maxsimplex5puncturedsphere}
\end{figure}
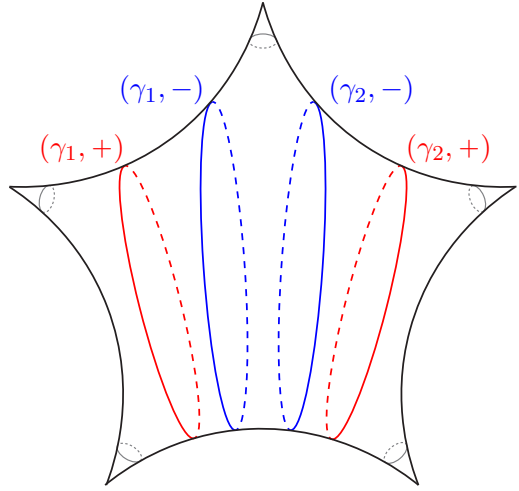

The signed curve complex of $\Sigma_{g,n}$ is closely related to its ordinary curve complex. The ordinary curve complex embeds inside $\widetilde{\mathcal{C}}(\Sigma_{g,n})$ as the subcomplex determined by vertices $(\gamma,\epsilon)$ with $\epsilon=+1$, and $\widetilde{\mathcal{C}}(\Sigma_{g,n})$ deformation retracts onto this subcomplex. Thus Theorem \ref{thm:lagrangianspheres} gives:

 \begin{corollary}\label{cor:lagrangianspheres}
		There is a homotopy equivalence $\mathcal{L}(\textup{Bl}_5\mathbb {CP}^2 ) \simeq {\mathcal{C}}(\Sigma_{0,5})$.
\end{corollary}

 Theorem \ref{thm:lagrangianspheres} may be combined with known results on curve complexes to yield interesting geometric information on $\mathcal{L}(\textup{Bl}_5\mathbb {CP}^2)$. For example, Masur and Minsky showed in \cite{masur-minsky} that for $3g+n-4>0$, the curve complex of $\Sigma_{g,n}$ is Gromov hyperbolic and has infinite diameter. The inclusion of the ordinary curve complex into the signed curve complex mentioned in the previous paragraph is easily seen to be a quasi-isomorphism, and so we obtain:

\begin{corollary}\label{cor:diameterdelpezzo}
	$\mathcal{L}(\textup{Bl}_5\mathbb {CP}^2)$ is Gromov hyperbolic and has infinite diameter.
\end{corollary}

	Theorem \ref{thm:lagrangianspheres} implies that two Lagrangian spheres $\Lambda_{\gamma}^\epsilon$ and $\Lambda_{\gamma'}^{\epsilon'}$ in $\textup{Bl}_5\mathbb {CP}^2$ can be displaced by a Hamiltonian isotopy if and only if either $\gamma$ and $\gamma'$ are isotopic and $\epsilon\neq \epsilon'$, or $\gamma$ and $\gamma'$ are not isotopic and can be displaced by an isotopy. A natural question arises as to whether there is a general relationship between the number of possible intersection points between $\Lambda_{\gamma}^\epsilon$ and $\Lambda_{\gamma'}^{\epsilon'}$ and $i(\gamma,\gamma')$, the minimal number of intersection points between $\gamma$ and $\gamma'$ among all isotopies of these curves in $\Sigma_{0,5}$. The following confirms this for a large class of examples. 
	
	\begin{theorem}\label{thm:preciseintersectionresult}
		Suppose the signed curves $(\gamma,\epsilon)$ and $(\gamma',\epsilon')$ are not isotopic, and there is an essential non-peripheral simple closed curve in $\Sigma_{0,5}$ which is disjoint from $\gamma$ and $\gamma'$. Then
		\begin{equation}\label{eq:moreexactdeterminationofintnumbers}
			\mathfrak{n}(\Lambda_{\gamma}^{\epsilon},\Lambda_{\gamma'}^{\epsilon'}) = \frac{1}{2} i(\gamma,\gamma'),
		\end{equation}
		where the left side is the the minimal number of intersection points of $\Lambda_{\gamma}^\epsilon$ and $\Lambda_{\gamma'}^{\epsilon'}$ in $\textup{Bl}_5\mathbb {CP}^2$ among all Hamiltonian isotopies for which their intersection is transverse, see  \eqref{eq:minimallagrangianintersectiondefn}.
	\end{theorem}
	
\noindent We remark that there are pairs $(\gamma,\epsilon)$ and $(\gamma',\epsilon')$ satisfying the hypotheses of the proposition for which the quantities in \eqref{eq:moreexactdeterminationofintnumbers} are any given non-negative integer. The proof of Theorem \ref{thm:preciseintersectionresult} uses computations of instanton homology for two-bridge links. More general results relevant to the moduli spaces $M_{g,n}$ are given in \S \ref{sec:intersectionsandsu2simple}.

\subsubsection*{Lagrangian spheres in the intersection of two quadrics in $\mathbb{CP}^5$}

Assume $n$ is odd. Recall that the odd character variety $M_g^{\text{odd}}$ is the moduli space of flat $SU(2)$ connections on $\Sigma_{g,n}$ with $-1$ holonomy around each puncture. Newstead \cite{newstead} and Narasimhan and Ramanan \cite{narasimhan-ramanan} proved that the genus $2$ moduli space may be identified as
\[
	M_2^{\text{odd}} \cong Q_1 \cap Q_2 \subset \mathbb{C}\mathbb{P}^5
\]
where $Q_1$ and $Q_2$ are generic quadrics in $ \mathbb{C}\mathbb{P}^5$. For our purposes, this identification shall be understood as one of $6$-dimensional monotone symplectic manifolds. The methods that yield the above results for $M_{0,5}=\text{Bl}_5 \mathbb{C}\mathbb{P}^2 $ also give information about Lagrangian spheres in $Q_1\cap Q_2$. 

Write $\Sigma_g$ for the underlying closed surface of $\Sigma_{g,n}$ where we forget the marked points. Let $\gamma$ and $\gamma'$ be simple closed curves on $\Sigma_{g}$ missing the marked points which are essential on $\Sigma_g$, and choose $\epsilon,\epsilon'\in \{-,+\}$. Declare $(\gamma,\epsilon)\sim (\gamma',\epsilon')$ if and only if $\gamma$ is isotopic to $\gamma'$ in $\Sigma_g$ by an isotopy that crosses the set of punctures $m$ times (each time crossing a single puncture), where the parity of $m$ agrees with $\epsilon\epsilon'$. Define the {\emph{sign-twisted curve complex}} ${\mathcal{C}}^{\textup{tw}}(\Sigma_{g,n})$ to be the simplicial complex whose vertices are the equivalence classes $[\gamma,\epsilon]$ defined above; a collection of distinct vertices span a simplex if and only if the curves in the classes may be represented by pairwise disjoint curves. 

The simplicial isomorphism type of the sign-twisted complex only depends on the parity of $n$, which is fixed to be odd. Denote by $\mathcal{N}(\Sigma_g)$ the subcomplex of ${\mathcal{C}}^{\textup{tw}}(\Sigma_{g,n})$ spanned by the vertices whose underlying curves are non-separating.

 \begin{theorem}\label{thm:lagrangianspheresinquadricintersection}
		There is a simplicial injection from the subcomplex of non-separating curves in the sign-twisted curve complex of $\Sigma_2$ to the Lagrangian sphere complex of $Q_1\cap Q_2$:
		\begin{equation}\label{eq:genus2injectionintro}
			\mathcal{N}(\Sigma_2) \hookrightarrow  \mathcal{L}(Q_1\cap Q_2)
	\end{equation}
\end{theorem}

\noindent As in the case of Theorem \ref{thm:lagrangianspheres}, there is an action of $\pi_0 \textup{Symp}(Q_1\cap Q_2)$ on $\mathcal{L}(Q_1\cap Q_2)$ and one of an extension $\widehat{\Gamma}_2$ of the mapping class group of $\Sigma_2$ on $\mathcal{N}(\Sigma_2)$, and the injection of Theorem \ref{thm:lagrangianspheresinquadricintersection} is equivariant with respect to an injective homomorphism from $\widehat{\Gamma}_2$ to $\pi_0 \textup{Symp}(M_{2}^\text{odd})$.

 A natural question is whether \eqref{eq:genus2injectionintro} is an isomorphism, similar to Theorem \ref{thm:lagrangianspheres}. This would follow if one could show that the analogue of the homomorphism $\rho_{g,n}$ for $M_2^{\text{odd}}$ is an isomorphism (it is known to be injective \cite{smith,dsodd}), and that the symplectomorphism group of $Q_1\cap Q_2$ acts transitively on Lagrangian spheres, up to symplectic isotopy.
 
A standard argument shows that the natural map from $\mathcal{N}(\Sigma_2)$ to the ordinary curve complex of $\Sigma_2$ is a quasi-isomorphism, and thus similar to Corollary \ref{cor:diameterdelpezzo}, the main result of \cite{masur-minsky} implies:

\begin{corollary}\label{cor:diameterquadrics}
	$\mathcal{L}(Q_1\cap Q_2)$ has infinite diameter.
\end{corollary}

An analogue of Theorem \ref{thm:preciseintersectionresult} is also proved in this setting, yielding explicit computations for minimally intersecting Lagrangian spheres in $M_2^{\text{odd}}$. See Theorem \ref{thm:int-odd-g=2}.

\subsubsection*{Hamiltonian diffeomorphisms of $M_{g,n}$ and holonomy perturbations}

Central to the proofs of the above theorems are several results that relate the symplectic topology of $M_{g,n}$ to instanton Floer theory for links in $3$-manifolds. The first result concerns the relationship between Hamiltonian diffeomorphisms of the moduli space $M_{g,n}$ and the class of perturbations that are typically used in instanton Floer theory. These are called {\emph{holonomy perturbations}}, and were used by Floer in his original construction of instanton homology for integer homology $3$-spheres \cite{ floer-zhs3} and also by Taubes in his work on the Casson invariant \cite{taubes-casson}.

We consider holonomy perturbations in the thickened punctured surface $[-1,1]\times \Sigma_{g,n}$. For our purposes, such a holonomy perturbation is determined by a smooth function $\eta:\R^k\to \R$ and a collection of $k$ embedded loops $\Gamma=\{\gamma_1,\ldots,\gamma_k\}$ in $(-1,1)\times \Sigma_{g,n}$ that are pairwise disjoint and come equipped with framings. To this data is associated a moduli space $\mathfrak{X}_{\eta,\Gamma}([-1,1]\times \Sigma_{g,n})$ of $SU(2)$ connections on $[-1,1]\times \Sigma_{g,n}$ which have traceless holonomy around the punctures and which satisfy a perturbation of the flatness equation; elements of this moduli space have the property that they are flat away from a regular neighborhood of the loops $\Gamma$. In particular, restriction of connections to the boundary $\{\pm 1\}\times \Sigma_{g,n}$ determines maps $r_{\pm}$:

\begin{equation*}
\begin{tikzcd}[column sep=small, row sep=large]
& \mathfrak{X}_{\eta,\Gamma}([-1,1]\times \Sigma_{g,n}) \arrow[dl, "r_-"'] \arrow[dr, "r_+"] & \\
M_{g,n} & & M_{g,n}
\end{tikzcd}\label{eq:bundlecorrespondenceintro}
\end{equation*}

\begin{theorem}\label{thm:hamiltoniandiffeo}
	Assume $n$ is odd. Given any Hamiltonian diffeomorphism $\phi$ of $M_{g,n}$ there exists holonomy perturbation data $\eta,\Gamma$ such that $r_{\pm}$ are diffeomorphisms and $r_+\circ r_-^{-1}$ is a Hamiltonian diffeomorphism which is arbitrarily close to $\phi$ in the $C^\infty$-topology.
\end{theorem}

In short, any Hamiltonian diffeomorphism of $M_{g,n}$ can be $C^\infty$-approximated by a holonomy perturbation. The case in which $\Gamma$ consists of a single embedded loop on $\{0\}\times \Sigma_{g,n}$ was studied by Herald and Kirk \cite{herald-kirk}, and they identified $r_+\circ r_-^{-1}$ in this case with Goldman's Hamiltonian twist flow on $M_{g,n}$ studied in \cite{goldman-symplectic}. The proof of Theorem \ref{thm:hamiltoniandiffeo} involves understanding the more general case where $\Gamma$ is a collection of disjoint loops embedded in $\{0\}\times \Sigma_{g,n}$. 

A precursor to Theorem \ref{thm:hamiltoniandiffeo} is the work of Zentner \cite{zentner}, who considered the case $M_{1,0}$, the $SU(2)$ character variety of a $2$-torus. This is a pillowcase, a sphere with $4$ branch points of order $2$, but one can equally well work with $\Z/2$-equivariant Hamiltonian diffeomorphisms on its double branched covering $T^2$. Observing that the Hamiltonian diffeomorphisms of $M_{1,0}$ induced by $\Z/2$-equivariant shear maps on $T^2$ can be realized by holonomy perturbations, Zentner reduced his problem to one of $C^0$-approximating Hamiltonian diffeomorphisms of $T^2$ by compositions of shear maps, see \cite[Theorem 3.3]{zentner}; the regularity was later improved \cite{sivek-zentner}.

The work of Berger and Turaev \cite{bt} shows, more generally, that any Hamiltonian diffeomorphism of a torus $T^{2n}=\R^{2n}/\Z^{2n}$ with its standard symplectic form can be smoothly approximated by compositions of shears; in fact, only horizontal and vertical shears are needed. Our proof of Theorem \ref{thm:hamiltoniandiffeo} adapts the work of \cite{bt} in several key steps.

Theorem \ref{thm:hamiltoniandiffeo} also holds for moduli spaces of flat $SU(2)$ connections on $\Sigma_{g,n}$ with more general conjugacy class conditions for the holonomies around punctures; see Theorem \ref{thm:hamiltoniandiffeo-general}.

\subsubsection*{Lagrangians in $M_{g,n}$ and instanton homology}

The next result establishes a relationship between the intersections of certain pairs of Lagrangians in $M_{g,n}$ and Kronheimer and Mrowka's singular instanton homology for links. Consider a triple $(Y,L,w)$ where $Y$ is a closed, oriented, connected $3$-manifold, $L$ is a link in $Y$, and $w$ is an embedded compact $1$-manifold in $Y$ such that $w\cap L = \partial w$. Call $(Y,L,w)$ an {\emph{admissible link}} if there is an embedded oriented surface $\Sigma\subset Y$ transverse to $L\cup w$ and satisfying either (i) $|\Sigma\cap L|$ is odd or (ii) $\Sigma\cap L=\emptyset$ and $|\Sigma\cap w|$ is odd. For such triples Kronheimer and Mrowka define
\[
    I(Y,L,w),
\]
 the singular instanton Floer homology, an affinely $\Z/4$-graded abelian group \cite{km-unknot}. The surface $\Sigma$ induces an involution on $I(Y,L,w)$, and one can show that this involution has degree $2\pmod{4}$ in case (i) and degree $0\pmod{4}$ in case (ii). Thus, in case (i), the rank of $I(Y,L,w)$ is even.

Consider now a triple $(M,T,w)$ consisting of a compact, oriented, connected $3$-manifold with boundary, a {\it tangle} $T\subset M$, and a compact $1$-manifold $w\subset \text{int}M$ satisfying $w\cap T = \partial w$. In this paper, we use the convention that a tangle $T\subset M$ is a properly embedded $1$-manifold in $M$. Assume that $\partial (M,T)$ is the union of the surfaces with marked points $\Sigma_{g,n}$ and $\Sigma_{0,3}$. Consider the moduli space $\mathfrak{X}(M,T,w)$ of flat $SU(2)$ connections on $M\smallsetminus (T\cup w)$ with traceless holonomy around meridians of $T$ and $-1$ holonomy around meridians of $w$. In general this may be a rather singular space, but there always exists a holonomy perturbation data $\pi=(\eta,\Gamma)$ such that the associated perturbed-flat moduli space $\mathfrak{X}_\pi(M,T,w)$ is a smooth manifold, and such that the map 
\begin{equation*}\label{eq:3mfldlagrangianintro}
	\mathfrak{X}_\pi(M,T,w) \to M_{g,n}
\end{equation*}
induced by restriction of connections to the boundary is a Lagrangian immersion; here we use that the moduli space for $\Sigma_{g,n} \sqcup \Sigma_{0,3}$ can be identified with $M_{g,n}$ because $M_{0,3}$ is a point. The existence of such a perturbation $\pi$ follows by adapting the work of Herald \cite{herald}, where no punctures are present; see also \cite{dfl,herald-kirk}. In \S \ref{sec:inequality}, we give a new, less-analytic proof that such perturbations exist. Given two such triples $(M_i,T_i,w_i)$ we glue them to obtain
\begin{equation}\label{eq:tripledecomposedintro}
	(Y,L,w) = (M_1,T_1,w_1)\cup_{\Sigma_{g,n} \sqcup \Sigma_{0,3}} (M_2,T_2,w_2)
\end{equation}
where we reverse the orientation of $M_2$ so that $Y$ inherits an orientation, and $w=w_1\sqcup w_2$.

Denote the fiber product of two Lagrangian immersions $r_i:\Lambda_i\to X$ in a symplectic manifold by $\Lambda_1\times_X \Lambda_2$. Write $\text{Ham}(X)$ for the group of Hamiltonian diffeomorphisms of $X$. For $\phi\in  \text{Ham}(X)$, the immersed Lagrangian $\phi\circ r_i : \Lambda_i\to X$ is denoted by $\phi(\Lambda_i)$. Then define
\begin{equation}\label{eq:minimallagrangianintersectiondefn}
	\mathfrak{n}(\Lambda_1,\Lambda_2) = \min \{ \# \left( \Lambda_1 \times_X \phi(\Lambda_2) \right) \; \mid \;  \phi\in \text{Ham}(X), \; \Lambda_1 \pitchfork \phi(\Lambda_2)\ \}
\end{equation}
If the $L_i$ are embedded, $\mathfrak{n}(\Lambda_1,\Lambda_2) $ is the minimal number of intersection points of $\Lambda_1$ and $\phi(\Lambda_2)$ ranging over all Hamiltonian diffeomorphisms $\phi$ for which such intersections are transverse.

\begin{theorem}\label{thm:lagrangianineq}
	Suppose $\Lambda_i = \mathfrak{X}_{\pi_i}(M_i,T_i,w_i)$ are Lagrangian immersions in $M_{g,n}$ as above, where $\partial (M_i,T_i)=\Sigma_{g,n}\sqcup \Sigma_{0,3}$ with $n$ odd, and the glued up triple $(Y,L,w)$ is an admissible link. Then
	\begin{equation}\label{eq:lagrankineqintro}
		\mathfrak{n}(\Lambda_1,\Lambda_2) \geq \frac{1}{2}\, \textup{rank}_\Z \; I(Y,L,w).
	\end{equation}
	The same inequality holds using the rank of instanton homology computed with respect to any integral domain coefficient ring. Furthermore, if $ I(Y,L,w)=0$ and the perbtation data $\pi_i$ defining $\Lambda_i$ is sufficiently small for $i\in \{1,2\}$, then $\mathfrak{n}(\Lambda_1,\Lambda_2)=0$.
\end{theorem}

\noindent Fixing $i\in \{1,2\}$, the Lagrangians $\mathfrak{X}_{\pi_i}(M_i,T_i,w_i)$ for different choices of perturbations $\pi_i$ (for which the perturbed character varieties are smooth) are not necessarily Hamiltonian isotopic, or even diffeomorphic. In general, they are expected to be related by Lagrangian cobordism.

The inequality of Theorem \ref{thm:lagrangianineq} is a consequence of Theorem \ref{thm:hamiltoniandiffeo} and the definitions of the terms. Indeed, Theorem \ref{thm:hamiltoniandiffeo} implies that $\mathfrak{n}(\Lambda_1,\Lambda_2)$ can be realized by a Hamiltonian diffeomorphism $\phi$ of $M_{g,n}$ which is induced by a holonomy perturbation $\pi_\Sigma$ in the thickened surface. By construction, the instanton homology $I(Y,L,w)$ may be computed from a chain complex which has generators two copies of $\Lambda_1 \times_{M_{g,n}} \phi(\Lambda_2)$; these are the perturbed-flat connections on $(Y,L,w)$ associated to the given holonomy perturbation data $\pi_1,\pi_2$ and $\pi_\Sigma$. The two copies of $\Lambda_1 \times_{M_{g,n}} \phi(\Lambda_2)$ are interchanged by an involution that negates the holonomy of a loop that passes once through $\Sigma_{g,n}$ and once through $\Sigma_{3,0}$. Alternatively, write $I(Y,L,w)_\Sigma$ for the quotient of $I(Y,L,w)$ by the involution induced by the surface $\Sigma_{g,n}$. Since this involution has degree $2$ (mod $4$), as abelian groups we have
\[
	I(Y,L,w) = I(Y,L,w)_\Sigma \oplus I(Y,L,w)_\Sigma
\]
and the right side of \eqref{eq:lagrankineqintro} may be written as $\text{rank}\, I(Y,L,w)_\Sigma$.

The utility of Theorem \ref{thm:lagrangianineq} stems from two points. The first is that many Lagrangians in $M_{g,n}$ arise from the construction $\mathfrak{X}_\pi(M,T,w)$. For example, up to symplectic isotopy all Lagrangian spheres in $M_{0,5}$ can be realized by this construction, an important ingredient for Theorem \ref{thm:lagrangianspheres}. The second point is an explicit understanding of when $I(Y,L,w)$, defined over $\Z$, vanishes:

\begin{theorem}\label{thm:detectionintro}
    Let $(Y,L,w)$ be an admissible link. Then $I(Y,L,w)$ is zero if and only if there exists an embedded $2$-sphere $S^2\subset Y$ such that one of the following is satisfied: 
    \begin{itemize}
    	\item $L$ intersects $S^2$ transversely in one point
    	\item  $L$ is disjoint from $S^2$ and the homological pairing $w\cdot L$ is odd
    \end{itemize}
\end{theorem}

\noindent This theorem is a generalization of \cite[Prop. 2.1]{dsodd}, which treats the case when $L$ is empty. The proof of Theorem \ref{thm:detectionintro} relies on Kronheimer and Mrowka's sutured instanton homology and their non-vanishing result for instanton homology of taut sutured manifolds \cite{km-sutures}. The statement and argument of Theorem \ref{thm:detectionintro} also extend the non-vanishing results due to Kronheimer and Mrowka from \cite[\S 7]{km-tait}. The last statement in Theorem \ref{thm:lagrangianineq} is a consequence of Theorem \ref{thm:detectionintro}. 

For a triple $(Y,L,w)$ which is decomposed as in \eqref{eq:tripledecomposedintro}, it is convenient to recast the detection result of Theorem \ref{thm:detectionintro} in terms of the topology of the given decomposition, and more specifically via the data of curves on the surface $\Sigma_{g,n}$ which bound certain disks in $(M_i,T_i)$. This is achieved using results on Heegaard splittings of $3$-manifolds with tangles, following \cite{hayashi-shimokawa}. 

The proof of Theorem \ref{thm:mappingclassgroupaction} also requires an understanding of when the distance of Heegaard splittings induced by powers of a pseudo-Anosov diffeomorphism of a surface is unbounded. This is established for surfaces without punctures in \cite{hempel,abrams-schleimer}, which was used in the authors' previous work \cite{dsodd}, taking cue from \cite{clarkson}. The case for surfaces with punctures which is necessary for Theorem \ref{thm:mappingclassgroupaction} is adapted from the work \cite{ichihara-saito}.

\subsubsection*{Relation to the Atiyah--Floer conjecture}

Inspired by the original Atiyah--Floer conjecture \cite{atiyah-newinvs} and generalizing the version established by the first author in collaboration with Fukaya and Lipyanskiy \cite{dfl} for the case where the link $L$ is empty, it is natural to conjecture that there is an isomorphism of abelian groups
\begin{equation}\label{eq:atiyahfloerintro}
	 I(Y,L,w)_\Sigma \cong \textup{HF}(\Lambda_1,\Lambda_2)
\end{equation}
for an admissible link decomposed as in \eqref{eq:tripledecomposedintro}, where the Lagrangians $\Lambda_i=\mathfrak{X}_{\pi_i}(M_i,T_i,w_i)$ are assumed to be smoothly embedded in $M_{g,n}$. The right side of \eqref{eq:atiyahfloerintro} denotes Lagrangian Floer homology of $\Lambda_1$ and $\Lambda_2$ defined over $\Z$ (here we use that $M_{g,n}$ is monotone).

Theorem \ref{thm:lagrangianineq} would be a direct corollary of \eqref{eq:atiyahfloerintro}. Furthermore, the authors expect that \eqref{eq:atiyahfloerintro} can be proved, with a good deal of work, by adapting \cite{dfl} (including the analysis of the mixed equation studied in \cite{dfl-mixed}) to the setting of gauge theory with prescribed singularities along the link $L$. However, in this paper, we are able to prove Theorem \ref{thm:lagrangianineq}, which leads to all of our applications, without having to establish the Atiyah--Floer isomorphism \eqref{eq:atiyahfloerintro}, by solely relying on the relationship between Hamiltonian diffeomorphisms of $M_{g,n}$ and holonomy perturbations (Theorem \ref{thm:hamiltoniandiffeo}). Moreover, Theorems \ref{thm:hamiltoniandiffeo} and \ref{thm:lagrangianineq} are established in a level of generality that includes the odd character varieties studied in \cite{dsodd}, and thus the work here provides an alternative proof to the results of that paper which does not rely on the Atiyah--Floer results in \cite{dfl}.

As a final remark, in principle it should be possible to prove Theorems \ref{thm:mappingclassgroupaction} and \ref{thm:lagrangianspheres} using Lagrangian Floer homology, a tool specifically designed for such problems, and not appeal to instanton homology. However, Lagrangian Floer homology, like most Floer homologies, is generally difficult to compute. An advantage of our approach is that it allows us to harness the progress that has been made on the side of instanton homology, as seen especially in the work of Kronheimer and Mrowka \cite{km-sutures}, and which is here represented by Theorem \ref{thm:detectionintro}.

\subsubsection*{Outline}

After some preliminaries on approximating Hamiltonian diffeomorphisms in \S \ref{sec:hamiltonianapprox}, in \S \ref{mod-space-Ham-diffeo} we introduce the moduli spaces of flat $SU(2)$ connections on punctured surfaces and study their Hamiltonian diffeomorphisms and symplectomorphisms. Holonomy perturbations, Lagrangians from $3$-manifolds with tangles, and the proof of Theorem \ref{thm:hamiltoniandiffeo} are treated in \S \ref{sec:holonomyandlagrangians}. In \S \ref{sec:inequality}, singular instanton homology is reviewed and Theorem \ref{thm:lagrangianineq} is proved. The detection result of Theorem \ref{thm:detectionintro} is proved in \S \ref{sec:detection}, and this theorem is then studied in terms of Heegaard splittings and mapping classes in \S \ref{sec:mappingclassgroupthm}, where Theorem \ref{thm:mappingclassgroupaction} is also proved. Applications to $\text{Bl}_5 \mathbb{C}\mathbb{P}^2$, including Theorem \ref{thm:lagrangianspheres} and Corollary \ref{cor:lagrangianspheres} are proved in \S \ref{sec:lagspherecomp}. In \S \ref{sec:intersectionsandsu2simple}, the relationship between intersection numbers of Lagrangians in $M_{g,n}$ and $SU(2)$-simple knots is studied, and Theorem \ref{thm:preciseintersectionresult} is proved. Finally, in \S \ref{sec:genus}, we consider Lagrangians in the odd character variety $M_{g}^{\text{odd}}$ and prove Theorem \ref{thm:lagrangianspheresinquadricintersection}.

\subsubsection*{Acknowledgments}

A MathOverflow question posed by Harry Reed in 2017 asked whether anything was known about Lagrangian sphere complexes for symplectic manifolds of dimension bigger than $2$ \cite{mathoverflow}. In response to that question, Jonny Evans gives some reasoning for why the complex for $\text{Bl}_5 \mathbb{CP}^2$ should be related to the curve complex of $\Sigma_{0,5}$ following ideas of Seidel from \cite{seidel-4d}. The authors thank Jonny Evans for a very informative email correspondence. 

The authors thank Tom Mrowka, who posed the problem of obtaining the results from \cite{dsodd} without using the Atiyah--Floer Conjecture results of \cite{dfl}, by more directly studying the relationship between Hamiltonian isotopies of $M_{g}^{\text{odd}}$ and holonomy perturbations. The authors thank the Simons Center for Geometry and Physics for hosting the workshop \emph{Gauge Theory and Floer Homology in Low Dimensional Topology} in the spring of 2025, where this work was initiated.

\subsubsection*{AI disclosure}

No AI/LLMs were used in the development of the results nor in the writing of this paper. The authors only benefited from AI insofar as its effects on the current mathematical climate, which has prompted them to finish the paper sooner than otherwise. 

%!TEX root = main.tex

\section{Approximating Hamiltonian diffeomorphisms}\label{sec:hamiltonianapprox}

In this section we provide some general results on the approximation of Hamiltonian diffeomorphisms. Let $(M,\omega)$ be a closed symplectic manifold. Given a smooth function $H:[0,1]\times M \to \R$, the associated Hamiltonian time-dependent vector field $X_{H_t}$ is defined by $\iota(X_{H_t}) \omega = dH_t$. The flow of $X_{t}$ is written $\phi_H^t$. A diffeomorphism of $M$ is {\emph{Hamiltonian}} if it is equal to $\phi_H^1$ for some $H$. Write $\text{Ham}(M,\omega)$ for the group of Hamiltonian diffeomorphisms, which is a subgroup of $\text{Symp}(M,\omega)$. Given a smooth function $f:M\to \R$, by viewing $f$ as a function $[0,1]\times M\to \R$ constant in $[0,1]$, we write $X_f$ for the associated vector field and $\phi_f^t$ for the associated flow as above; note $\phi_f^t$ is defined for all $t\in \R$ under our assumption that $M$ is compact.

\begin{prop}\label{prop:cinfinityapprox}
    Suppose $(M,\omega)$ is a smooth closed symplectic manifold. Let $S\subset C^\infty(M;\R)$ be such that the $\R$-span of $S$ is $C^1$-dense in $C^\infty(M;\R)$. Then
    \begin{equation}\label{eq:hamiltoniangenset}
       \mathcal{G}(S) := \left\{ \phi_{f_1}^{t_1}\circ \cdots \circ \phi_{f_m}^{t_m} \; \mid \; t_1,\ldots,t_m\in \R, \; f_1,\ldots, f_m\in S \right\}
    \end{equation}
    is dense in {\emph{$\text{Ham}(M,\omega)$}} with respect to the $C^0$-topology.
\end{prop}

The proof we give follows \cite{bt} for some key steps. 

\begin{proof}
    Any $\phi^1_H\in \text{Ham}(M,\omega)$ is a composition $\phi^1_{f_1}\circ \cdots \circ \phi^1_{f_m}$ where $f_i\in C^\infty(M;\R)$. This follows from work of Banyaga \cite{banyaga}. (For our purposes, one could get by with much less, as approximating $\phi^t_H$ by such compositions is elementary.) Next, since $\phi_f^{at} = \phi_{af}^t$ for $a\in \R$, there is no loss of generality in assuming, from the start, that $S$ is closed under scaling by $\R$. 
    
    If $f\in C^{\infty}(M;\R)$, then by assumption there is some $g$ in the span of $S$ arbitrarily close to $f$ in the $C^1$-topology. In particular, the Hamiltonian vector fields $X_f$ and $X_g$ are $C^0$-close, and consequently so too are the associated flows $\phi_f^t$ and $\phi_g^t$ for any fixed time $t$. 

    Thus to complete the proof it remains to show the following: if $f$ is a finite sum of elements from $S$, then the Hamiltonian flow of $f$ can be $C^0$-approximated by a composition of flows each coming from an element of $S$. It suffices to approximate the time $1$ flow $\phi_f^1$ where $f=f_1+f_2$ and $f_1,f_2\in S$, the more general case following by induction.

    For $f$ fixed as above, let $g:M\to M$ be a diffeomorphism satisfying
    \[
       d ( \phi^{1/N}_f, \,  g )_{C^0} :=  \sup_{x\in M}d\left(\phi^{1/N}_f(x), \, g(x)\right) \leq h(N)/N
    \]
    where $h$ is some function. Here the distance $d$ on $M$ is defined using a chosen Riemannian metric on $M$. Then taking the $N$-fold composition $g^N$ gives an approximation to $\phi_f^1$. Specifically, 
    \begin{align*}
         d\left(\phi^{1}_f, \, g^N \right)_{C^0} & \; \leq \;   \sum_{i=0}^{N-1} d\left((\phi^{1/N}_f)^{i+1} \circ g^{N-i-1} , \; (\phi^{1/N}_f)^{i} \circ g^{N-i} \right)_{C^0}\\
         & \; = \;  \sum_{i=0}^{N-1} d\left((\phi^{i/N}_f) \circ \phi_f^{1/N} , \; (\phi^{i/N}_f) \circ g  \right)_{C^0}\\
        & \; \leq \;  \sum_{i=0}^{N-1}  e^{ i L /N}   d ( \phi_f^{1/N}, \; g )_{C^0} \\
        & \; \leq \;  \sum_{i=0}^{N-1}  \frac{e^{ i L /N}}{N}  h(N) \; \leq \;  e^{L}h(N) 
    \end{align*}
    From the second to the third line, the inequality $d(\phi_f^t(x),\phi^t_f(y))\leq e^{tL}d(x,y)$ was applied with $t=i/N$, where $L$ is a constant depending only on $f$; this is an extension of Gronwall's inequality to the Riemannian setting \cite[Thm. 1]{gronwall}.

    Let us apply the above to the problem of approximating $\phi_{f}^1$ where $f=f_1+f_2$. First note
    \[
        d\left(\phi_{f}^{1/N} , \, \phi^{1/N}_{f_1}\circ\phi^{1/N}_{f_2} \right)_{C^0} \leq C/N^2  
    \]
    where $C$ is independent of $N$. Indeed, in local coordinates both $\phi_{f}^{t}$ and $\phi^{t}_{f_1}\circ\phi^{t}_{f_2}$ are equal to
    \[
        \text{id}+t(X_{f_1}+X_{f_2}) + O(t^2).
    \]
    Now letting $g=\phi^{1/N}_{f_1}\circ\phi^{1/N}_{f_2}$ and $h(N)=C/N$ in the previous paragraph we obtain
    \begin{equation}\label{eq:inequalitycompflows}
         d\left( \phi^{1}_f, \, (\phi^{1/N}_{f_1}\circ\phi^{1/N}_{f_2})^N \right)_{C^0} \leq C e^L/N.
    \end{equation}
    Thus by taking $N$ large enough we obtain a composition of flows arising from $f_1$ and $f_2$ which $C^0$-approximates $\phi_f^1$ to any desired accuracy. 
\end{proof}

We next consider hypotheses which improve the regularity in the conclusion of Proposition \ref{prop:cinfinityapprox}. First we recall some terminology. A symplectic manifold $(M,\omega)$ is {\emph{real analytic}} if $M$ is a real analytic manifold and $\omega$ is a real analytic $2$-form. The notion of complex analytic symplectic manifold is defined similarly; in this case, $\omega$ is a holomorphic $2$-form, i.e. of type $(2,0)$. A {\emph{real structure}} on a complex analytic symplectic manifold $(M_\C,\omega_\C)$ is an anti-holomorphic symplectic involution $\tau$. Let $M_\R$ be the fixed point set of $\tau$, called the {\emph{real locus}} of $\tau$. Then $(M_\R,\omega_\R)$ is a real analytic symplectic manifold, where $\omega_\R=\omega_\C|_{M_\R}$.

Suppose $(M,\omega)$ is a real analytic submanifold of the real locus of $(M_\C, \omega_\C,\tau)$. Denote by $C^\text{an}(M;\R)$ the set of real analytic functions $f:M\to \R$, topologized in the following way: a sequence $\{f_i\}_{i=0}^{\infty}$ converges to $f$ if and only if there exists an open neighborhood $V\subset M_\C$ of $M$ such that $f$ and $f_i$ for $i$ large enough each extend to complex analytic functions $V\to \C$, and $\sup_{x\in V} | f(x)-f_i(x) | \to 0$ as $i\to \infty$. Note $C^\text{an}$-convergence implies $C^\infty$-convergence.

\begin{lemma}\label{lemma:improveregularity}
In Proposition \ref{prop:cinfinityapprox}, assume $(M,\omega)$ is a real analytic symplectic submanifold of the real locus of a complex analytic symplectic manifold $(M_\C,\omega_\C)$ with real structure, $S\subset C^{\textup{an}}(M;\R)$, and $\textup{span}_\R(S)$ is dense in $C^{\textup{an}}(M;\R)$. Then $\mathcal{G}(S)$ is $C^\infty$-dense in $\textup{Ham}(M,\omega)$.
\end{lemma}

\noindent Note that in this statement, $M$ is compact but $M_\C$ may not be compact.

\begin{proof}

The proof is similar to that of Proposition \ref{prop:cinfinityapprox}. As in the first paragraph there, we may restrict our attention to time-independent Hamiltonians $f:M\to \R$, and assume that $S$ is closed under scaling by $\R$. If $f\in C^\infty(M;\R)$, then since $C^\text{an}(M;\R)$ is $C^\infty$-dense in $C^\infty(M;\R)$, and by the assumption on $S$, there is some $g$ in the span of $S$ arbitrarily close to $f$ in the $C^\infty$-topology. Then $X_f$ and $X_g$ are $C^\infty$-close and so too are the associated flows.

Thus, just as in the proof of Proposition \ref{prop:cinfinityapprox}, it remains to show that $\phi_{f}^1$, for $f=f_1+f_2$ with $f_1,f_2\in S$, can be $C^\infty$-approximated by a composition of flows induced by elements of $S$.

Let $V_0\subset M_\C$ be an open neighborhood of $M$ on which $f$ has a holomorphic extension, also called $f$. The vector field $X_f$ on $M$ then extends to a holomorphic vector field on $V_0$ by the same name, defined by $\iota(X_f)\omega_\C = df$. There exists a neighborhood $V_1\subset V_0$ of $M$ such that the flow $\phi^t_f$ of $X_f$ is a holomorphic map defined on $V_1$ and satisfies $\phi^t_f(V_1)\subset V_0$ for all $t\in [0,1]$. Let $V_2\subset M_\C$ be a neighborhood of $M$ with compact closure in $V_1$. Then inequality \eqref{eq:inequalitycompflows} holds on $V_2$:
\[
\sup_{x\in V_2} d\left(\phi^{1}_f(x), \, (\phi^{1/N}_{f_1}\circ\phi^{1/N}_{f_2})^N(x) \right) \leq C e^L/N
\]
The argument is similar to the one in the proof of Proposition \ref{prop:cinfinityapprox}, keeping track of the neighborhoods on which the flows are defined; for details see \cite{bt}, specifically Lemma 3.1 of that reference. The $C^0$-convergence of the sequence of holomorphic maps $(\phi^{1/N}_{f_1}\circ\phi^{1/N}_{f_2})^N$ to $\phi^1_f$ on $V_2$ implies $C^\infty$-convergence on $V_2$ and hence $C^\infty$-convergence of their restrictions to $M$.
\end{proof}

%!TEX root = main.tex

\section{Moduli spaces, Hamiltonian diffeomorphisms, and mapping class groups}\label{mod-space-Ham-diffeo}

Following the introduction, we write $\Sigma_{g,n}$ for the pair of a Riemann surface $\Sigma=\Sigma_g$ of genus $g$ and a collection $\mathcal{P}=\{p_1,\ldots,p_n\}$ of $n$ distinct points on $\Sigma_g$. For $t\in [0,1]$, let $\Gamma(t)\subset SU(2)$ denote the conjugacy class of elements containing $\text{diag}(e^{i\pi t},e^{-i\pi t})$. Note $\Gamma(0)$ and $\Gamma(1)$ are points while $\Gamma(t)$ is a 2-sphere for each $t\in (0,1)$. For a given $\mathbf{t}=(t_1,\ldots,t_n)\in [0,1]^n$, consider the map
\[
    \mathfrak{m}:SU(2)^{2g}\times \prod_{j=1}^n \Gamma(t_j) \to SU(2),
\]
\[
    \mathfrak{m}(A_1,B_1,\ldots,  A_g,B_g,C_1,\ldots,C_n) = [A_1,B_1]\cdots [A_g,B_g]C_1\cdots C_n.
\]
The group $SO(3)=SU(2)/\pm 1$ acts on the domain by simultaneous conjugation. Let
\begin{equation}\label{rep-theoretic-desc}
    M_{g,n}(\mathbf{t}) = \mathfrak{m}^{-1}(1)/SO(3).
\end{equation}
We refer to $\mathbf{t}$ as the holonomy parameter of $M_{g,n}(\mathbf{t})$. We say $\mathbf{t}$ satisfies the {\emph{non-integral assumption}} if for all subsets $J\subset \{1,\ldots,n\}$, the following quantity is not an integer:
\[
    \frac{1}{2}  \sum_{j\in J} t_j - \frac{1}{2} \sum_{j\not\in J} t_j 
\]
The map $\mathfrak{m}$ has $1$ as a regular value if and only if $\mathbf{t}$ satisfies the non-integral assumption. In this case, the conjugation action is free, and $M_{g,n}(\mathbf{t})$ is a closed smooth manifold of dimension $6g+2n_{\mathbf{t}}-6$ where $n_{\mathbf{t}}$ is the number of $j$ such that $t_j\in (0,1)$. Throughout this paper, we assume that the holonomy parameter $\mathbf{t}$ satisfies the non-integral assumption.

The moduli space $M_{g,n}(\mathbf{t})$ can be regarded as a character variety of $SU(2)$ representations assigned to $\Sigma_{g,n}$ and $\mathbf{t}$. Indeed, $M_{g,n}(\mathbf{t})$ is the space of conjugacy classes of homomorphisms $ \pi_1(\Sigma\smallsetminus \mathcal{P})\to SU(2)$ where the conjugacy class of a small loop enclosing $p_j\in \mathcal P$ is mapped to $\Gamma(t_j)$. For a standard presentation of the fundamental group
\begin{equation}\label{eq:pi1presentationforpuncturedsurface}
    \pi_1(\Sigma\smallsetminus \mathcal{P}) = \langle a_1,b_1,\ldots,a_g,b_g,c_1,\ldots,c_n \;\mid \;[a_1,b_1]\cdots [a_g,b_g]c_1\cdots c_n = 1 \rangle,
\end{equation}
and given an element of $\mathfrak{m}^{-1}(1)$, there is an $SU(2)$-representation of $\pi_1(\Sigma\smallsetminus \mathcal{P})$ that maps $a_i$, $b_i$, $c_j$ respectively to $A_i$, $B_i$, $C_j$. In this paper we use the interpretation of $M_{g,n}(\mathbf{t})$ in terms of a moduli space of flat connections, as we discuss in \S \ref{subsec:symplecticstructure}. A symplectic form $\omega$ for $M_{g,n}(\mathbf{t})$ was constructed in \cite{ghjw}, generalizing the $n=0$ case from \cite{goldman-symplectic}. In \S \ref{subsec:symplecticstructure}, we review the symplectic structure on $M_{g,n}(\mathbf{t})$ from the more geometric point of view of moduli of flat connections. Hamiltonian diffeomorphisms and symplectomorphisms of $M_{g,n}(\mathbf{t})$ are discussed in \S \ref{subsec:hamiltoniandiffmgn} and \S \ref{sec:symplectofrommcg}.

\subsection{Symplectic structure on $M_{g,n}(\mathbf{t})$}\label{subsec:symplecticstructure}

The most relevant interpretation of $M_{g,n}(\mathbf{t})$ for our purposes is as a space of flat $SU(2)$ connections on $\Sigma$ with prescribed singularities. To set up this interpretation of $M_{g,n}(\mathbf{t})$ more carefully, let $\Sigma^\circ$ denote $\Sigma \smallsetminus \mathcal{P}$ and identify a small neighborhood of $\mathcal{P}$ with the union of disks $\mathcal{P}\times D^2$. Then $\Sigma^\circ$ can be compactified to a surface $\overline \Sigma$ of genus $g$ with $n$ boundary components in the obvious way. Define $\mathcal{A}_{g,n}(\mathbf{t})$ as the space of connections on the trivial $SU(2)$-bundle over $\Sigma^\circ$ whose restriction to a neighborhood of the origin in $\{p_j\}\times D^2$ takes the form
\begin{equation}\label{connection-local-model}
\left[  
\begin{array}{cc}
	\frac{it_j}{2}&0\\
0&-\frac{it_j}{2}
\end{array}
\right]d\theta
\end{equation}
where $\theta$ denotes the angular coordinate on $D^2$.

Let $\mathcal G_{g,n}(\mathbf{t})$ be the space of gauge transformations of the trivial $SU(2)$-bundle over $\Sigma^\circ$ whose restriction to a neighborhood of the origin in $\{p_j\}\times D^2$ stabilizes the connection in \eqref{connection-local-model}. That is to say, such a restriction is a constant map into $SU(2)$ if $t_j=0$, is of the form $\text{diag}(e^{ 2i\pi r_j},e^{-2i\pi r_j})$ for some constant $r_j$ if $t_j\in (0,1)$, and is of the form $\text{diag}(e^{ i\theta/2},e^{-i\theta/2}) \cdot u_j\cdot  \text{diag}(e^{ -i\theta/2},e^{i\theta/2})$ for some $u_j\in SU(2)$ if $t_j=1$. Then $\mathcal G_{g,n}(\mathbf{t})$ acts on $\mathcal{A}_{g,n}(\mathbf{t})$ by taking pullback and $M_{g,n}(\mathbf{t})$ is the quotient of the space of flat connections in $\mathcal{A}_{g,n}(\mathbf{t})$ by the action of $\mathcal G_{g,n}(\mathbf{t})$. 

The linearization of the action of the gauge group $\mathcal G_{g,n}(\mathbf{t})$ and the flat equation at a flat connection $A\in \mathcal{A}_{g,n}(\mathbf{t})$ are controlled by the middle row of the diagram in Figure \ref{diagram-complex-surface}. Here $\Omega^i_c(\Sigma^\circ,\frak{g})$ is the space of compactly supported $i$-forms on $\Sigma^\circ$ with values in the Lie algebra $\frak{g}=\frak{su}(2)$ of $SU(2)$, and $\widetilde \Omega^0(\Sigma^\circ,\frak{g})$ is the space of $0$-forms on $\Sigma^\circ$ such that in a small neighborhood of $p_j$ for each $j$, it is a constant map into $\frak{g}$ if $t_j=0$, has the form $\text{diag}(ir_j,-ir_j)$ for some $r_j$ if $t_j\in (0,1)$, and is $\text{diag}(e^{ i\theta/2},e^{-i\theta/2}) \cdot \zeta_j\cdot  \text{diag}(e^{ -i\theta/2},e^{i\theta/2})$ for some $\zeta_j\in \frak{g}$ if $t_j=1$. We write $ H^i_A$ for the cohomology groups of the cochain complex that appears in the middle row of Figure \ref{diagram-complex-surface}. These cohomology groups are trivial in degrees $0$ and $2$ by the non-integral assumption on $\mathbf{t}$, and $H^1_A$ may be identified with the tangent space of the class of $A$ in $M_{g,n}(\mathbf{t})$.

        \begin{figure}[h]
            \centering
                \begin{tikzcd}[column sep=large, row sep=large]
                & 0 \arrow[d] & 0 \arrow[d] & 0 \arrow[d] &\\ 
                0\arrow[r] &\Omega_c^0(\Sigma^\circ,\frak{g})\arrow[d]\arrow[r, "d_{A}"] & 
                  \Omega^1_c(\Sigma^\circ,\frak{g}) \arrow[d] \arrow[r, "d_{A}"] &   \Omega^2_c(\Sigma^\circ,\frak{g})\arrow[r] \arrow[d, equal]&0\\
                0\arrow[r] &\widetilde \Omega^0(\Sigma^\circ,\frak{g})\arrow[r, "d_{A}"] \arrow[d]& 
                  \Omega^1_c(\Sigma^\circ,\frak{g}) \arrow[r, "d_{A}"] \arrow[d]&   \Omega^2_c(\Sigma^\circ,\frak{g})\arrow[r]\arrow[d]&0\\
               0\arrow[r] &\displaystyle \bigoplus_{1\le i\le n} \mathfrak{h}_j  \arrow[r]  \arrow[d] &0 \arrow[r] \arrow[d] & 0 \arrow[d]\arrow[r] &0\\
                & 0  & 0  & 0 &
            \end{tikzcd}
            \caption{}
            \label{diagram-complex-surface}
        \end{figure}

The other rows of the diagram in Figure \ref{diagram-complex-surface} also define cochain complexes. Here $\frak h_j$ denotes the subspace of $\mathfrak{g}$ consisting of elements that commute with $\text{diag}(e^{i\pi t_j},e^{-i\pi t_j})$. 
The degree $i$ cohomology group of the complex in the first row, denoted by $H^i_{c}(\Sigma^\circ;{\rm ad}_A)$, is isomorphic to $H^i(\overline \Sigma,\partial \overline \Sigma;{\rm ad}_A)$, where ${\rm ad}_A$ denotes the twisted coefficients determined by the flat connection $A$ associated to the adjoint action of $SU(2)$ on its Lie algebra. In particular, it is trivial if $i=0$ or $i=2$, and has dimension $6g+3n-6$ if $i=1$. The vertical maps in this diagram define an exact sequence of cochain complexes, and the associated long exact sequence of cohomology groups determines the exact sequence
\[
        \begin{tikzcd}
         0 \arrow[r] & \displaystyle\bigoplus_{1\leq j\leq n} \frak h_j  \arrow[r] &  H^1_{c}(\Sigma^\circ;{\rm ad}_A)   \arrow[r] &  H^1_{A}  \arrow[r] & 0
    \end{tikzcd}
\]
This shows that $H^1_{A}$, the tangent space of $M_{g,t}(\mathbf{t})$ at $[A]$, has dimension $6g+2n_{\mathbf{t}}-6$.
        
From this viewpoint of flat connections, following Atiyah--Bott \cite[\S 9]{atiyah-bott} (see also \cite[\S 2.2]{jeffrey-weitsman}), the symplectic form $\omega$ on $M_{g,n}(\mathbf{t})$ at the class of the connection $A$ is induced by
\begin{equation}\label{eq:sympformonmoduli}
    \omega(\alpha,\alpha') := -\frac{1}{4\pi^2}\int_{\Sigma}\text{tr}(\alpha \wedge \alpha'),
\end{equation}
where $\alpha, \alpha'\in \Omega^1_c(\Sigma^\circ,\mathfrak g)$. That $\omega$ depends only on the classes of $\alpha$ and $\alpha'$ in $H_A^1$ follows from Stokes' theorem. The nondegeneracy of $\omega$ is explained as follows. A variation of \eqref{diagram-complex-surface} is given by Figure \ref{diagram-complex-surface-bdry}, where the middle row is the de Rham complex of $\overline \Sigma$ with twisted coefficients ${\rm ad}_A$, and the third row is the quotient of the complex in the middle row by the one in the first row. In particular, the $i^{\rm th}$ cohomology groups of these complexes are respectively isomorphic to $H^i(\overline \Sigma;{\rm ad}_A)$ and $H^i(\partial \overline \Sigma;{\rm ad}_{A^\partial})$ where $A^\partial$ denotes the restriction of $A$ to the boundary. 
        \begin{figure}[h]
            \centering
                \begin{tikzcd}[column sep=large, row sep=large]
                & 0 \arrow[d] & 0 \arrow[d] & 0 \arrow[d] &\\ 
                0\arrow[r] &\Omega_c^0(\Sigma^\circ,\frak{g})\arrow[d]\arrow[r, "d_{A}"] & 
                  \Omega^1_c(\Sigma^\circ,\frak{g}) \arrow[d] \arrow[r, "d_{A}"] &   \Omega^2_c(\Sigma^\circ,\frak{g})\arrow[r] \arrow[d]&0\\
                0\arrow[r] &\Omega^0(\overline \Sigma,\frak{g})\arrow[r, "d_{A}"] \arrow[d]& 
                  \Omega^1(\overline \Sigma,\frak{g}) \arrow[r, "d_{A}"] \arrow[d]&   \Omega^2(\overline \Sigma,\frak{g})\arrow[r]\arrow[d]&0\\
               0\arrow[r] & \Omega^0_\partial(\overline \Sigma,\frak{g}) \arrow[r, "d_{A}"]  \arrow[d] &\Omega^1_\partial(\overline \Sigma,\frak{g}) \arrow[r, "d_{A}"] \arrow[d] & \Omega^2_\partial(\overline \Sigma,\frak{g}) \arrow[d]\arrow[r] &0\\
                & 0  & 0  & 0 &
            \end{tikzcd}
            \caption{}
            \label{diagram-complex-surface-bdry}
        \end{figure}

The inclusion of the middle row of Figure \ref{diagram-complex-surface} into the middle of Figure \ref{diagram-complex-surface-bdry} determines a map between these two diagrams, and the induced map between the corresponding long exact sequences of cohomology groups gives the following, with exact rows:
\[
        \begin{tikzcd}
       \displaystyle \bigoplus_{1\le j \le n} \mathfrak{h}_j \arrow[r]  \arrow[d, equal] &H^1(\overline \Sigma,\partial \overline \Sigma;\mathrm{ad}_A) \arrow[r, "f"] \arrow[d, equal] & H^1_{A}   \arrow[r] \arrow[d, "h"]& 0 \arrow[d] \\
        \displaystyle\bigoplus_{1\le j\le n} \mathfrak{h}_j \arrow[r]  & H^1(\overline \Sigma,\partial \overline \Sigma;\mathrm{ad}_A) \arrow[r]  & H^1(\overline \Sigma;\mathrm{ad}_A) \arrow[r, "g"] & H^1(\partial \overline \Sigma;\mathrm{ad}_{A^\partial})
    \end{tikzcd}
\]        
Denoting by $\langle\cdot,\cdot\rangle$ the standard pairing between $H^1(\overline \Sigma,\partial \overline \Sigma;\mathrm{ad}_A)$ and $ H^1(\overline \Sigma;\mathrm{ad}_A)$, for any $\alpha\in H^1_A$ and $\beta\in H^1(\overline \Sigma,\partial \overline \Sigma;\mathrm{ad}_A)$, we have
\begin{equation}\label{rel-pairing}
  \langle h(\alpha), \beta\rangle=\omega(\alpha,f(\beta)).
\end{equation}
An examination of the above diagram shows that $H^1_A$ is isomorphic to $\ker(g)$. Finally, the non-degeneracy of $\omega$ follows from the non-degeneracy of $\langle\cdot,\cdot\rangle$ and \eqref{rel-pairing}.

\subsubsection{The monotone case} 

The main results of this paper apply to the case that $M_{g,n}(\mathbf{t})$ is a monotone symplectic manifold. Recall that a symplectic manifold $(M,\omega)$ is {\emph{monotone}} if there is a positive constant $\kappa$ such that 
\[
  [\omega]=\kappa c_1(TM)
\]
where $[\omega]$ denotes the cohomology class of the symplectic form, and the first Chern class of $TM$ is computed with respect to an almost complex structure on $TM$ that is compatible with $\omega$. The symplectic structure on $M_{g,n}(\mathbf{t})$ is monotone in the case that each $t_i$ belongs to the set $\{0,1/2,1\}$, see \cite{MW:mon-moduli-flat}. Our normalization is chosen so that $\kappa=1$. In particular, this moduli space is monotone in the case that $n$ is odd and all components of $\mathbf{t}$ are $1/2$. This is the moduli space denoted by
\[
     M_{g,n} := M_{g,n}(\tfrac{1}{2},\ldots,\tfrac{1}{2})  
\]
 and we continue to use this notation for this special case throughout the paper. Another monotone moduli space is the case where $n$ is odd and all components of $\mathbf{t}$ are $1$:
\[
     M^\textup{odd}_g := M_{g,n}(1,\ldots,1).
 \]
This is the {\emph{odd character variety}}, and was the central focus of the authors' previous work \cite{dsodd}. It is easy to see that removing any number of even punctures from $\Sigma_{g,n}$ does not change this moduli space, and for this reason $n$ is often omitted in the notation. 

The above moduli spaces $M_{g,n}$ and $M_{g}^{\text{odd}}$ are essentially the only moduli spaces of the form $M_{g,n}(\mathbf{t})$ that give a smooth monotone symplectic manifold. To see this, let $\mathbf{t}=(t_1,\dots,t_n)$ satisfy the non-integral assumption and $t_i\in \{0,1/2,1\}$ for any $i$. We can remove punctures with $t_i=0$ and, as above, remove any set of an even number of punctures with $t_i=1$ without changing the diffeomorphism type of $M_{g,n}(\mathbf{t})$, because the conjugacy classes $\Gamma(0)$ and $\Gamma(1)$ consist of single central elements. This still leaves another possibility where $n$ is even, $\mathbf{t}$ has $n-1$ entires equal $1/2$, and the remaining entry equal $1$. We can see this moduli space is diffeomorphic to $M_{g,n-1}$ using a diffeomorphism that maps $[A_1,B_1,\ldots,  A_g,B_g,C_1,\ldots,C_{n-1}]$ in $M_{g,n-1}$, in its representation theoretic description, to $[A_1,B_1,\ldots,  A_g,B_g,C_1,\ldots,C_{n-2},-C_{n-1}, -1]$. It is not difficult to see that these diffeomorphisms are symplectomorphisms. 

We end this discussion by pointing out that monotonicity of $M_{g,n}$ is not used directly in the proof of Theorem \ref{thm:mappingclassgroupaction}, although the proof uses versions of instanton Floer homology that satisfy a counterpart of monotonicity.

\subsection{Hamiltonian diffeomorphisms of $M_{g,n}(\mathbf{t})$} \label{subsec:hamiltoniandiffmgn}

The goal of this section is to introduce some natural maps on  $M_{g,n}(\mathbf{t})$ and study Hamiltonian flows associated to them. Let $\gamma$ be an immersed loop in $\Sigma$, disjoint from the set of points $\mathcal{P}$. Define
\[
    \text{tr}_{\gamma}:M_{g,n}(\mathbf{t})\to [-2,2]
\]
whose value on $[A]\in M_{g,n}(\mathbf{t})$ is the trace of the holonomy of $A$ around $\gamma$. Note that $\text{tr}_{\gamma}$ only depends on the free homotopy class of $\gamma$ in $\Sigma\smallsetminus \mathcal{P}$, and does not depend on the direction: $\text{tr}_\gamma = \text{tr}_{\gamma^{-1}}$. For $\Gamma=\{\gamma_1,\ldots,\gamma_k\}$ a collection of immersed loops, define
\[
    T_\Gamma:M_{g,n}(\mathbf{t})\to \R, \qquad 
    T_\Gamma := \text{tr}_{\gamma_1}\text{tr}_{\gamma_2}\cdots \text{tr}_{\gamma_k}. 
\]
By convention, $T_{\emptyset}=1$.
We refer to $T_{\Gamma}$ as a {\emph{trace monomial}} on the moduli space. Call $\Gamma$ a {\emph{multicurve}} if the $\gamma_i$ are homotopic in $\Sigma\smallsetminus \mathcal{P}$ to essential embedded loops $\gamma'_i$ such that no $\gamma_i'$ is homotopic to a circle enclosing a point in $\mathcal{P}$, and each pair of distinct $\gamma'_i$, $\gamma_j'$ is disjoint. We sometimes will conflate $\Gamma$ with the union of $\gamma_1,\ldots, \gamma_k$.

\begin{prop}\label{prop:firsttraceapprox}
    Let $S$ be the set of trace monomials $T_\Gamma$ where $\Gamma$ runs over all multicurves of $\Sigma_{g,n}$. Then the $\R$-span of $S$ is $C^{\infty}$-dense in $C^\infty(M_{g,n}(\mathbf{t});\R)$.
\end{prop}

\begin{proof}
Consider $SU(2)^{2g+n}$ with the $SO(3)$-action induced by simultaneous conjugation, and the subset $X\subset SU(2)^{2g+n}$ of points with trivial $SO(3)$-stabilizer. There is a smooth embedding 
\[
    F:M_{g,n}(\mathbf{t})\to X/SO(3)
\]
obtained by viewing the quotient $X/SO(3)$ as the irreducible $SU(2)$-character variety for a subsurface of $\Sigma$ which retracts onto the 1-skeleton formed by the generating loops $a_i,b_i,c_j$ for $\pi_1(\Sigma\smallsetminus \mathcal{P})$. Furthermore, there is a smooth embedding
\[
    G:X/SO(3) \to \R^N
\]
defined as follows. Set $x_i=a_i$ and $x_{i+g}=b_i$ for $1\leq i \leq g$, and $ x_{j+2g}=c_j$ for $1\leq j \leq n$. Then $N=2g+n + { 2g+n \choose 2} + { 2g+n \choose 3}$ and the components of $G$ are given by
\begin{equation}\label{eq:tracefunctionsforembedding}
    \text{tr}_{x_i}, \qquad  \text{tr}_{x_ix_j}, \qquad  \text{tr}_{x_ix_jx_k}
\end{equation}
where $i,j,k$ run from $1$ to $2g+n$ and $i<j<k$. In fact, at least $n$ of these functions may be omitted, as in our setup the functions $\text{tr}_{c_j}$ are constant. The proof that $\chi$ is an embedding is discussed in \cite[Lemmas 5.13 \& 5.14]{donaldson-book}.\footnote{In \cite[p.140]{donaldson-book}, there is a minor error. The trace functions for $x_i$ and $x_ix_j^{-1}$ (the latter of which may be used instead of $x_i x_j$) are mentioned as being nearly sufficient for the purposes of obtaining the desired embedding $G$. It is then remarked that one needs to include trace functions on some other words, such as commutators of the elements $x_i$. However, commutators do not suffice, and one should instead use the words $x_ix_jx_k$.}
See \cite{procesi} for a more complete treatment in the algebraic context. As the polynomial functions are $C^{\infty}$-dense in $C^\infty(B;\R)$ for a closed ball $B\subset \R^N$ containing the image of $G$, we obtain the following, via $F$ and $G$: linear combinations of trace monomials
\[
    \sum_{i=1}^m r_i T_{\Gamma_i} \qquad (r_i\in\R)
\]
form a $C^\infty$-dense subset of $C^{\infty}(M_{g,n}(\mathbf{t});\R)$. To complete the proof it suffices to show that any trace monomial $T_{\Gamma}$, where $\Gamma=\{\gamma_1,\ldots,\gamma_k\}$ is a set of immersed curves, may be written as a linear combination of trace monomials $T_{\Gamma_i}$ where the $\Gamma_i$ are multicurves. This uses the well-known $\text{sl}_2$ skein relation for trace functions, as is explained below.

For $\Gamma=\{\gamma_1,\ldots,\gamma_k\}$ as above, denote by $d(\Gamma)$ the minimal number of double points among all generic immersions $\sqcup^k S^1\to \Sigma\smallsetminus \mathcal{P}$ where the $i^\text{th}$ component is homotopic to the inclusion of $\gamma_i$. Suppose $d(\Gamma)>0$. Upon reordering, either $\gamma_1$ is an immersion with a double point, or $\gamma_1$ and $\gamma_2$ have positive minimal intersection number. Suppose we are in the first case. 

\begin{figure}[t]
  \centering
\begin{tikzpicture}
  \node[anchor=south west, inner sep=0] (image) at (0,0) {\includegraphics[scale=0.8]{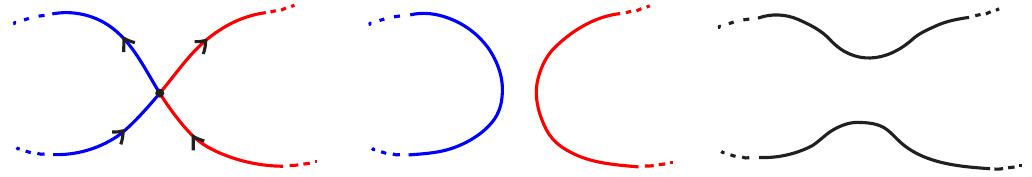}};
  \begin{scope}[x={(image.south east)}, y={(image.north west)}]
    \node at (0.07,0.8) {\large $a$};
    \node at (0.07,0.25) {\large $a$};
    \node at (0.24,0.78) {\large $b$};
    \node at (0.24,0.21) {\large $b$};
    \node at (0.45,0.5) {\large $\gamma_1'$};
    \node at (0.55,0.5) {\large $\gamma_1''$};
    \node at (0.133,0.45) {\large $p$};    
    \node at (0.15,0.00) {\large $\gamma_1$};
    \node at (0.5,0.00) {\large $\gamma_1^o$};
    \node at (0.835,0.00) {\large $\gamma_1^u$};
  \end{scope}
\end{tikzpicture}
\caption{ }
\label{fig:skein}
\end{figure}

View $\gamma_1$ as an element in $\pi_1(\Sigma\smallsetminus \mathcal{P},p)$ where the basepoint $p\in \Sigma\smallsetminus \mathcal{P}$ is one of the double points of $\gamma_1$. Then we can write $\gamma_1 = ab$ where $a,b\in \pi_1(\Sigma\smallsetminus \mathcal{P},p)$ and $a,b$ are represented by loops which, away from $p$, are immersed and transverse to one another, with minimal intersection. For a given representation $[\rho]\in M_{g,n}(\mathbf{t})$ write $A=\rho(a)$ and $B=\rho(b)$. The $SU(2)$ trace identity
\[
    \text{tr}(A)\text{tr}(B) = \text{tr}(AB) + \text{tr}(AB^{-1})
\]
shows that $\text{tr}_{\gamma_1'}(\rho)\text{tr}_{\gamma_1''}(\rho) = \text{tr}_{\gamma_1}(\rho)+\text{tr}_{\gamma_1^u}(\rho)$, where $\gamma_1^o=\gamma_1'\cup\gamma_1''$ are the two immersed loops which result from the oriented resolution of $\gamma_1$ at $p$ (with respect to some orientation), and $\gamma_1^u$ is the immersed loop formed from the unoriented resolution, see Figure \ref{fig:skein}. In particular, we have
\[
    T_{\Gamma} = T_{\Gamma^o} - T_{\Gamma^u} 
\]
where $\Gamma^o=\{\gamma_1',\gamma''_1,\gamma_2,\ldots,\gamma_k\}$ and $\Gamma^u=\{\gamma_1^u,\gamma_2,\ldots,\gamma_k\}$. Furthermore, $d(\Gamma^o)$ and $d(\Gamma^u)$ are strictly less than $d(\Gamma)$. A similar argument holds in the case for a pair of curves $\gamma_1$ and $\gamma_2$ with positive minimal intersection number. Thus by induction every $T_{\Gamma}$ is a linear combination of $T_{\Gamma_i}$ where $d(\Gamma_i)=0$. Note that if $d(\Gamma_i)=0$, either $\Gamma_i$ is a multicurve, or it is a multicurve together with some loops that are either nullhomotopic in $\Sigma\smallsetminus \mathcal{P}$ or homotopic to a loop around a point in $\mathcal{P}$. However, the trace function for a nullhomotopic loop is of course constant, as is that for the latter type of loop, by the definition of $M_{g,n}(\mathbf{t})$. Thus these loops may be discarded.
\end{proof}

\vspace{0.1cm}

Propositions \ref{prop:cinfinityapprox} and \ref{prop:firsttraceapprox} imply that $\mathcal{G}(S)$ is $C^0$-dense in $\text{Ham}(M_{g,n}(\mathbf{t}),\omega)$, where $S$ is the set of trace monomials of multicurves on $\Sigma_{g,n}$ and the set $\mathcal{G}(S)$ is defined in \eqref{eq:hamiltoniangenset}. Next, the regularity in this conclusion is improved through the use of Lemma \ref{lemma:improveregularity}.

\begin{prop}\label{prop:tracemondense}
    Let $S$ be the set of trace monomials of multicurves on $\Sigma_{g,n}$. Then
    \begin{equation}\label{eq:hamiltoniangensetformgn}
       \mathcal{G}(S) = \left\{ \phi_{f_1}^{t_1}\circ \cdots \circ \phi_{f_m}^{t_m} \; \mid \; t_1,\ldots,t_m\in \R, \; f_1,\ldots, f_m\in S \right\}
    \end{equation}
    is dense in $\textup{Ham}(M_{g,n}(\mathbf{t}), \omega)$ with respect to the $C^\infty$-topology.
\end{prop}

\begin{proof}
    Let $M_{g,n}(\mathbf{t})_\C$ be the moduli space of irreducible flat $SL(2,\C)$ connections on $\Sigma\smallsetminus \mathcal{P}$ such that the holonomy around a small loop encircling $p_j\in \mathcal{P}$ has trace $2\cos(\pi t_j)$, for each $j\in \{1,\ldots, n\}$. Then $M_{g,n}(\mathbf{t})_\C$ is a complex analytic manifold. It has a complex analytic symplectic form $\omega_\C$ defined just as in \eqref{eq:sympformonmoduli}, with $\alpha,\beta$ now in $\Omega^1(\Sigma;\mathfrak{s}\mathfrak{l}(2,\C))$. It inherits a real structure $\tau$ compatible with $\omega_\C$, induced by the involution $A\mapsto A^*$ on $SL(2,\C)$. By \cite[Lemma III.1.1]{morgan-shalen} (see also \cite{acosta}), the real locus consists of those connections which are gauge equivalent to either an $SU(2)$ or $SL(2,\R)$ flat connection. In particular, under the non-integral assumption on $\mathbf{t}$, $(M_{g,n}(\mathbf{t}),\omega)$ is symplectically embedded inside the real locus of $(M_{g,n}(\mathbf{t})_\C,\omega_\C,\tau)$.

    Consider the following diagram, where vertical arrows are inclusions:
    \begin{equation}
    \begin{tikzcd}
        M_{g,n}(\mathbf{t}) \arrow[d] \arrow[r, "\iota_\R"] & \R^N \arrow[d] \\
        M_{g,n}(\mathbf{t})_\C \arrow[r, "\iota_\C"] & \C^N
    \end{tikzcd}
    \end{equation}
    The top arrow $\iota_\R$ is the embedding $G\circ F$ from the proof of Proposition \ref{prop:firsttraceapprox}, whose components are defined by the trace functions \eqref{eq:tracefunctionsforembedding}, here interpreted as traces of holonomies along the corresponding loops. These trace functions have natural complex-valued extensions to $M_{g,n}(\mathbf{t})_\C$ and together they define the bottom arrow $\iota_\C$ which is a complex analytic embedding. 
 
    By the proof of Proposition \ref{prop:firsttraceapprox}, a trace monomial $T_\Gamma$ for a multicurve $\Gamma$ on $\Sigma_{g,n}$ is equal to $\eta\circ \iota_\R$ for some polynomial $\eta:\R^N\to \R$. Viewing $\eta$ as a complex polynomial, $\eta\circ \iota_\C$ gives a complex analytic extension of $T_\Gamma$. It is now clear that $(M_{g,n}(\mathbf{t}),\omega)$ and the set $S$ of trace monomials on multicurves satisfy the hypotheses of Lemma \ref{lemma:improveregularity}, yielding the result.
\end{proof}

We next compute $\phi_f^t$ for any given $f\in S$. In fact, our computation will be more general. Let $\Gamma=\{\gamma_1,\ldots,\gamma_k\}$ be a multicurve on $\Sigma_{g,n}$ and $\eta:\R^k\to \R$ a smooth function. Consider
\[
    f_{\eta,\Gamma}:M_{g,n}(\mathbf{t})\to \R,
\]
\begin{equation}\label{eq:tracefunctionetasurface}
    f_{\eta,\Gamma}(\rho)  := \frac{1}{4\pi^2}  \eta( \text{tr}_{\gamma_1}(\rho), \ldots, \text{tr}_{\gamma_k}(\rho)).
\end{equation}
Note that choosing $\eta(x_1,\ldots,x_k)=4\pi^2 x_1 \cdots x_k$ yields the trace monomial $f_{\eta,\Gamma}=T_\Gamma$.

To state a formula for $\phi_f^t$, we first introduce some notation. Let $\gamma$ be an embedded circle in a manifold $X$, and $E$ be an $SU(2)$-bundle over $X$. As above, define $\text{tr}_{\gamma}$ to be the gauge-invariant function on the space of connections on $E$ obtained by taking the trace of the holonomy of a connection around $\gamma$, based at any point. The derivative of $\text{tr}_{\gamma}$ at $A$ is given by
\begin{equation}\label{eq:derivativeoftracehol}
     (d\text{tr}_{\gamma})_A(\alpha) = -\int_{\gamma} \text{tr}(T_{A,\gamma} \alpha ),
\end{equation}
where $\alpha\in \Omega^1(X,\text{ad}E)$. In this formula, $T_{A,\gamma}$ is the section of $\text{ad}E|_{\gamma}$ defined as follows. First, let $\Pi:SU(2)\to \mathfrak{s}\mathfrak{u}(2)$ be the conjugation-invariant function
\begin{equation}\label{eq:piprojmap}
   \Pi(g) := g - \tfrac{1}{2}\text{tr}(g)\text{id}
\end{equation} 
Then, for a point $p$ on $\gamma$, the holonomy of the connection $A$ along $\gamma$ determines an automorphism $\text{hol}_\gamma(A)$ of $E_p$, and $\Pi(\text{hol}_\gamma(A))$ is an element of $(\text{ad}E)_p \subset \text{End}(E)_p$ which gives the value of $T_{A,\gamma}$ at $p$. See \cite[Lemma 3]{braam-donaldson} for more details.

Fix a multicurve $\Gamma=\{\gamma_1,\ldots,\gamma_{k}\}$ on $\Sigma_{g,n}$.  Orient each $\gamma_i$, and choose an orientation-preserving identification of a tubular neighborhood of $\gamma_i$ in $\Sigma^\circ$ with $S^1\times (-1,1)$. For $s\in (-1,1)$, let $\gamma_i^s$ be the loop parallel to $\gamma_i$ determined by the points in $S^1\times (-1,1)$ with second coordinate equal to $s$. 
   Let $\beta_i$ be a 1-form supported on this neighborhood, given by the pullback of a compactly supported $1$-form on $(-1,1)$ with integral $1$. 
   In particular, $\beta_i$ restricted to $\gamma_i$ vanishes. We also arrange that the $\beta_i$ have pairwise disjoint supports, so that $\beta_i$ restricted to $\gamma_j$ vanishes for all $i,j$.
  For any connection $A$ on $\Sigma^\circ$, the sections $T_{A,\gamma_i^s}$ for different values of $s$ determine a section of $\text{ad}E$ over the tubular neighborhood of $\gamma_i$, which we denote by $T_{A,i}$.
   In particular, $T_{A,i}\beta_i$ is a well-defined element of $\Omega^1_c(\frak g)$. 
   If $A$ is flat, then $T_{A,i}$ is $d_A$-parallel and thus $T_{A,i}\beta_i$ is in the kernel of $d_A$.

\begin{prop}\label{prop:hamcomp}
    Let $\Gamma$ be a multicurve on $\Sigma_{g,n}$ and $\eta:\R^k\to \R$ a smooth function, and write $f$ for the function $f_{\eta,\Gamma}:M_{g,n}(\mathbf{t})\to \R$. Then the Hamiltonian vector field for $f$ may be represented by
    \[
        X_{f} = - \sum_{i=1}^k \partial_i \eta \cdot T_{A,i}\beta_i
    \]
     where the terms on the right are defined as above, and $\partial_i\eta=(\partial_i \eta)(\textup{tr}_{\gamma_1}(A),\ldots,\textup{tr}_{\gamma_k}(A))$.
     Furthermore, for any $[A]\in M_{g,n}(\mathbf{t})$, we have $\phi_{f}^t([A])=[A_t]$ with
     \begin{equation}\label{formula-phitf}
       A_t=A- t\sum_{i=1}^k \partial_i \eta \cdot T_{A,i}\beta_i.
     \end{equation}
\end{prop}

\begin{proof}
    First note that $\text{tr}_{\gamma_i}$ can be rewritten as 
    \[
	\text{tr}_{\gamma_i}=\int_{(-1,1)}\text{tr}_{\gamma_i^s}\cdot \beta_i
    \]
    because holonomies of a flat connection on the loops $\gamma_i^s$ is independent of $s$.
    Let $[A]\in M_{g,n}(\mathbf{t})$ and $\alpha\in \Omega^1_c(\frak g)$ represent a tangent vector to $[A]$. 
    Using this relation and \eqref{eq:derivativeoftracehol}, we compute the derivative of $f(A)=\frac{1}{4\pi^2}\eta(\text{tr}_{\gamma_1}(A),\ldots, \text{tr}_{\gamma_k}(A))$ along $\alpha$ as
    \begin{equation}\label{eq:computehamiltonian}
        (df)_A(\alpha) =\frac{1}{4\pi^2} \int_{\Sigma}\text{tr}\big(\sum_{i=1}^k \partial_i \eta \cdot T_{A,i} \beta_{i} \wedge \alpha\big).
    \end{equation}
    By definition of the symplectic form $\omega$ in \eqref{eq:sympformonmoduli}, we obtain from the computation in \eqref{eq:computehamiltonian} that the Hamiltonian vector field for $f$ may be represented by 
    \[
        X_{f} = - \sum_{i=1}^k \partial_i \eta \cdot T_{A,i}\beta_i
    \]
    
    To identify the flow generated by the vector field $X_f$ on $M_{g,n}(\mathbf{t})$, first note that $A_t$ in \eqref{formula-phitf} is a flat connection because
    $\beta_i\wedge \beta_j$ is trivial for all $i,j$ and $T_{A,i}\beta_i$ is $d_A$ closed. Furthermore, the restrictions of $A_t$ and $A$ to each $\gamma_i^s$ agree with each other. Thus, $T_{A_t,i}$ is independent of $t$.
    This implies that $A_t$ represents a trajectory of the Hamiltonian flow of $f$ because it satisfies 
     \begin{equation*}\label{eq:floweq}
        \frac{d}{dt} A_t = - \sum_{i=1}^k \partial_i \eta \cdot T_{A_t, i} \beta_{\gamma_i}. \qedhere
    \end{equation*}   
 \end{proof}

Recall that any flat $SU(2)$ connection is determined by knowing the traces of its holonomies along any loop. Although not needed for the sequel, we describe $\phi_{f}^t([A])$ in these terms. Let $\sigma$ be an immersed loop in $\Sigma^\circ$ which is transverse to $\Gamma$, such that the intersections of $\sigma$ with $\Gamma$ occur at embedded points of $\sigma$. Orient $\sigma$ and choose an embedded point $q_0\in \sigma$ disjoint from $\Gamma$. Starting from $q_0$, traverse $\sigma$ in the direction of its orientation and let $q_1,\ldots,q_\ell$ be the intersection points of $\sigma$ with the curves in $\Gamma$ in the order in which they appear. Let $\gamma_{i_j}$ be the curve in $\Gamma$ containing $q_j$.

\begin{prop}\label{prop:hamcomphol}
    \[
        \textup{tr}_\sigma\left(\phi_{f}^t([A])\right) = \textup{tr}\left[\textup{exp}\left(t\partial_{i_\ell} \eta \cdot \Pi(\textup{hol}_{\gamma_{i_\ell}}(A) ) \right)\cdots \textup{exp}\left( t\partial_{i_1} \eta \cdot \Pi(\textup{hol}_{\gamma_{i_1}}(A) ) \right)  \textup{hol}_{\sigma}(A)\right].
    \]
\end{prop}
 
 \noindent The holonomies appearing are defined as follows. Choose a framing of the bundle at $q_0$ to define the holonomy $\text{hol}_\sigma(A)\in SU(2)$. Using the flat connection $A$, parallel transport the framing at $q_0$ to $q_j$ in the oriented direction of $\sigma$. Define $\text{hol}_{\gamma_{i_j}}(A)\in SU(2)$ using the framing at $q_j$ and the orientation of $\gamma_{i_j}$ for which $\smash{\gamma_{i_j}'(q_j)\wedge \sigma'(q_j)}$ agrees with the orientation of $\Sigma$. The proof of Proposition \ref{prop:hamcomphol} is a direct computation using the expression for $A_t$ from Proposition \ref{prop:hamcomp}.  

\subsection{Symplectomorphisms on $M_{g,n}(\mathbf{t})$ from the mapping class group}\label{sec:symplectofrommcg}

The mapping class group of $\Sigma_{g,n}$ provides another source of symplectomorphisms of $M_{g,n}(\mathbf{t})$. Let 
\[
	\psi:\Sigma_{g,n}\to \Sigma_{g,n}
\]
be a diffeomorphism. In particular, $\psi$ restricts to a permutation of $\mathcal{P}=\{p_1,\ldots,p_n\}$. We assume that if $\psi(p_i)=p_j$, then $t_i=t_j$. By applying an isotopy, we may assume that $\psi$ maps the disk neighborhood of $p_i$ to the disk neighborhood of $p_j$ by the identity map. In particular, if $A\in \mathcal{A}_{g,n}(\mathbf{t})$ is flat, then the pullback of $A$ by $\psi^{-1}$ is a flat connection in $\mathcal{A}_{g,n}(\mathbf{t})$. (To define this pullback, we lift $\psi$ into an isomorphism of the trivial $SU(2)$ bundle in the obvious way.) This induces a diffeomorphism of $M_{g,n}(\mathbf{t})$ which only depends on the class of $\psi$ in ${\rm Mod}(\Sigma_{g,n})$. The differential of this map is also given by taking pullback with respect to $\psi^{-1}$. From this we can see that this map is a symplectormorphism of $M_{g,n}(\mathbf{t})$. Thus we obtain a homomorphism from a finite index subgroup of ${\rm Mod}(\Sigma_{g,n})$ to the symplectomorphism group of $M_{g,n}(\mathbf{t})$.

 In the special case that $M_{g,n}(\mathbf{t})=M_{g,n}$, since all $t_i=1/2$, in the above construction there is no constraint on how the elements of ${\rm Mod}(\Sigma_{g,n})$ permute the marked points. By remembering only the symplectic isotopy class of the resulting map on $M_{g,n}$ we obtain a homomorphism
\begin{equation}\label{hom-ModSigma}
  {\rm Mod}(\Sigma_{g,n}) \to \pi_0 \text{Symp}(M_{g,n}).
\end{equation}

There is some flexibility in the above construction that allows us to produce additional symplectomorphisms of $M_{g,n}$. In the previous construction, we may use non-trivial lifts of $\psi$ to the level of $SU(2)$-bundles. First, we remark that if we change this lift by composing with a map $u:\Sigma^\circ\to SU(2)$ that is constant and diagonal in the disc neighborhoods of the punctures, we do not obtain a new symplectomorphism of $M_{g,n}$ because these maps are exactly the elements of the gauge group $\mathcal G_{g,n}$, which are modded out in the definition of $M_{g,n}$. 

However, given a map $v:\Sigma^\circ\to SO(3)$ which is constant and in $SO(2)\subset SO(3)$ in the punctured neighborhoods, we can still pull back a connection with respect to $v$ because $SO(3)$ is the adjoint group of $SU(2)$. Write ${\mathcal G}'_{g,n}$ for the space of all such maps. Changing $v\in  {\mathcal G}'_{g,n}$ to $ \bar u\cdot v$, where $\bar u$ is the composition of $u\in {\mathcal G}_{g,n}$ with the projection $SU(2)\to SO(3)$, does not change the induced map on $M_{g,n}$. In particular, we obtain a homomorphism from ${\mathcal G}'_{g,n}/\mathcal G_{g,n}$ into $\pi_0 \text{Symp}(M_{g,n})$. Any element $v\in {\mathcal G}'_{g,n}$ induces a homomorphism $v_*:H_1(\Sigma^\circ;\Z/2)\to H_1(SO(3);\Z/2)$, and this element of $H^1(\Sigma^\circ;\Z/2)$ is the obstruction to lift $v$ to a map $\Sigma^\circ\to SU(2)$. Thus, we have ${\mathcal G}'_{g,n}/\mathcal G_{g,n}\cong H^1(\Sigma^\circ;\Z/2)$ and the above construction gives a homomorphism
\begin{equation}\label{hom-H1}
 H^1(\Sigma^\circ;\Z/2) \to \pi_0 \text{Symp}(M_{g,n}).
\end{equation}
This homomorphism and \eqref{hom-ModSigma} can be combined to give 
\begin{equation}\label{hom-H2}
  \rho_{g,n}: \Gamma_{g,n} \to \pi_0 \text{Symp}(M_{g,n}).
\end{equation}
where $\Gamma_{g,n}$ is the semi-direct product ${\rm Mod}(\Sigma_{g,n})  \ltimes H^1(\Sigma^\circ;\mathbb Z/2)$ corresponding to the standard action of ${\rm Mod}(\Sigma_{g,n}) $ on $H^1(\Sigma^\circ;\mathbb Z/2)$.

An alternative definition of \eqref{hom-H2} can be given using the representation-theoretic description of $M_{g,n}$. The action of $[\psi]\in \text{Mod}(\Sigma_{g,n})$ on $M_{g,n}$ defining the homomorphism \eqref{hom-ModSigma} is obtained by viewing an element of  $M_{g,n}$ as a respresentation $\pi_1(\Sigma^\circ)\to SU(2)$ and precomposing with the automorphism of $\pi_1(\Sigma^\circ)$ induced by $\psi$. For \eqref{hom-H1}, let $(A_1,B_1,\ldots,  A_g,B_g,C_1,\ldots,C_n)$ represent an element of $M_{g,n}$ as in definition \eqref{rep-theoretic-desc}. Then $\zeta\in H^1(\Sigma^\circ;\mathbb Z/2)$ sends this element to
\[
  ((-1)^{\zeta(a_1)}A_1,(-1)^{\zeta(b_1)}B_1,\ldots, (-1)^{\zeta(a_g)} A_g,(-1)^{\zeta(b_g)}B_g,(-1)^{\zeta(c_1)}C_1,\ldots,(-1)^{\zeta(c_n)}C_n)
\]
using the notation of \eqref{eq:pi1presentationforpuncturedsurface}. Here we are using that the conjugacy class $\Gamma(1/2)$ is invariant with respect to multiplication by $-1$. For general $\bf t$, we can define similar symplectomorphisms of $M_{g,n}(\mathbf{t})$ using elements $\zeta\in H^1(\Sigma^\circ;\mathbb Z/2)$ that satisfy the constraint $\zeta(c_i)=0$ whenever $t_i\neq 1/2$.

Finally, we recall some facts about the cohomology ring of $M_{g,n}$ which can be found in \cite[Theorem 3.1]{xie-zhang}. The ring $H^\ast(M_{g,n};\Q)$ is a quotient of the polynomial ring
\[
	\Q[\alpha,\delta_1,\ldots,\delta_n,\psi_1,\ldots,\psi_{2g}]
\]
by an ideal of relations, where the degrees of $\alpha$ and $\delta_i$ are $2$, and that of $\psi_j$ is $3$. The class $\alpha$ is a multiple of the symplectic form of $M_{g,n}$, each $\delta_i$ corresponds to a marked point $p_i\in \mathcal{P}$, and each $\psi_j$ corresponds to the loop $a_j$ for $1\leq j\leq g $ and to the loop $b_{j-g}$ for $g+1 \leq j\leq 2g$. If $3g+n>4$, then the lowest degree nonzero element in the ideal of relations has degree at least $4$, and 
\begin{gather}
	H^2(M_{g,n};\Q) = \Q\alpha \oplus \Q \langle\mathcal{P}\rangle \nonumber \\[2mm]
	H^3(M_{g,n};\Q) =H_1(\Sigma;\Q) \label{eq: h1h3idmod}
\end{gather}
where $\Q \langle\mathcal{P}\rangle$ is the vector space with basis the set of marked points $\mathcal{P}$. Suppose $\phi\in \text{Mod}(\Sigma_{g,n})$. With these identifications, the map on $H^\ast(M_{g,n};\Q)$ induced by the symplectomorphism associated to $\phi$ fixes $\alpha$ (since it is a symplectomorphism), acts on $\Q\langle \mathcal{P}\rangle $ using the permutation $\phi|_{\mathcal{P}}$, and acts on $H^3(M_{g,n};\Q) $ in the same way as map $\phi_\ast$ on $H_\ast(\Sigma;\Q)$. 

The action $\zeta^\ast$ on $H^\ast(M_{g,n};\Q)$ induced by the symplectomorphism of $M_{g,n}$ coming from an element $\zeta\in H^1(\Sigma^\circ;\Z/2)$ satisfies the following:
\begin{gather*}
	\zeta^\ast(\alpha)=\alpha, \qquad \zeta^\ast(\delta_i) = (-1)^{\zeta(c_i)} \delta_i,  \qquad \zeta^\ast(\psi_j) =\psi_j
\end{gather*} 
These actions $\zeta^\ast$ on the cohomology ring are called {\emph{flip symmetries}} in \cite{xie-zhang}. That $\zeta^\ast$ fixes the $\psi_j$ classes goes back to an observation of Thaddeus \cite[Lemma 12.1]{thaddeus-intro}.

 Let $\text{Symp}_h(M_{g,n})$ denote the subgroup of symplectomorphisms in $\text{Symp}(M_{g,n})$ which induce the identity map on the cohomology ring $H^\ast(M_{g,n};\Z)$. Define the {\emph{pure}} mapping class group,
 \[
 	\textup{PMod}(\Sigma_{g,n}),
 \] 
 to be the subgroup of $\textup{Mod}(\Sigma_{g,n})$ consisting of the mapping classes that act identically on the set of marked points $\mathcal{P}$.
 Then the above discussion, combined with the fact that the cohomology $H^\ast(M_{g,n};\Z)$ is torsion-free, yields the following result.

  \begin{prop}\label{prop:puremappingclassgroup}
 	If $3g+n>4$, then $\rho_{g,n}^{-1}\left( \pi_0(\textup{Symp}_h(M_{g,n})) \right)=  \textup{PMod}(\Sigma_{g,n}) \ltimes H^1(\Sigma;\Z/2) $.
 \end{prop}
 
 \noindent Here we have identified $H^1(\Sigma;\Z/2)$ with its image under $H^1(\Sigma;\Z/2)\to H^1(\Sigma^\circ;\Z/2)$. Note that for $g=0$, we obtain $\rho_{0,n}^{-1}\left( \pi_0(\textup{Symp}_h(M_{0,n})) \right)=  \textup{PMod}(\Sigma_{0,n})$.

%!TEX root = main.tex

\section{Lagrangians from 3-manifolds and holonomy perturbations}\label{sec:holonomyandlagrangians} 
We next discuss the 3-dimensional counterpart of the construction of the previous section. To start, let $M$ be a compact connected oriented $3$-manifold with boundary $\partial M=\Sigma$ and $T= T_1 \cup \cdots \cup T_m \subset M$ be a tangle with boundary $\partial T = \mathcal{P} \subset \Sigma$. We also fix $\mathbf{t}=(t_1,\ldots,t_m)\in [0,1]^{m}$, which we refer to as a holonomy parameter for the tangle $T$. Analogous to $M_{g,n}(\mathbf{t})$, define the character variety $\mathfrak{X}(M,T;\mathbf{t})$ to be the moduli space of all conjugacy classes of representations 
\[
  \rho: \pi_1(M\smallsetminus T) \to SU(2)
\]
that map the conjugacy class of a meridian of $T_i$ to $\Gamma(t_i)\subset SU(2)$. Since each conjugacy class $\Gamma(t)$ is invariant with respect to inversion, the latter constraint does not depend on the choice of orientations for the meridians of $T$.

The character variety $\mathfrak{X}(M,T;\mathbf{t})$ is not necessarily smooth, and we may need to perturb it to obtain a smooth manifold. In fact, the perturbed character variety is expected to give an immersed Lagrangian in the character variety associated to $(\Sigma,\mathcal P)$. In \S \ref{subsec:unperturbed}, we study the unperturbed character variety $\mathfrak{X}(M,T;\mathbf{t})$ and give a sufficient condition to ensure that it is cut out transversely and determines a Lagrangian. Then, in \S \ref{hol-pert}, we introduce holonomy perturbations for a tuple $(M,T;\mathbf{t})$ and show that they may be used to approximate Hamiltonian isotopies of the symplectic manifold $M_{g,n}(\mathbf{t})$ in a suitable sense, proving Theorem \ref{thm:hamiltoniandiffeo}. Finally, in \S \ref{subsec:perturbed-3man-Lag} we combine the earlier results of this section to discuss how for a more general $(M,T;\mathbf{t})$, we may use holonomy perturbations to construct immersed Lagrangians for symplectic manifolds of the form $M_{g,n}(\mathbf{t})$.

\subsection{Unperturbed 3-manifold Lagrangians}\label{subsec:unperturbed}

As in the 2-dimensional case, we need a more geometrical description of $\mathfrak{X}(M,T;\mathbf{t})$, in terms of flat connections. For simplicity, we assume for now that the holonomy parameter $t_j$ of any closed component $T_j$ of $T$ belongs to the open interval $(0,1)$. Identify a tubular neighborhood of $T$ with $T\times D^2$, and let $\theta_j$ and $\phi_j$ be coordinate functions on $T_j\times D^2$ respectively given by the angular coordinate on $D^2$ and the standard coordinate on $T_j$ after its identification with $\R/\Z$ or $[0,1]$ depending on whether it is closed or not. Define $\mathcal A_{M,T}(\mathbf{t})$ as the space of connections on $M^\circ:=M\smallsetminus T$ whose restriction to a neighborhood of $T_j=T_j\times \{0\}$ in $T_j\times D^2$ is of the form
\begin{equation*}\label{connection-local-model-1}
\left[  
\begin{array}{cc}
	\frac{it_j}{2}&0\\
0&-\frac{it_j}{2}
\end{array}
\right]d\theta_j
\end{equation*}
if $T_j$ is an arc component of $T$, and is of the form
\begin{equation*}\label{connection-local-model-2}
\left[  
\begin{array}{cc}
	\frac{is_j}{2}&0\\
0&-\frac{is_j}{2}
\end{array}
\right]d\phi_j+\left[  
\begin{array}{cc}
	\frac{it_j}{2}&0\\
0&-\frac{it_j}{2}
\end{array}
\right]d\theta_j
\end{equation*}
if $T_j$ is closed. In the latter case, $s_j$ is an arbitrary real number.

The space of connections $\mathcal A_{M,T}(\mathbf{t})$ is an affine space modeled over $\widetilde {\Omega}^1(M^\circ,\mathfrak g)$, which is the space of all $1$-forms on $M^\circ$ that have the form 
\begin{equation*}\label{affine-space}
\left[  
\begin{array}{cc}
	\frac{ir_j}{2}&0\\
0&-\frac{ir_j}{2}
\end{array}
\right]d\phi_j
\end{equation*}
for some $r_j\in \mathbb R$ in a neighborhood of any closed component $T_j$ of $T$. Let $\Omega_c^i(M^\circ,\mathfrak g)$ denote the space of all compactly supported $i$-forms on $M^\circ$ with values in the Lie algebra $\mathfrak g$. Then $\widetilde {\Omega}^1(M^\circ,\mathfrak g)$ and $\Omega_c^i(M^\circ,\mathfrak g)$ fit into an exact sequence which is given as the middle column in Figure \ref{diagram-complex}. In this diagram, $\mathfrak c$ denotes the subset of $\{1,\cdots,n\}$ determined by the closed components of $T$.

We may pull back any connection in $\mathcal A_{M,T}(\mathbf{t})$ by a gauge transformation $u:M^\circ\to SU(2)$ of the trivial $SU(2)$ bundle over $M^\circ$, and we define $\mathcal G_{M,T}(\mathbf{t})$ as the group of all such gauge transformations $u$ that map $\mathcal A_{M,T}(\mathbf{t})$ into $\mathcal A_{M,T}(\mathbf{t})$ by taking pullback. This imposes the following constraints on $u$. For an arc component $T_j$ of $T$, $u$ in a neighborhood of $T_j$ is a constant map into $SU(2)$ if $t_j=0$, is of the form $\text{diag}(e^{ 2i\pi r_j},e^{-2i\pi r_j})$ for some $r_j$ if $t_j\in (0,1)$, and is of the form $\text{diag}(e^{-i\theta_j/2},e^{i\theta/2}) \cdot u_0\cdot  \text{diag}(e^{ i\theta/2},e^{-i\theta/2})$ for some $u_0\in SU(2)$ if $t_j=1$. For a closed component $T_j \subset T$, by our assumption we have $t_j\in (0,1)$ and $u$ is of the form
\[
	\text{diag}(e^{ 2i\pi (r_j+m_j\phi_j)},e^{-2i\pi (r_j+m_j\phi_j)})
\]
 in a neighborhood of $T_j$, where $m_j \in \Z$ and $r_j\in \R$. The Lie algebra of the gauge group $\mathcal G_{M,T}(M^\circ;\mathbf{t})$ is the space $\widetilde \Omega^0(M^\circ,\frak{g})$ of $0$-forms on $M^\circ$ with restriction to a neighborhood of $T_j$ for each $j$ a constant map into $\frak{g}$ if $t_j=0$, of the form $\text{diag}(ir_j,-ir_j)$ for some $r_j$ if $t_j\in (0,1)$, and $\text{diag}(e^{ i\theta_j/2},e^{-i\theta_j/2}) \cdot \zeta_j\cdot  \text{diag}(e^{ -i\theta_j/2},e^{i\theta_j/2})$ for some $\zeta_j\in \frak{g}$ if $t_j=1$.

 The space $\mathfrak{X}(M,T;\mathbf{t})$ may be identified with the space of $\bA\in \mathcal A_{M,T}(\mathbf{t})$ such that 
\begin{equation}\label{flat}
  F_{\bA}=0,
\end{equation}
modulo the action of $\mathcal G_{M,T}(\mathbf{t})$. To obtain a smooth space, we need to guarantee that all elements of this moduli space are irreducible and cut out transversely. 

Next, we introduce conditions on $(M,T;\mathbf{t})$ that ensure that every element of $\mathfrak{X}(M,T;\mathbf{t})$ is irreducible. We require that the boundary surface $\Sigma = \partial M$ has two connected components $\Sigma_\pm$ with genera $g_\pm$. Let $\mathcal P_\pm:=\mathcal P\cap \Sigma_\pm$ be given as
\[
  \mathcal P_+=\{x_1,\dots,x_{n_+}\},\hspace{1cm} \mathcal P_-=\{x_1',\dots,x_{n_-}'\},
\]
such that $x_i$ belongs to the connected component $T_{j_i}$ and $x_i'$ belongs to the connected component $T_{k_i}$. In short we say that the boundary of $(M,T;\mathbf{t})$, denoted by $\partial(M,T;\mathbf{t})$, is equal to $(\Sigma_+,\mathcal P_+,\mathbf{t}_{+})\sqcup (\Sigma_-,\mathcal P_-,\mathbf{t}_{-})$, where 
\[
  \mathbf{t}_{+}=(t_{j_1}, \dots, t_{j_{n_+}}),\hspace{1cm}\mathbf{t}_{-}=(t_{k_1}, \dots, t_{k_{n_-}}).
\]
We say $(M,T;\mathbf{t})$ satisfies the non-integral assumption if $\mathbf t_{\pm}$ satisfy the non-integral assumption in the previous sense. In general, restriction to the boundary determines the correspondence 
\begin{equation}
\begin{tikzcd}[column sep=small, row sep=large]
	&\mathfrak{X}(M,T;\mathbf{t}) \arrow[dl, "r_-"'] \arrow[dr, "r_+"] & \\
	M_{g_{-},n_{-}}(\mathbf{t}_{-}) & & M_{g_{+},n_{+}}(\mathbf{t}_{+})
\end{tikzcd}\label{eq:bundlecorrespondence-general}
\end{equation}
In particular, under the non-integral assumption, any element of $\mathfrak{X}(M,T;\mathbf{t})$ is irreducible because its restriction to the boundary is irreducible.

Unlike the 2-dimensional case, $\mathfrak{X}(M,T;\mathbf{t})$ is not necessarily cut out transversely even after we guarantee irreducibility of the elements of this moduli space. To achieve transversality we either need further assumptions on $(M,T;\mathbf{t})$ or to perturb \eqref{flat}. First we discuss a sufficient condition that ensures that the elements of  $\mathfrak{X}(M,T;\mathbf{t})$ are cut out transversely. 
\begin{prop}\label{unperturbed-3-man-Lag}
	Suppose the inclusion map $\Sigma_+\hookrightarrow M$ induces a surjective map from $\pi_1(\Sigma_+\smallsetminus \mathcal P_+)$ to $\pi_1(M\smallsetminus T)$. Then $\mathfrak{X}(M,T;\mathbf{t})$ is a smooth manifold. 	
	Furthermore, the restriction map 
	\begin{equation}\label{r-Lag}
	 r=(r_+,r_-):\mathfrak{X}(M,T;\mathbf{t})\to M_{g_{+},n_{+}}(\mathbf{t}_{+})\times M_{g_{-},n_{-}}(\mathbf{t}_{-})
	\end{equation}
	defines a Lagrangian embedding. 
\end{prop}
\begin{proof}
	Fix a connection $\bA_0$ satisfying $F_{\bA_0}=0$. The linearization of $F_{\bA}$ as a map from $ \mathcal A_{M,T}(\mathbf{t})$ to $\Omega^2_c(M^\circ,\frak{g})$ at $\bA_0$ is given by 
	\[
	    d_{\bA_0}:\widetilde  \Omega^1(M^\circ,\frak{g})\to  \Omega^2_c(M^\circ,\frak{g}).
	\]
	However, this map is not surjective even under the given assumption on fundamental groups. Following \cite{herald}, we replace this map as follows. Let $\Pi_{\bA_0}$ be the projection map $\Omega^2_c(M^\circ,\frak{g})$ to the kernel of 
	${\rm ker}(d_{\bA_0}):\Omega^2_c(M^\circ,\frak{g}) \to \Omega^3_c(M^\circ,\frak{g})$. Then 
	\begin{equation}\label{flat-proj}
	  	\Pi_{\bA_0}\circ F_{\bA}:\mathcal A_{M,T}(\mathbf{t}) \to \Omega^2_c(M^\circ,\frak{g})\cap \ker (d_{\bA_0})
	\end{equation}
	and $F_{\bA}$ have the same zeroes in a small enough neighborhood of $\bA_0$.
	To see this, note that $d_{\bA}F_{\bA}=0$ by the Bianchi identity. Thus if $\Pi_{\bA_0}\circ F_{\bA}=0$, then $F_\bA$ belongs to $\ker(d_{\bA})$ and is orthogonal to 
	$\ker(d_{\bA_0})$, which is impossible if $F_{\bA}\neq 0$ and $\bA$ is close enough to $\bA_0$. 	To analyze the behavior of $\mathfrak{X}(M,T;\mathbf{t})$ in a neighborhood of $[\bA_0]$, we may thus consider the linearization of 
	\eqref{flat-proj} instead of $F_{\bA}$. This linearization is given by the map 
	\[
		d_{\bA_0}:\widetilde  \Omega^1(M^\circ,\frak{g}) \to \Omega^2_c(M^\circ,\frak{g})\cap {\rm ker}(d_{{\bA_0}}),
	\]
	which appears in the middle row of the diagram in 
	Figure \ref{diagram-complex}. If this operator is surjective for any $[\bA_0]$, then $\mathfrak{X}(M,T;\mathbf{t})$ 
	is a smooth manifold. 
	 
        \begin{figure}[t]
            \centering
                \begin{tikzcd}[column sep=large, row sep=large]
                & 0 \arrow[d] & 0 \arrow[d] & 0 \arrow[d] &\\ 
                0\arrow[r] &\Omega_c^0(M^\circ,\frak{g})\arrow[d]\arrow[r, "d_{{\bA_0}}"] & 
                  \Omega^1_c(M^\circ,\frak{g}) \arrow[d] \arrow[r, "d_{{\bA_0}}"] &   \Omega^2_c(M^\circ,\frak{g})\cap {\rm ker}(d_{{\bA_0}})\arrow[r] \arrow[d, equal]&0\\
                0\arrow[r] &\widetilde \Omega^0(M^\circ,\frak{g})\arrow[r, "d_{{\bA_0}}"] \arrow[d]& 
                  \widetilde  \Omega^1(M^\circ,\frak{g}) \arrow[r, "d_{{\bA_0}}"] \arrow[d]&   \Omega^2_c(M^\circ,\frak{g})\cap {\rm ker}(d_{{\bA_0}})\arrow[r]\arrow[d]&0\\
               0\arrow[r] & \displaystyle \bigoplus_{1\le j\le n} \mathfrak{h}_j \arrow[r, "0"]  \arrow[d] &\displaystyle \bigoplus_{j\in \mathfrak c} \mathfrak{h}_j \arrow[r] \arrow[d] & 0 \arrow[d]\arrow[r] &0\\
                & 0  & 0  & 0 &
            \end{tikzcd}
            \caption{}
            \label{diagram-complex}
        \end{figure}

	The other non-trivial map in the middle of Figure \ref{diagram-complex} is the linearization of the action of the gauge group. In particular, this row 
	gives a deformation complex for $\mathfrak{X}(M,T;\mathbf{t})$ in a neighborhood of $[\bA_0]$.
	Let $H^i_{\bA_0}$ denote the cohomology groups of this complex. Since $\bA_0$ is irreducible, $H^0_{\bA_0}$ is trivial, and we wish to show that $H^2_{\bA_0}$ is also trivial. 
	Each row of the diagram in Figure \ref{diagram-complex} defines a cochain complex. The cohomology groups of the complex in the first row can be identified with
	the cohomology with compact support group $H^i_c(M^\circ;{\rm ad}_{\bA_0})$. 
	If $\overline M$ denotes the complement of a tubular neighborhood $N(T)$ of $T$ in $M$ and $S(T)$ denotes the boundary of 
	$N(T)$ minus the disk neighborhoods of $\mathcal P$ in $\Sigma$, then $H^i_c(M^\circ;{\rm ad}_{\bA_0})$ can be identified with $H^i(\overline M,S(T);{\rm ad}_{\bA_0})$. In particular, the boundary of $\overline M$ is the union of $S(T)$ and 
	$\overline \Sigma$, and $H^i_c(M^\circ;{\rm ad}_{\bA_0})$ is isomorphic to $H^{3-i}(\overline M,\overline \Sigma;{\rm ad}_{\bA_0})$ by Poincar\'{e}-Lefschetz duality. 
	
	The vertical arrows of the diagram in Figure \ref{diagram-complex} define an exact sequence of cochain complexes. The associated exact sequence of cohomology groups implies that there is a surjective map
	\begin{equation}\label{surj-tangent}
	  H^2_c(M^\circ;{\rm ad}_{\bA_0})\cong H^{1}(\overline M,\overline \Sigma;{\rm ad}_{\bA_0}) \twoheadrightarrow H^2_{\bA_0}
	\end{equation}	
	and induces the following relation at the level of Euler characteristics:
	\begin{equation}\label{dim-Lag}
	  \dim H^1_{\bA_0}-\dim H^2_{\bA_0}=-\sum_{j\notin \mathfrak c}\dim \mathfrak h_j-\sum_{i=0}^2 (-1)^{i} \dim H^{i}(\overline M,S(T);{\rm ad}_{\bA_0}).
	\end{equation}
	The exact sequence for the pair $(\overline M,\overline \Sigma)$ with twisted coefficients is given by
	\[
	  \begin{tikzcd}
                \dots\arrow[r] &H^i(\overline M,\overline \Sigma;{\rm ad}_{\bA_0})\arrow[r] & 
                H^i(\overline M;{\rm ad}_{\bA_0})\arrow[r] &H^i(\overline \Sigma;{\rm ad}_{A_0}) \arrow[r] & \dots
            \end{tikzcd}
	\]
	where $A_0$ is the restriction of $\bA_0$ to the boundary. Since $\bA_0$ and $A_0$ are irreducible, all the cohomology groups in this sequence with $i=0$ vanish. Note that $\overline \Sigma$ has the connected components $\overline \Sigma_+$, 
	$\overline \Sigma_-$, and our assumption implies that the inclusion of $\overline \Sigma_+$ in $\overline M$ induces a surjective map of fundamental groups. In particular, the map from $H^1(\overline M;{\rm ad}_{\bA_0})$ to $H^1(\overline \Sigma;{\rm ad}_{A_0})$ is injective.
	Consequently, $H^1(\overline M,\overline \Sigma;{\rm ad}_{\bA_0})$ is trivial, and the surjection in \eqref{surj-tangent} implies that $H^2_{\bA_0}$ is trivial. Thus, the moduli space $\mathfrak{X}(M,T;\mathbf{t})$ is a smooth manifold of dimension
	\begin{equation*}\label{dim-Lag-2}
	  \dim H^1_{\bA_0}=-\sum_{j\notin \mathfrak c}\dim \mathfrak h_j-\sum_{i=0}^3 (-1)^{i} \dim H^{i}(\overline M,S(T);{\rm ad}_{\bA_0}).
	\end{equation*}
	This identity follows from \eqref{dim-Lag}, and the observation $H^{3}(\overline M,S(T);{\rm ad}_{\bA_0})\cong H^{0}(\overline M,\overline \Sigma;{\rm ad}_{\bA_0})$ is trivial. 
	The second sum on the right hand side is the Euler characteristic for the cohomology groups of the pair $(\overline M,S(T))$ with coefficients in the 
	flat rank three bundle ${\rm ad}_{\bA_0}$. In particular, this Euler characteristic is equal to $3\chi(\overline M,S(T))$, and thus
	\begin{equation}\label{dim-Lag-2.5}
	  \dim H^1_{\bA_0}=-\sum_{j\notin \mathfrak c}\dim \mathfrak h_j-3\chi(\overline M,S(T)).
	\end{equation}
	where $\chi(\overline M,S(T))$ denotes the Euler characteristic of the pair $(\overline M,S(T))$,
	In the following, we use a similar notation for Euler characteristics of pairs. 
	
	The exact sequence of cohomology groups for the triple 
	$(\overline M,\partial \overline M,S(T))$ and the excision isomorphism $H^i(\partial \overline M,S(T))\cong H^i(\overline \Sigma,\partial \overline \Sigma)$ give rise to the exact sequence
	\[
	  \begin{tikzcd}
                \dots\arrow[r] & 
                H^i(\overline M,\partial \overline M)\arrow[r] &H^i(\overline M,S(T)) \arrow[r] & H^{i}(\overline \Sigma,\partial \overline \Sigma)\arrow[r] &\dots
            \end{tikzcd}
	\]
	This exact sequence implies that 
	\[
	  \chi(\overline M,S(T))=\chi(\overline \Sigma,\partial \overline \Sigma)+\chi(\overline M,\partial \overline M)
	\]
	A similar relation at the level of Euler characteristics for the exact sequence associated to the pair $(\overline M,\partial \overline M)$ together with Poincar\'e duality isomorphism $H^{i}(\overline M,\partial \overline M)\cong H^{3-i}(\overline M)$
	implies that 
	\begin{align*}
	  \chi(\overline M,\partial \overline M)&=-\frac{1}{2} \chi(\partial \overline M)=-\frac{1}{2} \chi(\overline \Sigma,\partial \overline \Sigma)\label{euler-char-eq}
	\end{align*}
	The second identity follows easily from the fact that $S(T)$ is a union of cylinders. Thus, we obtain
	\[
	  \chi(\overline M,S(T))=\frac{1}{2} \chi(\overline \Sigma,\partial \overline \Sigma)
	\]	
	Combining this identity and \eqref{dim-Lag-2.5} implies that the dimension of $\mathfrak{X}(M,T;\mathbf{t})$ is exactly half of the dimension of $M_{g_{+},n_{+}}(\mathbf{t}_{+})\times M_{g_{-},n_{-}}(\mathbf{t}_{-})$.
	
	The assumption on fundamental groups implies that the map $r_+:\mathfrak{X}(M,T;\mathbf{t})\to M_{g_{+},n_{+}}(\mathbf{t}_{+})$ is an injection. Furthermore, the map $r_+$ is an immersion. To see this, let $\bA_0$ be as above and
	 let $A_0^+$ denote its restriction to $\Sigma_+$. Then the derivative of $r_+$ determines a homomorphism 
	 \begin{equation}\label{r+*}
	 	 (r_+)_*:H^1_{\bA_0}\to H^1_{A_0^+},
	\end{equation} 
	which is induced by the restriction of elements of $\widetilde  \Omega^1(M^\circ,\frak{g})$ to 
	 the boundary component $\Sigma_+$ of $M$. Let $\zeta\in \widetilde  \Omega^1(M^\circ,\frak{g})$ represent an element of  $\smash{H^1_{\bA_0}}$ such that its restriction to $\Sigma_+$ represents the trivial element in $\smash{H^1_{A_0^+}}$.
	This implies that $\smash{d_{\bA_0}\zeta=0}$ and there is $\lambda\in \widetilde \Omega^0(\Sigma_+,\frak{g})$ such that $\smash{d_{A_0^+}\lambda}$ is equal to the restriction of $\zeta$ to $\Sigma_+$. 
	Now using the assumption on the fundamental groups and invoking a similar argument as in the proof of \cite[Lemma 4.23]{knot-complement-problem-nullhomotopic},  
	we may extend $\lambda$ into an element of 
	$\lambda\in \widetilde \Omega^0(M^\circ,\frak{g})$ that maps to $\zeta$ via $d_{\bA_0}$. In particular, $\zeta$ represents the trivial element in $H^1_{\bA_0}$. This shows that \eqref{r+*} is an injection. Thus, the map $r$ in 
	\eqref{r-Lag} defines an embedding. 
	
	To complete the proof, we need to show that $\mathfrak{X}(M,T;\mathbf{t})$ is an isotropic submanifold. Let 
	$\hat \alpha, \hat \alpha'\in \widetilde  \Omega^1(M^\circ,\frak{g})$ represent elements of $H^1_{\bA_0}$,
	and $\alpha=r_*(\hat \alpha)$, $\alpha'=r_*(\hat \alpha')$ denote their restriction to the boundary of $M$. Then by the Stokes theorem we have
	\begin{align*}
		\omega(\alpha,\alpha') 
		&= -\frac{1}{4\pi^2}\int_{\Sigma}\text{tr}(\alpha \wedge \alpha')\\[2mm]
		&= -\frac{1}{4\pi^2}\int_{M} d\,\text{tr}(\hat \alpha \wedge \hat \alpha')\\[2mm]
		&= -\frac{1}{4\pi^2}\int_{M} \text{tr}(d_{\bA_0} \hat \alpha \wedge \hat \alpha'-\hat \alpha \wedge d_{\bA_0} \hat \alpha')=0.
	\end{align*}
	This shows that $r$ maps the tangent space $H^1_{\bA_0}$ of  $\mathfrak{X}(M,T;\mathbf{t})$ into an isotropic subspace of the tangent space of $M_{g_{+},n_{+}}(\mathbf{t}_{+})\times M_{g_{-},n_{-}}(\mathbf{t}_{-})$.
\end{proof}

The above discussion can be generalized to the case that some of the closed components of $T$ have holonomy parameters $0$ and $1$. First we remove any component $T_i$ of $T$ with $t_i=0$ as it does not affect the character variety $\mathfrak{X}(M,T;\mathbf{t})$. If $w$ is the union of the closed components $T_i$ of $T$ with $t_i=1$, then we may form a $PU(2)$-bundle $P$ over $M$ with $w_2(P)$ represented by the Poincar\'e dual of $w$.  Then we repeat the above discussion after replacing $T$ with $T\smallsetminus w$, the trivial $SU(2)$-bundle with the bundle $P$ over the complement of a tubular neighborhood of $T\smallsetminus w$, and the gauge group $\mathcal G_{M,T}(\mathbf{t})$ with the determinant one gauge group. This gauge group consists of sections of the fiber bundle associated to $P$ corresponding to the adjoint action of $PU(2)$ on $SU(2)$ that satisfies similar constraints as above in tubular neighborhoods of the components of $T\smallsetminus w$. The gauge equivalence classes of flat connections can be still identified with $\mathfrak{X}(M,T;\mathbf{t})$. The analogue of Proposition \ref{unperturbed-3-man-Lag} then holds with the assumption that the inclusion of $\Sigma_+$ in $M$ induces a surjective map of $\pi_1(\Sigma_+\smallsetminus \mathcal P_+)$ into $\pi_1(M\smallsetminus (T\cup w) )$. For the rest of this section, we use a similar modification for the closed components of $T$ with holonomy parameters $0$ and $1$.

\subsection{Holonomy perturbations}\label{hol-pert}

For a general $(M,T;\mathbf{t})$, with no assumptions on the fundamental groups, the claim in Proposition \ref{unperturbed-3-man-Lag} does not necessarily hold, and we need to perturb the defining equation of $\mathfrak{X}(M,T;\mathbf{t})$ in order to guarantee smoothness. We next discuss the holonomy perturbations that may be used to perturb $\mathfrak{X}(M,T;\mathbf{t})$. Our treatment follows \cite{donaldson-book} adapted to case of non-closed 3-manifolds; some other references are \cite{ floer-zhs3, taubes-casson, braam-donaldson, herald, yaft}.

For a connection $\bA\in \mathcal A_{M,T}(\mathbf{t})$ and an embedded loop $\gamma$ in the interior of $M^\circ$, recall that $\text{tr}_{\gamma}(\bA)$ denotes the trace of the holonomy of $\bA$ along $\gamma$. Fix one such embedded loop $\gamma$, and extend this to an embedding of $S^1\times D^2$ into $\text{int}(M^\circ)$, and for $z\in D^2$ write $\gamma^z$ for the restriction of the embedding to $S^1\times \{z\}$. In particular, $\gamma=\gamma^0$. Let $\nu$ be a smooth positive $2$-form on $D^2$ with integral $1$. Define
\begin{equation}\label{holonomy-map-cylinder}
    \sigma_\gamma(\bA) = \int_{D^2} \text{tr}_{\gamma^z}(\bA) \nu(z)
\end{equation}
A priori, $\sigma$ depends on the extension of the embedding of $\gamma$ and the $2$-form $\nu$, but this is suppressed from the notation. Note that $\sigma_\gamma$ is invariant with respect to the action of $\mathcal G_{M,T}(\mathbf{t})$ and may be viewed as a real-valued function on the configuration space $\mathcal A_{M,T}(\mathbf{t})/\mathcal G_{M,T}(\mathbf{t})$.

More generally, given $\Gamma = \{\gamma_1,\ldots,\gamma_k\}$, where each $\gamma_i$ is an embedded loop in the interior of $M^\circ$, and a smooth function $\eta:\R^k\to \R$, we may form the gauge-invariant function
\[
    f_{\eta,\Gamma}(\bA)  = \eta( \sigma_{\gamma_1}(\bA), \ldots, \sigma_{\gamma_k}(\bA)).
\]
Such functions are called {\emph{admissible}} in \cite[\S 5.5]{donaldson-book}. A similar computation as in \eqref{eq:computehamiltonian} shows that for a tangent vector $\alpha\in \widetilde  \Omega^1(M^\circ,\frak{g})$ to $\mathcal A_{M,T}(\mathbf{t})$, the derivative of $f_{\eta,\Gamma}$ in the direction of $\alpha$ is
\begin{equation}\label{eq:derivativeoffetaGamma}
     (df_{\eta,\Gamma})_{\bA}(\alpha) = \int_{M} \text{tr}(\zeta_{\eta,\Gamma}(\bA) \alpha ),
\end{equation}
\begin{equation*}\label{zeta-etaGamma-A}
	\zeta_{\eta,\Gamma}(\bA)=-\sum_{i=1}^k \partial_i\eta(\text{tr}_{\gamma_1}({\bA}),\ldots,\text{tr}_{\gamma_k}({\bA})) T_{{\bA},i} \nu_i,
\end{equation*}
where $\nu_i$ is the $2$-form supported in the solid torus neighborhood of $\gamma_i$ obtained from the form $\nu$ on $D^2$ pulled back to $S^1\times D^2$, and $T_{{\bA},i}$ is the section of the adjoint bundle over the solid torus neighborhood of $\gamma_i$ whose restriction over each $\gamma_i^z$ is the section $T_{{\bA},\gamma_i^z}$ as described earlier in \eqref{eq:derivativeoftracehol}. Explicitly, the value of $T_{{\bA},\gamma_i^z}$ at $p=(t,z)\in S^1\times D^2$ is $\Pi(\textup{hol}_{\gamma_i^z}({\bA}))$ where $\Pi$ is defined in \eqref{eq:piprojmap} and the holonomy is viewed as an element of $SU(2)$. Since $\zeta_{\eta,\Gamma}(\bA)$ is a 2-form with values in the adjoint bundle and thus, it can be used to perturb \eqref{flat}, leading to the following equation:
\begin{equation}\label{eq:perturbedflat}
    F_{\bA} = \sum_{i=1}^k \partial_i\eta(\text{tr}_{\gamma_1}({\bA}),\ldots,\text{tr}_{\gamma_k}({\bA})) T_{{\bA},i} \nu_i.
\end{equation}
The set of gauge-equivalence classes of solutions to \eqref{eq:perturbedflat} is denoted by
\[
    \mathfrak{X}_{\eta,\Gamma}(M,T;\mathbf{t}).
\]
We refer to $(\Gamma,\eta)$ as the perturbation data of $\mathfrak{X}_{\eta,\Gamma}(M,T;\mathbf{t})$. The restriction of an element of $\mathfrak{X}_{\eta,\Gamma}(M,T;\mathbf{t})$ to the boundary is still flat because $\zeta_{\eta,\Gamma}(\bA)$ is trivial on the boundary of $M$. For example, assuming the boundary of $(M,T;\mathbf{t})$ is equal to $(\Sigma_+,\mathcal P_+,\mathbf{t}_{+})\sqcup (\Sigma_-,\mathcal P_-,\mathbf{t}_{-})$ with notation as in the previous subsection, restriction induces a map 
	\begin{equation*} 
	 r=(r_+,r_-):\mathfrak{X}_{\eta,\Gamma}(M,T;\mathbf{t})\to M_{g_{+},n_{+}}(\mathbf{t}_{+})\times M_{g_{-},n_{-}}(\mathbf{t}_{-})
	\end{equation*}
which generalizes the map \eqref{r-Lag} considered in the unperturbed case.

In the case that the loops $\gamma_i\in \Gamma$ are pairwise disjoint (with the thickened embeddings chosen disjointly), there is also a representation-theoretic description of solutions to \eqref{eq:perturbedflat}. In the following statement, the longitude of $\gamma_i$ is induced by the framing given by its extended embedding from $S^1\times D^2$, and $\Gamma$ is conflated with the union of the images of the $\gamma_i$ in $M$.  

\begin{prop}\label{prop:perturbedcriticalset}
    Suppose $\Gamma$ consists of pairwise disjoint embedded loops. Then the set of solutions to \eqref{eq:perturbedflat} modulo gauge is in bijection with the set of conjugacy classes of homomorphisms 
    \[
        \rho:\pi_1(M\smallsetminus (\Gamma\cup T)) \to SU(2)
    \]
    which satisfy the following, where $\mu_j$, $\lambda_j$ are the meridian and longitude for $\gamma_j$:
\begin{equation}\label{eq:perturbedcriticalsetconditions}
        \rho(\mu_j) = \textup{exp}\left( \partial_{j} \eta \cdot  \pi(\rho(\lambda_j) ) \right) \qquad (1\leq j \leq k).
    \end{equation}
 Furthermore, $\rho$ is required to map the conjugacy class of a meridian of $T_i$ to the conjugacy class $\Gamma(t_i)\subset SU(2)$ for all $1\leq i\leq m$.
\end{prop}  

\noindent This proposition in the case of closed 3-manifolds and empty tangles is essentially due to Floer, see \cite[\S 3]{floer-dehn} and \cite[Lemma 4]{braam-donaldson}. The same proof carries through without change to give a proof of Proposition \ref{prop:perturbedcriticalset}. An immediate consequence of this proposition is that, at least in the case that the loops in $\Gamma$ are disjoint, $\mathfrak{X}_{\eta,\Gamma}(M,T;\mathbf{t})$ is independent of the choice of $\nu$ and the extensions of the embeddings of the loops in $\Gamma$.

For a representation $\rho$ as in Proposition \ref{prop:perturbedcriticalset}, $\rho(\mu_j)$ and $\rho(\lambda_j)$ commute, and by an $SU(2)$ conjugation can be simultaneously diagonalized. We may assume this from the start:
\[
    \rho(\mu_j) = \left[\begin{array}{cc} e^{i \tau_j} &  0 \\ 0 & e^{-i\tau_j} \end{array}\right], \qquad \rho(\lambda_j) = \left[\begin{array}{cc} e^{i \sigma_j} &  0 \\ 0 & e^{-i\sigma_j} \end{array}\right].
\]
Then $\Pi(\rho(\lambda_j))$ is the diagonal matrix with entries $i\sin(\sigma_j)$, $-i\sin(\sigma_j)$. Using $\text{tr}_{\gamma_j}(\rho)=2\cos(\sigma_j)$ and $\partial_j\eta = \partial_j \eta(\text{tr}_{\gamma_1}(\rho),\ldots,\text{tr}_{\gamma_k}(\rho))$, equality \eqref{eq:perturbedcriticalsetconditions} yields
\[
    \tau_j  \equiv -\frac{1}{2}\frac{\partial}{\partial \sigma_j} \left( \eta(2\cos(\sigma_1),\ldots,2\cos(\sigma_k)\right) \pmod{2\pi}
\]

We end this subsection by a remark on the behavior of perturbed character varieties with respect to composing tangles. Let $(M,T;\mathbf{t})$ and $(M',T';\mathbf{t}')$ satisfy the non-integral assumption and
\begin{align*}
  \partial(M,T;\mathbf{t}) & =(\Sigma_-,\mathcal P_-,\mathbf{t}_{-})\sqcup (\Sigma,\mathcal P,\mathbf{s}), \\[2mm]
  \partial(M',T';\mathbf{t}') &=(-\Sigma,-\mathcal P,\mathbf{s})\sqcup (\Sigma_+,\mathcal P_+,\mathbf{t}_+).
\end{align*}
Then we have character varieties associated to $(M,T;\mathbf{t})$ and $(M',T';\mathbf{t}')$. More generally, let $(\eta, \Gamma)$ and $(\eta', \Gamma')$ be perturbation data for these tuples. Then we may have correspondences
\begin{equation*}
\begin{tikzcd}[column sep=small, row sep=large]
	&\mathfrak{X}_{\eta,\Gamma}(M,T;\mathbf{t}) \arrow[dl, "r_-"'] \arrow[dr, "r_+"] &&\mathfrak{X}_{\eta',\Gamma'}(M',T';\mathbf{t}') \arrow[dl, "r_-'"'] \arrow[dr, "r_+'"]&  \\
	M_{g_{-},n_{-}}(\mathbf{t}_{-}) & & M_{g,n}(\mathbf{s})& & M_{g_{+},n_{+}}(\mathbf{t}_{+})
\end{tikzcd}\label{eq:bundlecorrespondence-general-two}
\end{equation*}
Here $M_{g_{\pm},n_{\pm}}(\mathbf{t}_{\pm})$ and $M_{g,n}(\mathbf{s})$ are respectively character varieties associated to $(\Sigma_\pm,\mathcal P_\pm,\mathbf{t}_{\pm})$ and $(\Sigma,\mathcal P,\mathbf{s})$. To be more precise, the target of $r_-'$ is the character variety associated to $(-\Sigma,-\mathcal P,\mathbf{s})$, which as a symplectic manifold may be identified as $M_{g,n}(\mathbf{s})$ with the negative of its symplectic structure. Glue $(M,T)$ and $(M',T')$ along their common boundary component $(\Sigma,\mathcal P)$ to obtain 
\[
  (M_\#,T_\#)=(M,T) \cup_{(\Sigma,\mathcal P)}(M',T'),
\] 
and note that the holonomy parameters $\mathbf{t}$, $\mathbf{t}'$ induce a holonomy parameter $\mathbf{t}_\#$ for $(M_\#,T_\#)$ which satisfies the non-integral assumption. We obtain perturbation data $(\eta_\#, \Gamma_\#)$ for $(M_\#,T_\#;\mathbf{t}_\#)$ from $(\eta, \Gamma)$ and $(\eta', \Gamma')$. Then we have the following relation between the correspondences:
\begin{equation}\label{corr-relation}
	\mathfrak{X}_{\eta_\#,\Gamma_\#}(M_\#,T_\#;\mathbf{t}_\#)=\mathfrak{X}_{\eta,\Gamma}(M,T;\mathbf{t})\times_{M_{g,n}(\mathbf{s})}\mathfrak{X}_{\eta',\Gamma'}(M',T';\mathbf{t}'),
\end{equation}
where the right hand side denotes the fiber product over $M_{g,n}(\mathbf{s})$. The identification in \eqref{corr-relation} is compatible with the correspondence maps into $M_{g_{-},n_{-}}(\mathbf{t}_{-})$ and $M_{g_{+},n_{+}}(\mathbf{t}_{+})$. Furthermore, \eqref{corr-relation} does not require that the 3-manifold character varieties are cut out transversely.

\subsection{Approximating Hamiltonian diffeomorphisms by holonomy perturbations}
As a particular case of interest, let $M=[-1,1]\times \Sigma$ where $\Sigma$ has genus $g$, and $T=[-1,1]\times \mathcal{P}$ where $\mathcal{P}=\{p_1,\ldots,p_n\}$ are distinct points on $\Sigma$; we abbreviate $(M,T)=[-1,1]\times \Sigma_{g,n}$. Choose $\mathbf{t}\in [0,1]^n$ satisfying the non-integral assumption to prescribe the singular holonomy conditions around meridians of the arcs $[-1,1]\times \{p_j\}$ for $1\leq j \leq n$. There is a natural identification
\[
    \mathfrak{X}([-1,1]\times \Sigma_{g,n};\mathbf{t}) = M_{g,n}(\mathbf{t})
\]
by restriction of flat connections to any slice $\{s\}\times \Sigma_{g,n}$. 

Next suppose $\Gamma=\{\gamma_1,\ldots,\gamma_k\}$ is a multicurve on $\{0\}\times\Sigma_{g,n}$ and $\eta:\R^k\to \R$ is some smooth function. Following the general scheme, we form the perturbed-flat moduli space by using solid torus neighborhoods of the $\gamma_i$ which are pairwise disjoint and disjoint from the arcs $[-1,1]\times \{p_j\}$ and the boundary $\{\pm 1\}\times \Sigma$. We obtain a correspondence
\begin{equation}
\begin{tikzcd}[column sep=small, row sep=large]
& \mathfrak{X}_{\eta,\Gamma}([-1,1]\times \Sigma_{g,n};\mathbf{t}) \arrow[dl, "r_-"'] \arrow[dr, "r_+"] & \\
M_{g,n}(\mathbf{t}) & & M_{g,n}(\mathbf{t})
\end{tikzcd}\label{eq:bundlecorrespondence}
\end{equation}
where $r_{\pm}$ sends a perturbed-flat connection to its restriction on $\{\pm 1\}\times \Sigma$. The next proposition identifies the data of the correspondence \eqref{eq:bundlecorrespondence} with the Hamiltonian flow of the function $f_{\eta,\Gamma}$ on $M_{g,n}(\mathbf{t})$ defined in \eqref{eq:tracefunctionetasurface}. The result is a generalization of one due to Herald and Kirk \cite[Thm. 7.3]{herald-kirk}, who considered the case of a single curve $\Gamma=\{\gamma_1\}$. Already in this case, our approach, which uses the computation of $\phi_f^t$ from Proposition \ref{prop:hamcomp}, is more direct. 

\begin{prop}\label{prop:Ham=Hol}
    Let $\Gamma=\{\gamma_1,\ldots,\gamma_k\}$ be a multicurve on $\Sigma_{g,n}$ and $\eta:\R^k\to \R$ be a smooth function. 
    Consider the correspondence \eqref{eq:bundlecorrespondence} induced by a holonomy perturbation in $[-1,1]\times \Sigma_{g,n}$ which is determined by $\Gamma$ and $\eta$.
    Then the maps $r_+$ and $r_-$ are diffeomorphisms and
    \[
        r_{+}\circ (r_-)^{-1} = \phi^1_{f}
    \]
    where $\phi^1_{f}$ is the time $1$ Hamiltonian diffeomorphism of $M_{g,n}(\mathbf{t})$ induced by $f=f_{\eta,\Gamma}$.
\end{prop}

\begin{proof}
The maps $r_\pm$ are induced by pre-composing a representation $\rho$ with the inclusion-induced homomorphism $\pi_1(\{\pm 1\}\times (\Sigma\smallsetminus \mathcal{P}))\to \pi_1([-1,1]\times (\Sigma\smallsetminus \mathcal{P})\smallsetminus \{0\}\times \Gamma)$. The space $[-1,1]\times (\Sigma\smallsetminus \mathcal{P})\smallsetminus \{0\}\times \Gamma$ is homotopy equivalent to the result of taking $\{-1\}\times (\Sigma\smallsetminus \mathcal{P})$ and attaching for each $1\leq j \leq k$ a $2$-torus $T_j= S^1\times S^1$ by gluing $\text{pt}\times S^1$ in $T_j$ to $\gamma_j$. The longitude $\lambda_j$ of $\gamma_j$ corresponds to $\text{pt}\times S^1$ in $T_j$ while $S^1\times \text{pt}$ in $T_j$ corresponds to the meridian $\mu_j$.

Let $[\rho]\in M_{g,n}(\mathbf{t})$. Proposition \ref{prop:perturbedcriticalset} implies that the elements in the fiber 
\[
  r_-^{-1}([\rho]) \subset\mathfrak{X}_{\eta,\Gamma}([-1,1]\times \Sigma_{g,n};\mathbf{t})
\]  
correspond to certain extensions of $\rho:\pi_1(\{\pm 1\}\times (\Sigma\smallsetminus \mathcal{P}))\to SU(2)$ to homomorphisms 
\[
    \pi_1([-1,1]\times (\Sigma\smallsetminus \mathcal{P})\smallsetminus \{0\}\times \Gamma)\to SU(2)
\]
Such an extension is determined by its values on the loops $S^1\times \textup{pt}\subset T_j$ (appropriately tied to the basepoint) which correspond to the meridians $\mu_j$. Condition \eqref{eq:perturbedcriticalsetconditions} which defines $\mathfrak{X}_{\eta,\Gamma}([-1,1]\times \Sigma_{g,n};\mathbf{t})$ uniquely specifies these values in terms of the values of $\rho$ on the curves $\gamma_j$. Thus $r_-$ is a bijection, hence a diffeomorphism. The case for $r_+$ is similar.

It remains to show that $r_+\circ (r_-)^{-1} = \phi_f^1$. Proposition \ref{prop:perturbedcriticalset} implies that $r_+\circ (r_-)^{-1}$ is independent of the choice of the extensions of $\gamma_i$ and the $2$-form $\nu$. We make the following choices that are convenient for our purposes. First we extend the embedding $\gamma_i:S^1\to \Sigma$ to an embedding of $S^1\times (-1,1)$ as in the previous section and then thicken this embedding into an embedding of $S^1\times (-1,1)\times (-\frac{1}{2},\frac{1}{2})$ into $\Sigma\times [-1,1]$ using the inclusion of $(-\frac{1}{2},\frac{1}{2})$ in $[-1,1]$. Fix an increasing function $g:\mathbb R \to \mathbb R$ that $g(t)=0$ for $t\leq -\frac{1}{3}$, $g(t)=1$ for $t\geq \frac{1}{3}$, and define the 2-form $\nu=g'(t)\beta\wedge dt$ where $\beta$ is the 2-form on $(-1,1)$ introduced in the previous section pulled back to $D^2\cong (-1,1)\times (-\frac{1}{2},\frac{1}{2})$ and $t$ denotes the standard coordinate on the factor $(-\frac{1}{2},\frac{1}{2})$.

Let $A$ be a flat connection with $[A]\in M_{g,n}(\mathbf{t})$. Then the connection on $[-1,1]\times \Sigma_{g,n}$ given by
\[
  \bA=A-g(t)\sum_{i=1}^k \partial_i \eta \cdot T_{A,i}\beta_i
\]
satisfies \eqref{eq:perturbedflat} for the above choices of $\nu$ and the extensions of $\gamma_i$. In particular, $[\bA]$ represents an element of $\mathfrak{X}_{\eta,\Gamma}([-1,1]\times \Sigma_{g,n};\mathbf{t})$ such that $r_-([\bA])=[A]$ and $r_+([\bA])=\phi^1_{f}([A])$. The latter identity follows from Proposition \ref{prop:hamcomp}.
\end{proof}

The following is a generalization of Theorem \ref{thm:hamiltoniandiffeo} to more general holonomy parameters.
\begin{theorem}\label{thm:hamiltoniandiffeo-general}
	Let $\mathbf{t}$ be a holonomy parameter for $\Sigma_{g,n}$ satisfying the non-integral assumption, and let $\phi:M_{g,n}(\mathbf{t})\to M_{g,n}(\mathbf{t})$ be a Hamiltonian diffeomorphism. Then there is holonomy perturbation data $\eta,\Gamma$ such that 
	$r_{\pm}$ are diffeomorphisms and $r_+\circ r_-^{-1}$ is a Hamiltonian diffeomorphism which arbitrarily close to $\phi$ in the $C^\infty$-topology.
\end{theorem}
\begin{proof}
	Let $\phi:M_{g,n}(\mathbf{t})\to M_{g,n}(\mathbf{t})$ be a Hamiltonian diffeomorphism of $M_{g,n}(\mathbf{t})$. By Proposition \ref{prop:tracemondense}, there are multicurves $\Gamma_i$ and functions $\eta_i$ with $1\leq i\leq k$
	such that if $\phi_i$ is the time $1$ Hamiltonian flow of $f_{\eta_i,\Gamma_j}:M_{g,n}(\mathbf{t})\to \mathbb R$, then $\phi$ is arbitrarily close to the composition $\phi_1\circ \dots \circ \phi_k$. (Note that 
	Proposition \ref{prop:tracemondense} asserts that a special choice of functions $\eta_i$ given as scalar multiples of multiplication maps would suffice. But we will not need this stronger version of the proposition.)
	Proposition \ref{prop:Ham=Hol} implies that $\eta_i$, $\Gamma_i$ induce perturbation data for $([-1,1]\times \Sigma_{g,n};\mathbf{t})$, giving
	\begin{equation}
	\begin{tikzcd}[column sep=small, row sep=large]
		& \mathfrak{X}_{\eta_i,\Gamma_i}([-1,1]\times \Sigma_{g,n};\mathbf{t}) \arrow[dl, "r^i_-"'] \arrow[dr, "r^i_+"] & \\
		M_{g,n}(\mathbf{t}) & & M_{g,n}(\mathbf{t})
	\end{tikzcd}\label{eq:bundlecorrespondence}
	\end{equation}
	such that $r^i_+\circ (r^i_-)^{-1}=\phi_i$. Now by gluing the perturbation data $(\eta_i,\Gamma_i)$ for $1\leq i\leq k$ as in the discussion at the end of previous subsection, we obtain perturbation data $(\eta_\#,\Gamma_\#)$
	for the result of gluing $k$ copies of $[-1,1]\times \Sigma_{g,n}$, which may be identified with $[-1,1]\times \Sigma_{g,n}$ again after reparametrization. For the associated correspondence 
	\begin{equation}
	\begin{tikzcd}[column sep=small, row sep=large]
		& \mathfrak{X}_{\eta_\#,\Gamma_\#}([-1,1]\times \Sigma_{g,n};\mathbf{t}) \arrow[dl, "r^\#_-"'] \arrow[dr, "r^\#_+"] & \\
		M_{g,n}(\mathbf{t}) & & M_{g,n}(\mathbf{t})
	\end{tikzcd}\label{eq:bundlecorrespondence-compose}
	\end{equation}	
	we have $r^\#_+\circ (r^\#_-)^{-1}=\phi_1\circ \dots \circ \phi_k$. This completes the proof of the claim.
\end{proof}

\subsection{Perturbed 3-manifold Lagrangians}\label{subsec:perturbed-3man-Lag}

We next return to the case of general $(M,T;\mathbf{t})$ satisfying the non-integral assumption. Proposition~\ref{unperturbed-3-man-Lag} explains how we can associate a Lagrangian to $(M,T;\mathbf{t})$ under a restrictive assumption on fundamental groups, and our goal in this part is to use holonomy perturbations to obtain 3-manifold Lagrangians for more general $(M,T;\mathbf{t})$. 

First we review some terminology, mostly following \cite{hayashi-shimokawa}. Let $F$ be a closed orientable surface. A {\emph{compression body}} $C$ is an orientable connected $3$-manifold obtained from a ball $B$ or $F\times [0,1]$ by attaching a finite number of $1$-handles to $\partial B$ or $F\times \{1\}$. Denote by $\partial_- C= F\times \{0\}$ and $\partial_+ C=\partial C\smallsetminus \partial_- C$. If $C$ is obtained from a ball, $\partial_- C$ is empty, and $C$ is a handlebody. A tangle $T=t_1\cup\cdots \cup t_n\subset C$ is {\emph{trivial}} if each $t_i$ is an arc and satisfies one of the following:
\begin{itemize}
\item $t_i$ is {\emph{$\partial_+$-parallel}}: there is an embedded disk $D\subset C$ with $t_i\subset \partial D$ and $D\cap \partial C = \overline{\partial D\smallsetminus t_i} \subset \partial_+ C$, and such that $D\cap t_j=\emptyset $ for $j\neq i$;
\item $t_i$ is {\emph{vertical}}: there is a homeomorphism from $C$ to $F\times [0,1]$ with $1$-handles attached along $F\times \{1\}$, where $\partial_-C$ is sent to $F\times \{0\}$, and $t_i$ is sent to $\{p\}\times [0,1]$ for some $p\in F$.
\end{itemize}
We will sometimes use the notation $\partial_\pm (C,T) = (\partial_\pm C, \partial_\pm C\cap T)$. For an example of a compression body with trivial tangle, see the left picture in Figure \ref{fig:standardpieces}. Note that the right picture in Figure \ref{fig:standardpieces} depicts a compression body, but the tangle is not trivial.

Consider a pair $(M,T)$ where $M$ is a compact connected orientable $3$-manifold and $T$ is a tangle. A {\emph{Heegaard splitting}} of $(M,T)$ is the data of an embedded surface $\Sigma\subset M$ that is transverse to $T$, with $\mathcal{P}:=\Sigma\cap T$ a finite set of points, and such that $(\Sigma,\mathcal{P})$ divides $(M,T)$ into the union
\begin{equation}\label{eq:heegaardsplitting}
    (M,T) = (C_1,T_1) \cup_{(\Sigma,\mathcal{P})} (C_2,T_2)
\end{equation}
where $(C_i,T_i)$ is a compression body with trivial tangle and $\partial_+ C_1 = \partial_+ C_2=\Sigma$. A holonomy parameter for the tangle $T$ is called {\emph{strongly non-integral}} with respect to the Heegaard splitting \eqref{eq:heegaardsplitting} if the induced holonomy parameters for $(C_1,T_1)$ and $(C_2,T_2)$ satisfy the non-integral assumption. By \cite[Lemma 2.1]{hayashi-shimokawa}, every $3$-manifold with tangle $(M,T)$ admits a Heegaard splitting.

\begin{theorem}\label{3-man-Lag}
	Let $\mathbf{t}$ be a holonomy parameter for a pair $(M,T)$ that is strongly non-integral with respect to some Heegaard splitting of $(M,T)$.
	Then there is a choice of perturbation data $\eta$, $\Gamma$ for $(M,T;\mathbf{t})$ such that $\mathfrak{X}_{\eta,\Gamma}(M,T;\mathbf{t})$ is a smooth 
	manifold and the restriction map to the boundary $r:\mathfrak{X}_{\eta,\Gamma}(M,T;\mathbf{t}) \to M_{g,n}(\mathbf{t})$ defines an immersed Lagrangian. 
\end{theorem}
\begin{proof}
	Fixing a Heegaard splitting of $(M,T)$ as in \eqref{eq:heegaardsplitting}, we obtain a decomposition of $ (M,T) $ as 
	\[
	  (C_1,T_1) \cup_{(\Sigma,\mathcal{P})} [-1,1]\times (\Sigma,\mathcal{P}) \cup_{(\Sigma,\mathcal{P})} (C_2,T_2)
	\]
	Let $\mathbf{t}_1$, $\mathbf{t}_2$, $\mathbf{s}$ be the induced holonomy parameters on $(C_1,T_1)$, $(C_2,T_2)$, $[-1,1]\times (\Sigma,\mathcal{P})$. 
	Let $g$ be the genus of $\Sigma$, $n$ be the size of $\mathcal P$, and with a slight abuse of notation, we write $\mathbf{s}$ for the induced holonomy parameter on $(\Sigma,\mathcal{P})$. 
	Let also $g_i$, $n_i$ and $\mathbf{s}_i$ respectively denote the genus of $\partial_-C_i$, the size of $\partial_-T_i$ and the induced holonomy parameter on $\partial_-(C_i,T_i)$. 
	Fix a Hamiltonian isotopy $\phi_0$ of $M_{g,n}(\mathbf{s})$, and let $\eta$, $\Gamma$ be as in Theorem \ref{thm:hamiltoniandiffeo-general} such that for the induced correspondence 
	\begin{equation}
		\begin{tikzcd}[column sep=small, row sep=large]
		& \mathfrak{X}_{\eta,\Gamma}([-1,1]\times \Sigma_{g,n};\mathbf{s}) \arrow[dl, "r_-"'] \arrow[dr, "r_+"] & \\
		M_{g,n}(\mathbf{s}) & & M_{g,n}(\mathbf{s})
		\end{tikzcd}\label{eq:bundlecorrespondence}
	\end{equation}
	we have $\mathfrak{X}_{\eta,\Gamma}([-1,1]\times \Sigma_{g,n};\mathbf{s})=M_{g,n}(\mathbf{s})$ and $\phi:=r_{+}\circ (r_-)^{-1}$ is a Hamiltonian isotopy close to $\phi_0$.
	Then $(\Gamma,\eta)$ determines a choice of holonomy perturbation data for $(M,T)$, which is still denoted by $(\Gamma,\eta)$, and using \eqref{corr-relation}, we have
	\begin{align}
	  \mathfrak{X}_{\eta,\Gamma}(M,T;\mathbf{t})&= \mathfrak{X}(C_1,T_1;\mathbf{t}_1) \times_{M_{g,n}(\mathbf{s})} M_{g,n}(\mathbf{s}) \times_{M_{g,n}(\mathbf{s})}  \mathfrak{X}(C_2,T_2;\mathbf{t}_2)\nonumber\\[2mm]
	  &= \mathfrak{X}(C_1,T_1;\mathbf{t}_1) \times_{M_{g,n}(\mathbf{s})}  \mathfrak{X}(C_2,T_2;\mathbf{t}_2)\label{Lag-fiber-prod}
	\end{align}
	Denoting the restriction maps associated to $\mathfrak{X}(C_i,T_i;\mathbf{t}_i)$ by
	\begin{equation}\label{rest-map-compressionbody}
	  r_i=(r_-^i,r_+^i):\mathfrak{X}(C_i,T_i;\mathbf{t}_i)\to M_{g_{i},n_{i}}(\mathbf{s}_{i})\times M_{g,n}(\mathbf{s}),
	\end{equation}
	the fiber product in \eqref{Lag-fiber-prod} is taken with respect to the maps 
	\begin{equation}\label{fib-prod-maps}
	  \phi\circ r_+^1:\mathfrak{X}(C_1,T_1;\mathbf{t}_1)\to M_{g,n}(\mathbf{s}),\hspace{1cm} r_+^2:\mathfrak{X}(C_2,T_2;\mathbf{t}_2)\to M_{g,n}(\mathbf{s}).
	\end{equation}  
	From the assumptions, the inclusion of $\Sigma\smallsetminus \mathcal P$ into $C_i\smallsetminus T_i$ induces a surjection of fundamental groups. In particular, Proposition \ref{unperturbed-3-man-Lag} implies that 
	the map $r_i$ in \eqref{rest-map-compressionbody} determines a Lagrangian. 
	Since $\phi$ is $C^\infty$ close to $\phi_0$, which may be taken to be an arbitrary Hamiltonian isotopy, we may arrange for $\eta$, $\Gamma$ such that the maps in \eqref{fib-prod-maps} are transversal. 
	In particular, the fiber product in \eqref{Lag-fiber-prod} is a smooth manifold, and the map from $\mathfrak{X}_{\eta,\Gamma}(M,T;\mathbf{t})$ to $M_{{g_1},n_{1}}(\mathbf{s}_{1})\times M_{g_{2},n_{2}}(\mathbf{s}_{2})$ 
	induced by $(r_-^1,r_-^2)$ defines an immersed Lagrangian. 
\end{proof}

Theorem \ref{3-man-Lag} holds more generally without the assumption that the holonomy parameter is strongly non-integral with respect to some Heegaard splitting. Furthermore, if $\mathfrak{X}_{\eta,\Gamma}(M,T;\mathbf{t})$ and $\mathfrak{X}_{\eta',\Gamma'}(M,T;\mathbf{t})$ are immersed Lagrangians associated to two choices of holonomy parameters for $(M,T;\mathbf{t})$, then these Lagrangians are related by an immersed Lagrangian cobordism. Since we do not need these more general claims, we will not prove them here and refer the reader to \cite{herald,dfl} for similar results in the case that $T$ is empty.  

\subsubsection{The case of webs}

The above proposition may be generalized and improved in several directions. First we may allow $T$ to be an embedded graph in $M$ that has trivalent vertices in the interior of $M$ and outside a small neighborhood of trivalent vertices is a properly embedded $1$-manifold. This is an extension of the notion of {\it webs} in \cite{km-tait} to the case of 3-manifolds with boundary, and we use the same name web here. A holonomy parameter $\mathbf{t}$ for a web $T$ is an assignment of a real number $t_i\in [0,1]$ to each embedded edge $T_i$ of $T$ (a connected component of $T$ without a trivalent vertex is also regarded as an edge). In the same way as tangles, $\mathfrak{X}(M,T;\mathbf{t})$ is the space of all conjugacy classes of $SU(2)$-representations of the complement of $T$ such that a meridian of an edge $T_i$ of $T$ is mapped to the conjugacy class $\Gamma(t_i)$. We say $\mathbf{t}$ satisfies the non-integral assumption if $\partial(M,T;\mathbf{t})=(\Sigma_+,\mathcal P_+;\mathbf{t}_{+})\sqcup (\Sigma_-,\mathcal P_-;\mathbf{t}_{-})$ and $\mathbf{t}_{\pm}$ satisfy the non-integral assumption in the previous sense. This assumption implies that $\mathfrak{X}(M,T;\mathbf{t})$ consists of irreducible elements.

\begin{figure}[t]
  \centering
\begin{tikzpicture}
  \node[anchor=south west, inner sep=0] (image) at (0,0) {\includegraphics[scale=1.25]{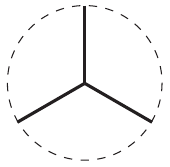}};
  \begin{scope}[x={(image.south east)}, y={(image.north west)}]
    \node at (0.01,0.13) {\large $t_j$};
    \node at (1.0,0.13) {\large $t_k$};
    \node at (0.5,1.1) {\large $t_i$};
  \end{scope}
\end{tikzpicture}
\caption{ }
\label{fig:trivalent}
\end{figure}

Removing small neighborhoods of trivalent vertices turns $(M,T)$ into a pair $(M',T')$ of a 3-manifold and a tangle and a holonomy parameter $\mathbf{t}$ for $(M,T)$ determines a holonomy parameter $\mathbf{t}'$ for $(M',T')$. The pair $(M',T')$ has a boundary component corresponding to each trivalent vertex of $T$ that is $\Sigma_{0,3}$ with the three marked points labeled by three components $t_i$, $t_j$ and $t_k$ of $\mathbf{t}$. See Figure \ref{fig:trivalent}. The character variety $M_{0,3}(t_i,t_j,t_k)$ has exactly one element if the inequalities 
\begin{equation}\label{inequlaties-S2-3markedpts}
	t_i\leq t_j+t_k ,\hspace{1cm}t_j\leq t_i+t_k,\hspace{1cm}t_k\leq t_i+t_j,\hspace{1cm} t_i+t_j+t_k\leq 2
\end{equation}
are satisfied, and otherwise $M_{0,3}(t_i,t_j,t_k)$ is empty. Furthermore, if $M_{0,3}(t_i,t_j,t_k)$ is non-empty, the unique element of this moduli space is irreducible if and only if all the inequalities in \eqref{inequlaties-S2-3markedpts} are strict. The character variety $\mathfrak{X}(M',T';\mathbf{t}')$ can be identified with that of $\mathfrak{X}(M,T;\mathbf{t})$. We may also follow similar proofs as above to obtain analogues of Proposition \ref{unperturbed-3-man-Lag} and Theorem \ref{3-man-Lag}.

\subsubsection{The monotone case} \label{subsec:monotonecaselag}

For the rest of the paper, we focus on $(M,T;\mathbf{t})$ with boundary $(\Sigma_+,\mathcal P_+,\mathbf{t}_{+})\sqcup (\Sigma_-,\mathcal P_-,\mathbf{t}_{-})$ where $\Sigma_{\pm}$ are connected, all entries of $\mathbf{t}$ are $1/2$ or $1$, and the boundary holonomy parameters $\mathbf{t}_{\pm}$ consist entirely of entries equal to $1/2$. Because of this, we reset our notation from now on, and use $T$ to denote the components of the tangle with label $1/2$ and $w$ to denote the components of the tangle with label $1$. Fix a Heegaard splitting of $(M,T)$ as in \eqref{eq:heegaardsplitting}, we now write
\[
    (M,T\cup w) = (C_1,T_1\cup w_1) \cup_{(\Sigma,\mathcal{P})} (C_2,T_2\cup w_2)
\] 
where as before the $1$-manifold $w$ is disjoint from $\Sigma$, and $\mathcal P_\pm\subset \Sigma_\pm$ has an odd number of points, each labeled with $1/2$. In particular, $\mathbf{t}$ is strongly non-integral with respect to any Heegaard splitting of $(M,T)$, and we may apply Theorem \ref{3-man-Lag} to obtain $\eta$, $\Gamma$ such that 
\[
	\mathfrak{X}_{\eta,\Gamma}(M,T, w)=\mathfrak{X}_{\eta,\Gamma}(M,T\cup w;\mathbf{t})
\]
is an immersed Lagrangian of $M_{g_-,n_-}\times M_{g_+,n_+}$. As we explained above, we may assume that $T\cup w$ has trivalent vertices. For $\mathfrak{X}_{\eta,\Gamma}(M,T, w)$ to be non-empty, we need to guarantee that \eqref{inequlaties-S2-3markedpts} is satisfied at any trivalent vertex. This implies that all the edges incident to a trivalent vertex of $T$ are all labeled by $1/2$ or exactly one of them is labeled by $1$. In particular, $w$ is an embedded $1$-manifold in $M$ with boundary satisfying $\partial w=w\cap T$.

As a special example of this construction, let $M=[-1,1]\times \Sigma_g$. Fix a set of $n$ marked points with $n$ odd, and a $1$-dimensional submanifold $Z$ of $\Sigma_g$ with $\partial Z=Z\cap \mathcal P$. Let $T=[-1,1]\times \mathcal P$ and $w=\{0\}\times Z$. Label the edges in $T$ with $1/2$ and the edges in $\{0\}\times Z$ with $1$. If 
\begin{equation}
\begin{tikzcd}[column sep=small, row sep=large]
& \mathfrak{X}(M,T,w) \arrow[dl, "r_-"'] \arrow[dr, "r_+"] & \\
M_{g,n} & & M_{g,n}
\end{tikzcd}\label{eq:bundlecorrespondence-Z}
\end{equation}
is the associated correspondence, then it is easy to see that $r_-$ and $r_+$ are diffeomorphisms and $r_+\circ (r_-)^{-1}$ is the symplectomorphism from the previous section that is associated to the element $\zeta\in H^1(\Sigma^\circ;\mathbb Z/2)$ given by the Poincar\'e dual of $Z$.

%!TEX root = main.tex

\section{Singular instanton homology and 3-manifold Lagrangians}\label{sec:inequality}

In this section we review the necessary aspects of singular instanton Floer homology, and prove Theorem \ref{thm:lagrangianineq} from the introduction; see Theorem \ref{thm:lagrangianinequalityresultgeneral}.

Let $(Y,L,w)$ be an admissible link as defined in the introduction. As mentioned there, singular instanton Floer homology of $(Y,L,w)$ is a $\Z/4$-graded abelian group $I(Y,L,w)$. This instanton homology group can be regarded as the Morse homology of a Chern--Simons functional $CS$ on a certain configuration space $\mathcal B(Y,L,w)$ of isomorphism classes of $PU(2)$-connections $\bA$ on $Y$ that are singular along $L$ and the topological type of the $PU(2)$-bundle carrying $\bA$ is determined by $w$. In particular, $I(Y,L,w)$ is the homology of a chain complex $(C(Y,L,w),d)$:
\[
	I(Y,L,w) = H_\ast( C(Y,L,w), d).
\]
The definition of $C(Y,L,w)$ uses similar (perturbed) character varieties as in the previous section associated to $(Y,L,w)$. The differential $d$ is defined in terms of solutions to a certain version of the anti-self-duality equation. As such, the analytical aspects of the construction of $I(Y,L,w)$ are more subtle than the construction of perturbed character varieties. There are several approaches to address this, all of which yield isomorphic instanton homologies $I(Y,L,w)$. Our exposition here follows \cite{km-unknot} and we refer the reader to this reference for more details. 

To set up the theory as in \cite{km-unknot}, first fix an orbifold metric on $Y$ which is smooth away from $L$ and has orbifold angle $\pi$ along $L$. We write $\check Y$ for the orbifold 3-manifold determined by the pair $(Y,L)$, which has $L$ as the locus of singular points with isotropy groups $\Z/2$. As explained in \cite{km-unknot}, $w$ defines an orbifold $SO(3)$-bundle $\check P$ on $\check Y$, and we may form $\mathcal A(Y,L,w)$ as the space of orbifold connections on $\check P$. The holonomy of any connection $\bA$ in $\mathcal A(Y,L,w)$ along meridians of $L$ are asymptotic to elements of order $2$ in $SO(3)$. The space of connections $\mathcal A(Y,L,w)$ is an affine space modeled on the Sobolev space $L^2_k(Y;\Lambda^1\otimes {\rm ad}(\check P))$ where $k\geq 3$ is a fixed integer and ${\rm ad}(\check P)$ is the orbifold vector bundle on $\check Y$ associated to $\check P$, corresponding to the adjoint action of $SO(3)$ on $\mathfrak{su}(2)$. The space $\mathcal A(Y,L,w)$ is denoted by $\mathcal C_k(Y,L,\check P)$ in \cite{km-unknot}.

There is a {\it determinant-$1$} gauge group $\mathcal G(Y,L,w)$ acting on $\mathcal A(Y,L,w)$, and the configuration space $\mathcal B(Y,L,w)$ is the quotient $\mathcal A(Y,L,w)/\mathcal G(Y,L,w)$. The choice of the gauge group allows us to associate to any closed loop $\gamma$ in $Y\smallsetminus (L\cup w)$ and any $[\bA]\in \mathcal A(Y,L,w)$ a well-defined conjugacy class in $SU(2)$, which is obtained as the conjugacy class of the holonomy of $\bA$ along $\gamma$ after lifting to $SU(2)$. If $\gamma$ is a meridian of $L$ (resp. $w$), then the holonomy of $[A]$ along $\gamma$ is asymptotic to a traceless element of $SU(2)$ (resp. $-1\in SU(2)$) as the size of the meridian goes to zero. The Lie algebra of the gauge group $\mathcal G(Y,L,w)$ may be identified with $L^2_{k+1}(Y;{\rm ad}(\check P))$, and the linearization of the action of the gauge group at $\bA\in \mathcal A(Y,L,w)$ is given by
\begin{equation}\label{d-A}
	d_{\bA}:L^2_{k+1}(Y;{\rm ad}(\check P)) \to L^2_{k}(Y;\Lambda^1\otimes {\rm ad}(\check P))
\end{equation}

The Chern--Simons functional is a functional $\mathcal A(Y,L,w)\to \mathbb R$, characterized up to a constant by the property that its formal gradient at any connection $\bA\in \mathcal A(Y,L,w)$ is equal to $\star F_\bA$ where the metric on $\mathcal A(Y,L,w)$ is induced by the $L^2$ metric on sections of $\Lambda^1\otimes {\rm ad}(\check P)$. This induces an $\mathbb R/\mathbb Z$-valued map on $\mathcal B(Y,L,w)$, which is denoted by $CS$. The set of critical points  $\mathfrak{X}(Y,L,w)$ of $CS$ are given by the gauge equivalence classes of flat connections in $\mathcal A(Y,L,w)$, and by taking holonomies, this set can be identified with the set of conjugacy classes of homomorphisms 
\begin{equation}\label{rho-char-variety}
  \rho:\pi_1(Y\smallsetminus (L\cup w)) \to SU(2) 
\end{equation}
which map any meridian of $L$ to a traceless element of $SU(2)$ and any meridian of $w$ to $-1$. 

There is an action of $H^1(Y\smallsetminus L;\Z/2)$ on $\mathfrak{X}(Y,L,w)$. We first define this action representation-theoretically. Any $\sigma\in H^1(Y\smallsetminus L;\Z/2)$ induces a homomorphism $\hat \sigma: \pi_1(Y\smallsetminus L) \to \{\pm1\}$, and for any $\rho$ as in \eqref{rho-char-variety} representing an element of $\mathfrak{X}(Y,L,w)$, we define 
\[
  \sigma\cdot \rho:\pi_1(Y \smallsetminus (L\cup w)) \to SU(2)
\]
where its value at any $\gamma\in \pi_1(Y\smallsetminus (L\cup w))$ is equal to $\hat \sigma(\gamma) \rho(\gamma)$. Alternatively, let $l_\sigma$ be the real line bundle associated to $\sigma$ and $B_\sigma$ the connection on $l_\sigma$ that has holonomies in $\{\pm1\}$. For any flat connection $\bA\in \mathcal A(Y,L,w)$ representing an element of $\mathfrak{X}(Y,L,w)$, the connection $\bA\otimes B_\sigma$ is another flat connection in $ \mathcal A(Y,L,w)$. This construction respects the action of $\mathcal G(Y,L,w)$ and induces the above action of $\sigma$ on $\mathfrak{X}(Y,L,w)$. This description shows that this action extends to $\mathcal A(Y,L,w)$, and preserves $CS$ up to a constant.

The critical points of $CS$ are not necessarily non-degenerate, and for this reason we may need to perturb $CS$. As in \S \ref{hol-pert}, fix $\Gamma = \{\gamma_1,\ldots,\gamma_k\}$, where each $\gamma_i$ is an embedded loop in $Y\smallsetminus (L\cup w)$, and a choice of smooth function $\eta:\R^k\to \R$. Thicken each loop $\gamma_i$ to an embedding of $S^1\times D^2$ into $Y\smallsetminus (L\cup w)$. For any $\bA\in \mathcal A(Y,L,w)$, define 
\[
    f_{\eta,\Gamma}(\bA)  = \eta( \sigma_{\gamma_1}(\bA), \ldots, \sigma_{\gamma_k}(\bA)).
\]
where $\sigma_{\gamma_i}(\bA)$ is defined as in \eqref{holonomy-map-cylinder}. Note that the restriction of $\bA$ to the solid torus determined by $\gamma_i$ is an ordinary connection. A connection $\bA\in \mathcal B(Y,L,w)$ represents a critical point of $CS+f_{\eta,\Gamma}$ exactly when it satisfies the equation
  \begin{equation}\label{eq:perturbedflat-closed}
    F_{\bA} = \sum_{i=1}^k \partial_i\eta(\text{tr}_{\gamma_1}({\bA}),\ldots,\text{tr}_{\gamma_k}({\bA})) T_{{\bA},i} \nu_i,
\end{equation}
which is the same as \eqref{eq:perturbedflat}, except now we are working in a closed $3$-manifold. Write $\mathfrak{X}_{\eta,\Gamma}(Y,L,w)$ for the gauge equivalence classes of such connections. 
For any admissible link $(Y,L,w)$, there is a choice of perturbation data $\eta$, $\Gamma$ such that the critical points of $CS+f_{\eta,\Gamma}$ are non-degenerate \cite{km-unknot,yaft,donaldson-book}. In Lemma \ref{non-dege-pert} below, we establish this non-degeneracy using perturbations adapted to a Heegaard splitting of $(Y,L,w)$.

For $i\in \{1,2\}$, let $(M_i,T_i,w_i)$ be a triple of a compact, oriented, connected $3$-manifold $M_i$ with boundary, a tangle $T_i\subset M_i$, and a compact $1$-manifold $w_i\subset \text{int}M_i$ satisfying $w_i\cap T_i = \partial w_i$. We also assume that $\partial (M_i,T_i)$ is the union of $\Sigma_{g,n}$ and $\Sigma_{g',n'}$, where $n$, $n'$ are odd integers. After reversing the orientation of $M_2$, we glue these triples along their common boundaries to obtain 
\begin{equation}\label{admissible-splitting-triple}
  (Y,L,w) = (M_1,T_1,w_1)\cup_{\Sigma_{g,n} \sqcup \Sigma_{g',n'}} (M_2,T_2,w_2)
\end{equation}
Then the embedding of $\Sigma_{g,n}$ into $(Y,L,w)$ as one of the boundary components of $(M_i,T_i,w_i)$ implies that $(Y,L,w)$ is admissible. Using Theorem \ref{3-man-Lag}, there is a choice of perturbation data $\eta_i$, $\Gamma_i$ for $(M_i,T_i,w_i)$ such that $\mathfrak{X}_{\eta_i,\Gamma_i}(M_i,T_i,w_i)$ defines an immersed Lagrangian \[r_i: \mathfrak{X}_{\eta_i,\Gamma_i}(M_i,T_i,w_i) \to M_{g,n}\times M_{g',n'}\] by restriction to the boundary. Note that combining $\Gamma_1$ and $\Gamma_2$ determines a family of loops $\Gamma$ in $Y$, and the functions $\eta_i:\mathbb R^{k_i}\to \mathbb R$ give $\eta:\mathbb R^{k_1+k_2}\to \mathbb R$ defined as \[\eta(x_1,\dots,x_{k_1+k_2})=\eta_1(x_1,\dots,x_{k_1})+\eta_2(x_{k_1+1},\dots,x_{k_1+k_2}).\]
In particular, we may use $\Gamma$ and $\eta$ to perturb the Chern--Simons functional of $(Y,L,w)$ and form the critical point set $\mathfrak{X}_{\eta,\Gamma}(Y,L,w)$.

Restricting any $[\bA]\in\mathfrak{X}_{\eta,\Gamma}(Y,L,w) $ to $M_i$ gives an element $[\bA_i] \in \mathfrak{X}_{\eta_i,\Gamma_i}(M_i,T_i)$ satisfying 
\begin{equation}\label{matching-bdry}
	r_1([\bA_1])=r_2([\bA_2]).
\end{equation} 
In the reverse direction, let $[\bA_i] \in \mathfrak{X}_{\eta_i,\Gamma_i}(M_i,T_i,w_i)$ be given such that \eqref{matching-bdry} holds. Then the connections $\bA_1$ and $\bA_2$ are flat in the neighborhood of the boundary and after applying appropriate gauge transformations, we may assume that each $\bA_i$ is the pullback of flat connections $A$ and $A'$ on tubular neighborhoods of the boundary components $\Sigma_{g,n}$ and $\Sigma_{g',n'}$. Since $A$ and $A'$ are both irreducible, we may glue $\bA_1$ and $\bA_2$ along each of the common boundaries of $(M_1,T_1)$ and $(M_2,T_2)$ by $\pm 1$. Changing both signs at the same time produces gauge equivalent connections on $(Y,L,w)$. But changing only one of the signs leads to a different gauge equivalence class that is related to the original one by the action of the element of $\sigma_0\in H^1(Y\smallsetminus L;\Z/2)$, given by the restriction of the Poincar\'e dual of $\Sigma_{g,n}$ to $Y\smallsetminus L$. Combining these observations, we see that 
\begin{equation}\label{identification-Lag-crit}
  \mathfrak{X}_{\eta,\Gamma}(Y,L,w)/\sigma_0 \cong \mathfrak{X}_{\eta_1,\Gamma_1}(M_1,T_1,w_1) \mathop{\times}\limits_{r_1,r_2}\mathfrak{X}_{\eta_2,\Gamma_2}(M_2,T_2,w_2).
\end{equation}
We remark that the involution induced by $\sigma_0$ acts freely on $\mathfrak{X}_{\eta,\Gamma}(Y,L,w)$. In fact, it is easy to see that if $[\bA]\in \mathcal B(Y,L,w)$ is fixed by the action of $\sigma_0$, then $[\bA]$ has a reduction to an $O(2)$-connection where the associated orientation bundle of the $O(2)$-bundle carrying $\bA$ is $l_{\sigma_0}$ (see \cite[Lemma 2.44]{dfl} for the proof of a similar fact). In particular, the restriction of $\bA$ to $\Sigma_{g,n}$ and $\Sigma_{g',n'}$ has a reduction to an $SO(2)$-connection. This implies that the restriction of $\bA$ to these embedded surfaces cannot be flat because all elements of $M_{g,n}$ and $M_{g',n'}$ are irreducible and they do not have $SO(2)$-reductions. Thus, any such $[\bA]$ is not an element of $\mathfrak{X}_{\eta,\Gamma}(Y,L,w)$.

\begin{lemma}\label{non-dege-pert}
	Using the above notation, the sets of loops $\Gamma_1$, $\Gamma_2$ can be enlarged and the functions $f_1$ and $f_2$ can be extended such that the set of
	critical points $\mathfrak{X}_{\eta,\Gamma}(Y,L,w)$ of the perturbed Chern--Simons functional $CS+f_{\eta,\Gamma}$ is non-degenerate.
\end{lemma}
\begin{proof}
	Throughout this proof, let $X=M_{g,n}\times M_{g',n'}$ and $\Lambda_i=\mathfrak{X}_{\eta_i,\Gamma_i}(M_i,T_i,w_i)$. There is an arbitrarily small Hamiltonian isotopy 
	$\phi:X\to X$ such that the immersed Lagrangians 
	\[\phi \circ r_1:\Lambda_1\to X,\hspace{1cm}r_2:\Lambda_2\to X\]
	are transversal to each other and the intersection points of $\Lambda_1$ and $\Lambda_2$. 
	Theorem \ref{thm:hamiltoniandiffeo-general} asserts that $\phi$ can be approximated by a holonomy perturbation in 
	$[-1,1]\times \Sigma_{g,n}$. In particular, we may assume that after possibly enlarging $\Gamma_1$ using a family of curves in a 
	tubular neighborhood of the boundary of $M_1$ and extending $f_1$, the immersed Lagrangians $\Lambda_1$ and $\Lambda_2$ are transversal to each other. 
	Therefore, the fiber product of $\Lambda_1$, $\Lambda_2$ over $X$ consists of finitely many points, away from the 
	self-intersection points of $\Lambda_1$ and $\Lambda_2$. In light of the identification in \eqref{identification-Lag-crit}, this implies that any critical point of 
	$CS+f_{\eta,\Gamma}$, up the action of $\sigma_0$, is uniquely determined by its restrictions to $M_1$, $M_2$.
	
	We expect that the transversality of $\Lambda_1$, $\Lambda_2$ implies that the perturbed functional $CS+f_{\eta,\Gamma}$ is non-degenerate 
	(see \cite[Lemma 2.48]{dfl} for the proof of a similar claim). However, this would require further development of the analytical setup for 3-manifolds Lagrangians. 
	Instead, we prove something weaker that suffices for our purposes: we show that, after an additional small
	perturbation, $CS+f_{\eta,\Gamma}$ is non-degenerate. The proof follows a similar argument as in \cite{donaldson-book}. 
	
	Fix a critical point of $CS+f_{\eta,\Gamma}$ and let $\bA$ be a connection representing this critical point. Since $\bA$ is flat in a neighborhood of 
	$\Sigma_{g,n}$, $\Sigma_{g',n'}$, we may assume that $\bA$ is the pullback of flat connections $A$, $A'$ in tubular neighborhoods $\nu(\Sigma_{g,n})$, $\nu(\Sigma_{g',n'})$ of 
	these embedded surfaces. There is a finite dimensional subspace $H_{\bA}$ of the tangent space of $\mathcal B(Y,L,w)$ at $[\bA]$ 
	representing the kernel of the Hessian of $CS+f_{\eta,\Gamma}$ at $[\bA]$. The space $H_{\bA}$ can be regarded as a subspace of 
	the orthogonal complement of the image of \eqref{d-A}. In particular, any non-zero element of $H_\bA$ is given by $\zeta \in L^2_{k}(Y;\Lambda^1\otimes {\rm ad}(\check P))$
	that is not of the form $d_A\kappa$ for any $\kappa \in L^2_{k+1}(Y;{\rm ad}(\check P))$.

	Let $\bA_i$, $\zeta_i$ respectively denote the restriction of $\bA$, $\zeta$ to $M_i$. We claim that $\zeta_1$, $\zeta_2$ cannot both be of the form $d_{\bA_i}\kappa_i$
	for some $L^2_{k+1}$ sections $\kappa_i$ of ${\rm ad}(\check P)$ over $M_i$. Otherwise, the restrictions $\lambda_i$, $\lambda_i'$ of $\kappa_i$ 
	over the tubular neighborhoods $\nu(\Sigma_{g,n})$, $\nu(\Sigma_{g',n'})$ satisfy 
	\[
	  d_A\lambda_1=d_A\lambda_2\hspace{1cm}d_{A'}\lambda_1'=d_{A'}\lambda_2'.
	\]  
	Since the connections $A$, $A'$ are irreducible, this implies that $\lambda_1=\lambda_1'$ and $\lambda_2=\lambda_2'$. In particular, $\kappa_1$, $\kappa_2$ agree in 
	the tubular neighborhoods of the boundary and we may glue them to find $\kappa$ with $d_A\kappa=\zeta$, which is a contradiction. 
	Assuming $\zeta_1$ is not of the form $d_{\bA_1}\kappa_1$, there is an embedded loop $c$ in $M_1\smallsetminus (T_1\cup w_1)$ such that the derivative of the map 
	${\rm tr}_c$ along $\zeta$ is non-trivial. This claim is the infinitesimal version of the fact that any $SU(2)$-connection that has the trace of its holonomy along any loop trivial 
	is gauge equivalent to the trivial connection. One can prove this claim by combining the arguments in Proposition \ref{prop:firsttraceapprox} and \cite[Proposition 4.60]{dfl} 
	(see also \cite[Lemma 4.23]{knot-complement-problem-nullhomotopic}). Applying this for a basis of $H_{\bA}$ for each critical point of  $CS+f_{\eta,\Gamma}$ gives a collection 
	of loops $\{c_i\}_{1\leq i\leq N}$ such that the derivative of the map 
	\[
	  ({\rm tr}_{c_1},\ldots, {\rm tr}_{c_N}): \mathcal B(Y,L,w) \to \mathbb R^N
	\]
	restricted to $H_{\bA}$ for each critical point $[\bA]$ of $CS+f_{\eta,\Gamma}$ is injective. 
	 Now the proof can be completed using a similar argument as in \cite[Proposition 5.15]{donaldson-book}.
\end{proof}

Next, let $(Y,L,w)$ be an admissible link and let $\eta$, $\Gamma$ be a choice of the perturbation data such that the critical points $\mathfrak{X}_{\eta,\Gamma}(Y,L,w)$ of $CS+f_{\eta,\Gamma}$ are non-degenerate. Then the chain group $C(Y,L,w)$ used in the definition of $I(Y,L,w)$ can be chosen to be the free abelian group generated by the elements of $\mathfrak{X}_{\eta,\Gamma}(Y,L,w)$. The differential $d$ is defined in terms of the moduli spaces of perturbed orbifold instantons on the bundle $\mathbb R\times \check P$ over $\mathbb R\times \check Y$ where the perturbation is induced by the gradient of $f_{\eta,\Gamma}$. In general, one might need to perturb $CS$ further to guarantee that all such moduli spaces are cut out transversely. This can be arranged without changing the critical set \cite[Proposition 3.18]{yaft}. In fact, we may pick this additional perturbation of $CS+f_{\eta,\Gamma}$ to be arbitrarily small and trivial in a neighborhood of the critical points $ \mathfrak{X}_{\eta,\Gamma}(Y,L,w)$ of $CS+f_{\eta,\Gamma}$. In particular, this additional perturbation does not change $C(Y,L,w)$. 

We are now in a position to prove Theorem \ref{thm:lagrangianineq}. We give a slightly more general result.

\begin{theorem}\label{thm:lagrangianinequalityresultgeneral}
	Suppose $(Y,L,w)$ is an admissible link with a decomposition as in \eqref{admissible-splitting-triple}, and suppose holonomy perturbation data $\eta_i,\Gamma_i$ for $i\in \{1,2\}$ is chosen such that each $\Lambda_i=\mathfrak{X}_{\eta_i,\Gamma_i}(M_i,T_i,w_i)$ defines an immersed Lagrangian in $M_{g,n}\times M_{g',n'}$. Then
		\begin{equation}\label{eq:lagrankineqgeneralproof}
		\mathfrak{n}(\Lambda_1,\Lambda_2) \geq \frac{1}{2}\, \textup{rank}_\Z \; I(Y,L,w).
	\end{equation}
	The same inequality holds using the rank of instanton homology computed with respect to any integral domain coefficient ring. Furthermore, if $ I(Y,L,w)=0$ and the perbtation data $\pi_i$ defining $\Lambda_i$ is sufficiently small for $i\in \{1,2\}$, then $\mathfrak{n}(\Lambda_1,\Lambda_2)=0$.
\end{theorem}

\begin{proof}[Proof of Theorem \ref{thm:lagrangianineq}]
Set $X=M_{g,n}\times M_{g',n'}$. Let $\phi:X\to X$ be a Hamiltonian isotopy such that $\Lambda_1$ is transversal to $\phi(\Lambda_2)$. As we explained in the proof of Theorem \ref{non-dege-pert}, we may use Theorem \ref{thm:hamiltoniandiffeo-general} to modify perturbation data $\eta_2$, $\Gamma_2$ to $\eta_2'$, $\Gamma_2'$ such that $\phi(\Lambda_2)$ is arbitrarily close to $\Lambda_2'=\mathfrak{X}_{\eta_2',\Gamma_2'}(M_2,T_2,w_2)$. In particular, the number of points in $\Lambda_1 \times_X \phi(\Lambda_2) $ and $\Lambda_1 \times_X \Lambda_2' $ are equal to each other. The perturbation data $\eta_i,\Gamma_i$ induce the perturbation data $\eta$, $\Gamma$ for the admissible link $(Y,L,w)$ in \eqref{admissible-splitting-triple}, and by \eqref{identification-Lag-crit}, the number of points in $ \mathfrak{X}_{\eta,\Gamma}(Y,L,w)$ is twice the size of $\Lambda_1 \times_X \Lambda_2' $ (or equivalently $\Lambda_1 \times_X \phi(\Lambda_2)$). Since  $I(Y,L,w)$ is the homology of the chain complex $(C(Y,L,w),d)$ where $C(Y,L,w)$ is generated by $\mathfrak{X}_{\eta,\Gamma}(Y,L,w)$, we conclude that the size of $\Lambda_1 \times_X \phi(\Lambda_2)$ is bounded below by half of the rank of $I(Y,L,w)$. This proves \eqref{eq:lagrankineqgeneralproof}.

To prove the last statement, let $I(Y,L,w)=0$. In the next section, we show that $I(Y,L,w)=0$ implies that $\mathfrak{X}(Y,L,w)$ is empty. That is, the singular instanton complex $C(Y,L,w)$ for the trivial perturbation is trivial. The identification in \eqref{identification-Lag-crit} shows that the fiber product of $\mathfrak{X}(M_1,T_1,w_1)$ and $\mathfrak{X}(M_2,T_2,w_2)$ over $M_{g,n}\times M_{g',n'}$ is trivial. But the spaces $\mathfrak{X}(M_1,T_1,w_1)$ and $\mathfrak{X}(M_2,T_2,w_2)$ are not necessarily cut out transversaly. However, by Theorem \ref{3-man-Lag}, there are arbitrarily small perturbation data $\eta_i$, $\Gamma_i$ such that $\Lambda_i=\mathfrak{X}_{\eta_i,\Gamma_i}(M_i,T_i,w_i)$ are immersed Lagrangians. Since the perturbation data are small, the image of $\Lambda_i$ and $\mathfrak{X}(M_i,T_i,w_i)$ in $M_{g,n}\times M_{g',n'}$ are close to each other and hence the fiber product of $\Lambda_1$ and $\Lambda_2$ is still empty. In particular, we have 
$\mathfrak{n}(\Lambda_1,\Lambda_2)=0$. 
\end{proof}

%!TEX root = main.tex

\section{A detection result in singular instanton homology}\label{sec:detection}

In this section, we characterize when singular instanton homology of an admissible link is non-trivial. The following is an extension of \cite[Proposition 2.1]{dsodd}. The statement and argument also build on non-vanishing results due to Kronheimer and Mrowka \cite[\S 7]{km-tait}.

\vspace{0.1cm}

\begin{theorem}\label{prop:detection}
    Let $(Y,L,w)$ be an admissible link. Then $I(Y,L,w)$ is zero if and only if one of the following three conditions is satisfied:
    \begin{enumerate}
        \item[{\rm{(i)}}] $Y=S^1\times S^2\# Y'$ where  $L\subset Y'$ and $w$ has odd pairing with $\{pt\}\times S^2$.
        \item[{\rm{(ii)}}] $Y=S^1\times S^2\# Y'$ where $L$ is the disjoint union of $S^1\times \{pt\}\subset S^1\times S^2$ and some $L'\subset Y'$.
        \item[{\rm{(iii)}}] $Y=Y_1\# Y_2$ where the separating sphere $S^2$ satisfies $w\cdot S^2$ is odd and $L\cap S^2 = \emptyset$.
    \end{enumerate}
\end{theorem}

\noindent Note (i) and (iii) may be combined into the condition that there exists a 2-sphere in $Y$, disjoint from $L$, which has odd pairing with $w$. In particular, Theorems \ref{prop:detection} and \ref{thm:detectionintro} are equivalent.

In preparation for the proof, we prove several lemmas. Recall that a compact $3$-manifold $M$ is {\emph{irreducible}} if every embedded $2$-sphere bounds a 3-ball, and {\emph{$\partial$-irreducible}} if every component of $\partial M$ is incompressible, i.e. there are no embedded disks $D$ in $M$ with $\partial D$ essential in $\partial M$. Given a link $L\subset Y$ we write $N(L)\subset Y$ for a closed regular neighborhood of $L$, and we say that $(Y,L)$ is irreducible or $\partial$-irreducible if this is the case for $Y\smallsetminus \text{int}N(L)$. For $p\in \Q\cup \{\infty\}$ let $(L(p,q),U_{p/q})$ be the result of doing $p/q$-surgery on a meridional loop around the unknot in $S^3$.

\vspace{0.1cm}

\begin{lemma}\label{lemma:tautconnectedsumdecomp}
    Given a 3-manifold with link $(Y,L)$, we may write $Y=Y_1\# \cdots \# Y_k$ where $L$ is the disjoint union of (possibly empty) links $L_i\subset Y_i$ such that for each $i$ one of the following holds: 
        \begin{itemize}
        \item $(Y_i,L_i)$ is irreducible and $\partial$-irreducible.
        \item $(Y_i,L_i)=(S^1\times S^2,\emptyset)$.
        \item $(Y_i,L_i)=(L(p,q),U_{p/q})$ for some $p/q\in \Q \cup \{\infty\}$.
    \end{itemize}
    The last case includes $(S^3,U)$, when $p/q=\infty$, and $(S^1\times S^2,S^1\times \{pt\})$, when $p/q=0$.
\end{lemma}

\vspace{0.0cm}

\begin{proof}
The prime decomposition of $Y\smallsetminus L$ allows us to write $Y=Y_1\# \cdots \# Y_k$ where $L$ is the disjoint union of $L_i\subset Y_i$ and each $Y_i\smallsetminus L_i$ is either irreducible or $Y_i=S^1\times S^2$ with $L_i=\emptyset$. Suppose $(Y_i,L_i)$ is not $\partial$-irreducible. Then there is a component $L_i'$ of $L_i$ and a disk $D\subset Y_i\smallsetminus \text{int}N(L_i')$ such that $\partial D$ is an essential simple closed curve on $\partial N(L_i')$. The union of a thickening $D\times I$ of $D$ together with $N(L_i')\cong S^1\times D^2$ is homeomorphic to $L(p,q)$ minus an open $3$-ball for some $p/q$. Accounting for $L_i'$ in this lens space gives $U_{p/q}$. By irreducibility of $(Y_i,L_i)$ we must have $(Y_i,L_i)=(L(p,q),U_{p/q})$.
\end{proof}

\vspace{0.1cm}

As Kronheimer and Mrowka's sutured instanton homology \cite{km-sutures} is used below, we recall some of the relevant terminology. A {\emph{balanced sutured manifold}} $(M,\gamma)$ consists of a compact oriented $3$-manifold $M$ with boundary, having no closed components, and an oriented 1-manifold $\gamma\subset \partial M$ whose components are called {\emph{sutures}}. The pair $(M,\gamma)$ is required to satisfy several properties explained in the following paragraph.

There is a disjoint pair of oriented surfaces with boundary $R_+$ and $R_-$ in $\partial M$ such that the complement of $R_+\cup R_-$ in $\partial M$ is the disjoint union of a collection of annular neighborhoods of the sutures. (Strictly speaking, the annular neighborhoods of the sutures are usually included in the data of $(M,\gamma)$.) The surface $R_+$ is distinguished as having the property that the orientation of the sutures parallel to $\partial R_+$ are induced by first using the boundary orientation of $R_+\subset \partial M$ induced by $M$ and again using the boundary orientation of $\partial R_+$ induced by $R_+$. Finally, it is required that $\chi(R_+)=\chi(R_-)$ and that each of $R_\pm$ has no closed components.

A balanced sutured manifold $(M,\gamma)$ is {\emph{taut}} if the $3$-manifold $M$ is irreducible, the surface $R_+\cup R_-$ is incompressible in $M$, and $R_\pm$ are each minimal in their homology classes with respect to the Thurston norm on $H_2(M,\partial M)$.

A {\emph{pairing}} for $(Y,L)$ is the data of a subset $\Pi$ of the set of link components of $L$ such that $|\Pi|$ is even, together with a partition of $\Pi$ into 2-element subsets. Plainly, a pairing matches up some of the components of $L$ into pairs. The {\emph{empty pairing}} is the unique pairing with $\Pi=\emptyset$.

To a link $(Y,L)$ with a pairing, associate a sutured manifold $(M,\gamma)$ as follows: $M$ is the result of taking $Y\smallsetminus \text{int}N(L)$ and gluing up the boundary tori of components from $\Pi$ in pairs according to the pairing; the particular way in which two boundary tori are glued together will not be important. Note $\partial M$ consists of the remaining boundary tori of link components not in $\Pi$. The sutures $\gamma$ consist of two oppositely oriented meridians on each torus boundary component on $\partial M$.

This construction produces a balanced sutured manifold $(M,\gamma)$ so long as $L$ has components not in $\Pi$. Otherwise, the result is a closed $3$-manifold. For convenience, we will abuse terminology and include closed 3-manifolds in the class of balanced sutured manifolds. In this setting, a closed $3$-manifold being taut will mean that it is irreducible.

\vspace{0.1cm}

\begin{lemma}\label{eq:irreducibltotaut}
    Suppose $(Y,L)$ is irreducible and $\partial$-irreducible. Then the sutured manifold $(M,\gamma)$ associated to any pairing of $(Y,L)$ is taut.
\end{lemma}

\begin{proof}
    Suppose $M$ is reducible, and consider a reducing sphere $S\subset \text{int}M$. Let $T\subset \text{int}M$ be the collection of $2$-tori which were glued up in pairs. Assume $S$ intersects $T$ transversely, and the number of components in $S\cap T$ is minimal among reducing spheres $S$. A standard argument using the minimality of $S\cap T$ implies that no component of $S\cap T$ is nullhomotopic in $T$. Suppose $S\cap T\neq \emptyset$. The result of removing a neighborhood of $S\cap T$ from $S$ gives a planar surface $S'$, a union of disks with holes, and $S$ is homeomorphic to the result of gluing up pairs of boundary circles in $S'$. As $S$ is a sphere, an Euler characteristic argument shows that $S'$ contains a disk. This disk gives a compressing disk for $Y\smallsetminus \text{int}N(L)$, contradicting the $\partial$-irreducibility of $(Y,L)$. Thus $S\cap T=\emptyset$, in which case $S$ is a reducing sphere for $Y\smallsetminus \text{int}N(L)$, contradicting the irreducibility of $(Y,L)$. It follows that $M$ is irreducible. 
    
    The surfaces $R_+$ and $R_-$ in $\partial M$ are each collections of annuli, and thus automatically are Thurston-norm minimizing. It remains to show that each annulus component of $R_\pm$ is incompressible in $M$. For this, suppose that there is a compressing disk $D$ for one of these annuli, and repeat the argument above, with $D$ playing the role of $S$. 
\end{proof}

\vspace{0.1cm}

For the next lemma, we utilize the notation of \cite{km-unknot} and write
\[
    I^\#(Y,L,w):=I(Y,L\sqcup H,w\sqcup w_H)
\]
where $H$ is a Hopf link in a small $3$-ball in $Y$ and $w_H$ is an arc connecting the components of $H$. To define $I^\#(Y,L,w)$, we do not require $(Y,L,w)$ to be admissible. Write $\F=\Z/2$.

\vspace{0.1cm}

\begin{lemma}\label{lemma:frameddetection}
    Given $(Y,L,w)$ where $(Y,L)$ is irreducible and $\partial$-irreducible, we have
    \[
        I^\#(Y,L,w;\F) \neq 0.
    \]
\end{lemma}

\begin{proof}
If $L$ is empty, the result follows from \cite[Proposition 2.1]{dsodd}, and so we assume $L\neq \emptyset$. Consider the set of components of $L$ that have an odd number of points from $w$:
\[
    \Pi := \{ L' \subset L \; : \;  L'\text{ is a component of } L \text{ and } \# L' \cap \partial w \text{ is odd} \}.
\]
Note that $|\Pi|$ is even. Choose a pairing of $(Y,L,w)$ with respect to $\Pi$. Note that $w$ is mod $2$ homologous to a $1$-manifold that misses components of $L$ not in $\Pi$ and has one boundary point on each component in $\Pi$. We assume from the start that $w$ has this property. Writing $(M,\gamma)$ for the balanced sutured manifold associated to $(Y,L)$ and the pairing chosen, we also have a $1$-manifold in $M$ which we still call $w\subset M$; it is obtained from $w\subset Y$ by arranging that when gluing the tori of $\Pi$ in pairs, we require that the points of $\partial w$ in each torus is appropriately matched. Consider the sutured instanton homology, defined as a vector space over $\C$, written
    \[
    SHI(M,\gamma,w).
    \]
Here, we slightly extend the sutured instanton homology from \cite{km-sutures} to include an embedded unoriented closed $1$-manifold $w\subset \text{int} M$; for details, see \cite{dis}. As explained there, \cite[Theorem 7.12]{km-sutures} generalizes: if $(M,\gamma)$ is taut, then $SHI(M,\gamma,w)\neq 0$. 

Recall that our conventions allow the possibility that $M$ is a closed $3$-manifold. In this case $SHI(M,\gamma,w)$ is defined to be $I(M,w;\C)$. Note that in our setup, when $M$ is closed, $w$ intersects each glued up torus in an odd number of points, and thus $(M,w)$ is admissible.

    Let $\pi$ be a collection of points on $L$ and denote by $L^\pi \subset Y$ the link obtained from $L$ by `adding an earring' around each point in $\pi$. Explicitly, at each point in $\pi$ we connect sum onto $L$ a copy of the Hopf link. This is done away from the $1$-manifold $w\subset Y$. Let $w^\pi\subset Y$ be the union of $w$ and one small arc connecting the two components of each Hopf link attached. Now assume that $\pi$ consists of exactly one point on each component of $L$ which is not in $\Pi$. Applying the excision argument of \cite[Proposition 5.7]{km-unknot} (see also \cite[Proposition 5.1]{xie}), there is an isomorphism 
    \[
        SHI(M,\gamma,w) \cong I(Y,L^\pi,w^\pi;\C).
    \]
    More precisely, Floer's excision is performed on $(Y,L^{\pi},w^\pi)$ in two ways: once for each component $L'\subset L$ not in $\Pi$, using a pair of tori, one surrounding $L'\subset L^\pi$ and the other torus encasing the earring; and once also for each pair of components $L',L''\subset L^\pi$ in $\Pi$ that are matched by the pairing, using a pair of tori surrounding each of $L'$ and $L''$. It is important that each excision involves a pair of tori that have odd intersection with $w^\pi$.

    Under the assumption that $(Y,L)$ is irreducible and $\partial$-irreducible, Lemma \ref{eq:irreducibltotaut} implies that $(M,\gamma)$ is taut. Then $I(Y,L^\pi,w^\pi;\C)\neq 0$, and by the universal coefficients theorem,
    \[
        I(Y,L^\pi,w^\pi;\F) \neq 0
    \]
    where $\F =\Z/2$. Next, for any admissible link $(Y_0,L_0,w_0)$ we have
    \begin{equation}\label{eq:dimdoublesharptonatural}
        \dim_\F I(Y_0,L_0\sqcup H,w_0\sqcup w_H;\F) = 2 \dim_\F I(Y_0,L^{\{p\}}_0,w^{\{p\}}_0;\F),
    \end{equation}
    where $p$ is any point on $L$, and, as before, $H$ is a Hopf link in a $3$-ball and $w_H$ is a small arc between the components of $H$. See \cite[Lemma 7.7]{km-tait}, \cite[Corollary 8.17]{ds1}. Note that the latter reference has \eqref{eq:dimdoublesharptonatural} stated for $(Y,K)$ where $K$ is a knot in an integer homology $3$-sphere, but the argument adapts to the case of admissible links. Thus we now have
    \begin{equation}\label{eq:detectionnonvanishstepproof}
    I(Y,L^{\pi\smallsetminus \{p\}}\sqcup H,w^{\pi\smallsetminus \{p\}}\sqcup w_H;\F) \neq 0
    \end{equation}
    where $p\in \pi$ is arbitrary. The next property to be used is the following: if $(Y_0,L_0,w_0)$ is any admissible link and $\sigma\neq\emptyset$ a collection of points on $L_0$, then for any $s\in \sigma$,
    \begin{equation}\label{eq:detectionpropertyproof}
        2\dim_\F I(Y_0,L^{\sigma\smallsetminus \{s\}}_0,w^{\sigma\smallsetminus \{s\}}_0;\F) \geq \dim_\F I(Y_0,L_0^{\sigma},w_0^{\sigma};\F).
    \end{equation}
    This follows from \cite[\S 8.2]{ds1} adapted to the case of admissible links, where it is shown that a chain complex for $I(Y_0,L_0^{\sigma},w_0^{\sigma})$ is given by the mapping cone of a chain map on a chain complex for the triple associated to $\sigma\smallsetminus \{s\}$. By applying property \eqref{eq:detectionpropertyproof} successively to $(Y,L^{\pi\smallsetminus \{p\}}\sqcup H,w^{\pi\smallsetminus \{p\}}\sqcup w_H)$ for each remaining point in $\pi\smallsetminus \{p\}$, we finally obtain from \eqref{eq:detectionnonvanishstepproof} that
    \begin{equation*}
         I^\#(Y,L,w) = I(Y,L\sqcup H,w\sqcup w_H;\F)\neq 0. \qedhere
    \end{equation*}
\end{proof}

\vspace{0.2cm}

\begin{proof}[Proof of Theorem \ref{prop:detection}]

        Let $(Y,L,w)$ be any admissible link. The vanishing of $I(Y,L,w)$ in cases (i)--(iii) is standard. If condition (i) or (iii) is satisfied, then there are no critical points of the associated (unperturbed) Chern--Simons functional. Indeed, any such critical point restricts to the relevant $2$-sphere described in (i) or (ii) to a point in $M^\text{odd}_{0}=\emptyset$. Similary, a critical point in case (ii) restricts to the relevant $2$-sphere to a point in $M_{0,1}=\emptyset$.
        
        Now suppose $(Y,L,w)$ does not satisfy (i)--(iii). Write $Y=Y_1\# \cdots \# Y_k$ where $L$ is the disjoint union of $L_i\subset Y_i$ where the $(Y_i,L_i)$ satisfy the conditions of Lemma \ref{lemma:tautconnectedsumdecomp}. By the assumption that (iii) does not hold, $w$ is mod 2 homologous to a $1$-manifold $w_1\cup \cdots \cup w_k$ where $w_i\subset Y_i$. By a straightforward generalization of the K\"{u}nneth formula from \cite[Corollary 5.9]{km-unknot}, we have
    \begin{equation}\label{eq:framedkunnethdecomp}
        I^\#(Y,L,w;\F) \cong \bigotimes_{i=1}^k I^\#(Y_i,L_i,w_i;\F).
    \end{equation}
    Lemma \ref{lemma:tautconnectedsumdecomp} lists the three possibilities for each $(Y_i,L_i)$.

    \begin{itemize}
        \item If $(Y_i,L_i)$ is irreducible and $\partial$-irreducible, then $I^\#(Y_i,L_i,w_i;\F)\neq 0$ by Lemma \ref{lemma:frameddetection}.
        \item If $(Y_i,L_i)=(S^1\times S^2,\emptyset)$, by the assumption that (i) does not hold, we have that $w_i$ is mod $2$ nullhomologous. Then $I^\#(Y_i,L_i,w_i;\F) = I^\#(S^1\times S^2;\F) \cong \F^2 \neq 0$.
        \item If $(Y_i,L_i)=(L(p, q), U_{p/q})$ for some $p/q\in \Q \cup \{\infty\}$, then, as (ii) does not hold, $p\neq 0$. Furthermore, $I^\#(L(p,q),U_{p/q},w;\F)\neq 0$ for any $w$. This may be seen by observing that the sutured manifold associated to $(L(p,q),U_{p/q})$ is taut whenever $p\neq 0$.
    \end{itemize}
    From these cases and \eqref{eq:framedkunnethdecomp} we obtain $I^\#(Y,L,w)\neq 0$. Finally, from Fukaya's connected sum theorem \cite{fukaya-connectedsum} over $\F$ adapted to the case of admissible links, there is a spectral sequence from $I(Y,L,w;\F)\otimes H_\ast(SO(3);\F)$ converging to $I^\#(Y,L,w;\F)$. Thus $I(Y,L,w;\F)\neq 0$.
\end{proof}

%!TEX root = main.tex

\section{Instanton homology and the symplectic mapping class group of $M_{g,n}$}\label{sec:mappingclassgroupthm}

In this section, we apply the earlier obtained results to study the symplectic topology of the moduli spaces $M_{g,n}$. In \S \ref{sec:heegaard}, we introduce a class of Lagrangians in $M_{g,n}$ induced by what we call {\emph{standard pairs}}, and in the setting where two such standard pairs are glued, we recast the detection result of Theorem \ref{prop:detection} in terms of Heegaard splitting data. In \S \ref{sec:detectionmappingclasses}, the detection result is studied from the viewpoint of mapping classes of $\Sigma_{g,n}$, and in \S \ref{subsec:pseudoanasov}, the case of pseudo-Anosov mapping classes is considered. Finally, in \S \ref{sec:mappingclassproof}, we prove Theorem \ref{thm:mappingclassgroupaction}.

\subsection{Detection of instanton homology via Heegaard splittings}\label{sec:heegaard}

Theorem \ref{prop:detection} relates non-vanishing of instanton homology to existence of spheres in 3-manifolds. We being this section by the study of such spheres (and disks) for a pair $(M,T)$ of a $3$-manifold and a tangle. A {\emph{splitting sphere}} of $(M,T)$ is an essential sphere in $\text{int}(M)\smallsetminus T$. A {\emph{$T$-compressing disk of $\partial M$}} is a properly embedded disk $D\subset M\smallsetminus T$ with $\partial D$ an essential curve on $\partial M$. 

With $(M,T)$ as above, let $F\subset M$ be a properly embedded surface, and $D$ an embedded disk in $M$ such that $D\cap F = \partial D$ and $D\cap T=\emptyset$. Take a regular neighborhood $N$ of $D$ in $M$, identified with $D\times [0,1]$, so that $N\cap F = \partial D \times [0,1]$ and $N\cap T=\emptyset$. The operation of replacing $\partial D\times [0,1]$ on $F$ with the two disks $D\times \{0,1\}$ is called a {\emph{$2$-surgery}} on $F$ along $D$. If $F'$ is one of the components of a $2$-surgery on $F$, it is said to be obtained by $2$-surgery on $F$.

\begingroup
\renewcommand\labelenumi{(\theenumi)}
\begin{lemma}\label{lemma:haken}
    Suppose $(M,T)$ is a $3$-manifold with tangle, with Heegaard decomposition
    \begin{equation}\label{eq:heegaardsplitting-2}
	(M,T) = (C_1,T_1) \cup_{(\Sigma,\mathcal{P})} (C_2,T_2)
    \end{equation}
    \begin{enumerate}
        \item If $(M,T)$ has a splitting sphere, then there exists a splitting sphere $S$ transverse to $\Sigma$ such that $S\cap \Sigma$ is an essential simple closed curve in $\Sigma\smallsetminus \mathcal{P}$. \label{lemma:haken1} 
        \item If $(M,T)$ has a $T$-compressing disk $D_0$ of $\partial M$, then it has a $T$-compressing disk $D$ with $\partial D = \partial D_0$ such that $D \cap \Sigma$ is an essential simple closed curve in $\Sigma\smallsetminus \mathcal{P}$.\label{lemma:haken1.5} 
        \item If $M$ has a sphere intersecting $T$ transversely in one point, then it has a sphere $S$ transverse to $T$ and $\Sigma$, with $S\cap T$ one point and $S\cap \Sigma$ an essential simple closed curve in $\Sigma\smallsetminus \mathcal{P}$.\label{lemma:haken2} 
    \end{enumerate}
In each case, the resulting sphere or disk is obtained from the original by $2$-surgeries and isotopy. 
\end{lemma}
\endgroup

\begin{proof}
    Items \eqref{lemma:haken1} and \eqref{lemma:haken1.5} follow directly from Theorem 1.3 (2) and (3) of \cite{hayashi-shimokawa}, respectively. 
    
    To deduce \eqref{lemma:haken2}, begin with a sphere $S_0$ in $M$ intersecting the tangle $T$ in one point $p$, transversely. First assume that $p$ lies on a $\partial_+$-parallel arc. After a possible isotopy of $S_0$, we may assume that $p$ is in the interior of $C_1$, and that there is a small open ball $B_1\subset \text{int}(C_1)$ with $B_1\cap T_1$ a trivial arc in $B_1$ and $S_0\cap B_1$ is a disk with $S_0\cap T_1 = \{p\}$ in $B_1$. Set $M'=M\smallsetminus B$, $T'=M'\cap T$, $C_1'=C_1\smallsetminus B$, and $T_1'=C_1'\cap T_1$. Note that $(C_1',T_1')$ is a compression body with trivial tangle. Then we have a Heegaard decomposition
\[
     (M',T') = (C_1', T_1' ) \cup_{(\Sigma,\mathcal{P})} (C_2,T_2).
 \]
    Moreover, $D_0 = S_0\cap M'$ is a disk in $M'\smallsetminus T'$ whose boundary is essential in the annulus $\partial \overline{B}\smallsetminus T'$. Now apply item (2) of the lemma to obtain a disk $D$ in $M'\smallsetminus T'$, with $\partial D = \partial D_0$, and which intersects $\Sigma\smallsetminus\mathcal{P}$ transversely in an essential closed curve. The desired sphere $S\subset M$ is obtained by gluing the closure of the disk $S_0\cap B_1$ to $D$ along their common boundary. 
    
    Now suppose $p$ lies on a vertical arc. In this case, the above construction does not apply, as $T_1'$ is not a trivial tangle in $C_1'$. Instead, let $D_1$ be a small disk neighborhood of $p$ in $S_0$. Let $A_1\subset C_1\smallsetminus T_1$ be an annulus with one boundary component given by $A_1 \cap S_0 = \partial D_1$ and the other equal to the boundary of a disk $D_2\subset \partial_- C_1$ which contains a single point of $\partial T$. The annulus $A_1$ may be obtained by running a tube surrounding the trivial arc going from $\partial D_1$ to $\partial_- C_1$. Cut out $D_1$ from $S_0$ and glue in $A_1$ to obtain a disk in $C_1\smallsetminus T_1$ with boundary $\partial D_2\subset \partial_- C_1$. To this disk we apply item (2) of the lemma to obtain a disk $D$ in $M\smallsetminus T$ with $\partial D=\partial D_2$ and which intersects $\Sigma\smallsetminus\mathcal{P}$ transversely in a single essential simple closed curve. Finally, gluing $D$ and $D_2$ and pushing it into the interior of $C_1$ gives the desired $2$-sphere $S$.
\end{proof}

\noindent We note that item \eqref{lemma:haken1} of Lemma \ref{lemma:haken} is a variation on a theorem of Haken \cite[\S 7]{haken}. The case where $T=\emptyset$ is due to Casson and Gordon \cite[Lemma 1.1]{cg-heegaard}. See Doll \cite{doll} for closely related results in the case where $M$ is closed. 

\begin{figure}[t]
\centering
\includegraphics[scale=0.7]{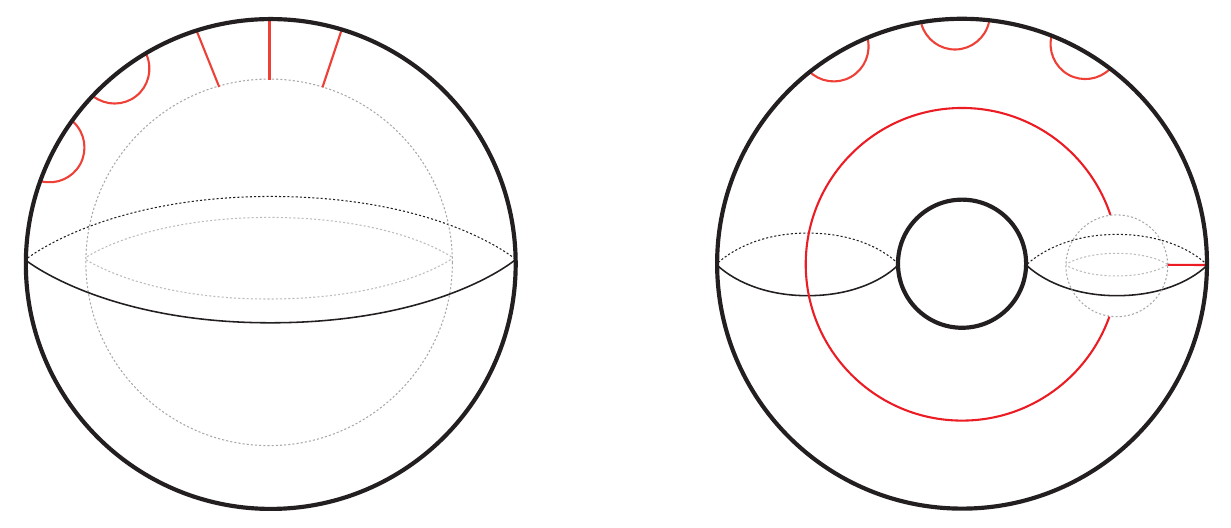}
\caption{\small A standard $(0,7)$ pair of type I (left) and a standard $(1,7)$ pair of type II (right). }
\label{fig:standardpieces}
\end{figure}

Consider a $3$-manifold with tangle $(C,T)$ where $C$ is obtained from $S^2\times [0,1]$ by attaching $g$ $1$-handles along $S^2\times \{1\}$ and $T$ is the union of vertical arcs $\{p_1,p_2,p_3\}\times [0,1]$ with $(n-3)/2$ many $\partial_+$-parallel arcs. We call $(C,T)$ a {\emph{standard $(g,n)$ pair of type I}}, or simply a {\emph{type I pair}}. Note
\begin{equation}\label{eq:boundaryofstandardpair}
	\partial_+ (C,T) \cong \Sigma_{g,n} \qquad \partial_- (C,T) \cong \Sigma_{0,3}
\end{equation}
and that $(C,T)$ is a compression body with trivial tangle. Next, suppose $(C,T)$ is such that $C$ is as above, but $T$ is the union of a vertical arc $\{p_1\}\times [0,1]$, a ``special'' arc $t_0$ which is the core of a $1$-handle and has endpoints on $S^2\times \{0\}$, and $(n-1)/2$ many $\partial_+$-parallel arcs. We call $(C,T)$ a {\emph{standard $(g,n)$ pair of type II}}, or a {\emph{type II pair}}. The boundaries of type II pairs also satisfy \eqref{eq:boundaryofstandardpair}. Note that, in contrast to the type I case, the tangle $T$ in a type II pair is not a trivial tangle due to the special arc $t_0$. Furthermore, observe that a standard $(g,n)$ pair is defined under the conditions
\begin{align*}
	\textup{type I}: & \quad g\geq 0 , \;\; n\geq 3, \;\; n \text{ odd}  \\[1mm]
		\textup{type II}:  & \quad  g\geq 1 , \;\; n\geq 1, \;\; n \text{ odd}  
\end{align*}
In what follows below, we will implictly assume that standard pairs come equipped with preferred orientation-preserving boundary identifications as in \eqref{eq:boundaryofstandardpair}.

Recall that the {\emph{curve complex}} $\mathcal{C}(\Sigma_{g,n})$ of $\Sigma_{g,n}=(\Sigma,\mathcal{P})$ is the simplicial complex where a $k+1$ simplex is a collection of $k+1$ distinct isotopy classes of essential and non-peripheral simple closed curves in $\Sigma\smallsetminus \mathcal{P}$ such that the isotopy classes have representative curves are pairwise disjoint. The distance $d$ between vertices $x,y\in \mathcal{C}(\Sigma_{g,n})$ is the minimum number of $1$-simplices in a simplicial path connecting $x,y$. Denote by $\mathcal{C}(\Sigma_{g,n})^{(i)}$ the $i$-skeleton of $C(\Sigma_{g,n})$, so that $\mathcal{C}(\Sigma_{g,n})^{(0)}$ is the set of vertices. We extend $d$ to pairs of subsets $\mathscr{S}_1, \mathscr{S}_2$ of $\mathcal{C}(\Sigma_{g,n})^{(0)}$ by defining
\[
	d(\mathscr{S}_1, \mathscr{S}_2 ) = \min\left\{ d(x_1,x_2)    \mid  x_i\in\mathscr{S}_i \right\}.
\]

Suppose $(C_1,T_1)$ and $(C_2,T_2)$ are each standard $(g,n)$ pairs (of either type I or II). Define
\begin{equation}\label{eq:standardpairsplitting}
	(Y,L) =  (C_1,T_1) \cup_{\Sigma_{g,n}\sqcup \Sigma_{0,3}} (C_2,T_2),
\end{equation}
 a closed connected oriented $3$-manifold with an embedded link, by gluing along $\Sigma_{g,n}$ and $\Sigma_{0,3}$ using the identifications $\partial_+ (C_i,T_i) \cong \Sigma_{g,n}$ and $\partial_- (C,T) \cong \Sigma_{0,3}$. The orientation of $(Y,L)$ is induced by that of $(C_1,T_1)$ and the reverse of $(C_2,T_2)$. Note that the underlying surface induced by the glued up $\Sigma_{0,3}$ shows that $(Y,L)$ is an admissible link. For $i\in \{1,2\}$ define
\begin{align*}
    \mathscr{D}_i &= \mathscr{D}(C_i,T_i) = \left\{ [x] \; \mid \; x=\partial D, \; D \text{ a disk in } C_i, \; D\cap T_i= \emptyset \right\} \subset \mathcal{C}(\Sigma_{g,n})^{(0)} \\[1mm]
    \mathscr{A}_i &= \mathscr{A}(C_i,T_i) = \left\{ [x] \; \mid \; x=\partial D, \; D \text{ a disk in } C_i, \; D\pitchfork T_i = \text{ point} \right\} \subset \mathcal{C}(\Sigma_{g,n})^{(0)} 
 \end{align*}

\begin{prop}\label{prop:instantonheegaarddetectionnobundle}
    Suppose $(Y,L)$ is obtained from gluing two standard pairs as in \eqref{eq:standardpairsplitting}. Then
    \[
        I(Y,L)=0 \quad \Leftrightarrow \quad \left(\mathscr{D}_1\cap \mathscr{A}_2\right) \cup \left(\mathscr{A}_1\cap \mathscr{D}_2\right) \neq \emptyset.
    \]
\end{prop}

\begin{proof}
    By Theorem \ref{prop:detection}, if $I(Y,L)=0$ then there is a $2$-sphere $S'\subset Y$ intersecting $L$ transversely in one point $p$. We may assume also that $S'$ is transverse to $(S,P)$, the sphere with $3$ points identified by gluing the two copies of $\Sigma_{0,3}$, and that $p\not\in S$. Let $S'$ further have the property that among all such spheres the number of components in $S\cap S'$ is minimal.
    
    Suppose $S\cap S'$ has a component which bounds a disk $D\subset S\smallsetminus P$ containing no other components of $S\cap S'$. Perform $2$-surgery on $S'$ along $D$, avoiding $p$, to obtain two spheres, one of which intersects $L$ transversely in the single point $p$. The sphere also has fewer components of intersection with $S$ than does $S'$, a contradiction. Thus all components in $S\cap S'$ are essential in $S\smallsetminus P$.
    
    If $S\cap S'$ is nonempty, it has a component in $S\smallsetminus P$ which bounds a disk $D\subset S$ with $D\cap P$ a single point, and such that $D$ contains no other components of $S\cap S'$. Do $2$-surgery on $S'$ along $D$ (with the caveat that $D\cap L\neq \emptyset$) to obtain two spheres, one of which intersects $L$ transversely in one point (the other sphere intersecting it in two points), and intersecting $S$ in fewer components than did $S'$. This is a contradiction, and thus $S\cap S'$ is empty.

    Given that $S\cap S'$ is empty, the sphere $S'$ may be viewed as embedded inside $M$, where 
    \[
    (M,T) = (C_1,T_1) \cup_{\Sigma_{g,n}} (C_2,T_2)
 \]
 is the result of cutting $(Y,L)$ along $(S,P)$. This may not be a Heegaard splitting, as type II pairs do not have trivial tangles. To remedy the situation, if $(C_i,T_i)$ is type II, let $(C_i',T_i')$ be the result of removing a small regular neighborhood of the corresponding special arc; if it is type I, let  $(C_i',T_i')$ be $(C_i,T_i)$ unchanged. Then we have a Heegaard splitting
\begin{equation}\label{eq:heegaardsplittingafterremovalofspecialarcs}
     (M',T') = (C_1', T_1' ) \cup_{(\Sigma,\mathcal{P})} (C'_2,T'_2).
 \end{equation}
which naturally embeds into $(M,T)$.

 If the sphere $S'$ does not intersect a special arc in a pair of type II, then $S'=S' \cap M'$. Applying Lemma \ref{lemma:haken} \eqref{lemma:haken2} to the sphere $S'$ in $(M',T')$, we obtain a sphere $S''$ in $M'$ which intersects $T'$ transversely in a single point $p$ and also intersects $\Sigma\smallsetminus\mathcal{P}$ transversely in an essential simple closed curve $x$. We then view $S''$ as a sphere in $M$. Note that $x$ is non-peripheral, for if it were, then the two disk components of $S''\smallsetminus x$ would each intersect the tangle, a contradiction. Next, if $p\in C_1$ then the disk of $S''\smallsetminus x$ on the side of $C_1$ intersects $T_1$ in one point, yielding $[x]\in \mathscr{A}_1$, while the disk of $S''\smallsetminus x$ on the side of $C_2$ misses $T_2$, yielding $[x]\in \mathscr{D}_2$. Thus $\mathscr{A}_1\cap \mathscr{D}_2\neq \emptyset$. Similarly, if $p\in C_2$ then we conclude $\mathscr{D}_1\cap \mathscr{A}_2\neq \emptyset$. 

If $S'$ intersects a special arc in a pair of type II, say $(C_1,T_1)$, then $S' \cap M'$ is a $T'$-compressing disk of $\partial M'$, and Lemma \ref{lemma:haken} \eqref{lemma:haken1.5} applies to give a $T'$-compressing disk $D$ with the same boundary and which intersects $\Sigma\smallsetminus\mathcal{P}$ transversely in an essential simple closed curve $x$. Filling back in the regular neighborhood of the special arc, the disk $D$ completes to a sphere $S''$ which intersects the special arc transversely in a single point, otherwise misses $T$, and intersects $\Sigma\smallsetminus\mathcal{P}$ transversely in an essential simple closed curve $x$. The disk of $S''\smallsetminus x$ on the side of $C_1$ yields $[x]\in \mathscr{A}_1$, while the disk of $S''\smallsetminus x$ on the side of $C_2$ yields $[x]\in \mathscr{D}_2$. Thus $\mathscr{A}_1\cap \mathscr{D}_2\neq \emptyset$. Similarly, if $S'$ had intersected a special arc in $(C_2,T_2)$, then we obtain $\mathscr{D}_1\cap \mathscr{A}_2\neq \emptyset$.

    The converse is immediate: if $[x]\in \mathscr{A}_i\cap \mathscr{D}_j\neq \emptyset$ where $\{i,j\}=\{1,2\}$, then one can glue disks from the defining properties of $\mathscr{A}_i$ and $\mathscr{D}_j$ to obtain a sphere in $Y$ intersecting $L$ transversely in one point, and Theorem \ref{prop:detection} implies $I(Y,L)=0$.
\end{proof}

We expand the above result to include non-trivial bundle data. We consider triples
\begin{equation}\label{eq:heegaardsplittingwithbundlem}
    (Y,L,w) = (C_1,T_1,w_1) \cup_{\Sigma_{g,n}\sqcup \Sigma_{0,3}} (C_2,T_2,w_2)
\end{equation}
where $(C_i,T_i)$ are standard pairs and $w=w_1\sqcup w_2$ is the union of $1$-manifolds $w_i\subset \text{int}(C_i)$ with $w_i\cap T_i = \partial w_i$. For $\varepsilon\in \{0,1\}$, define subsets of $\mathscr{D}_i$ as follows:
\begin{align*}
    \mathscr{D}_i^\varepsilon &= \mathscr{D}^\varepsilon(C_i,T_i,w_i) = \left\{ [x]  \mid  x=\partial D,  D \text{ a disk in } C_i,  D\cap T_i= \emptyset,  [D]\cdot [w_i]\equiv \varepsilon \!\!\! \pmod{2} \right\} 
\end{align*}

\begin{prop}\label{prop:instantonheegaarddetection}
Suppose the admissible link $(Y,L,w)$ is obtained from gluing two standard pairs $(C_i,T_i)$ with $1$-manifolds $w_i\subset \textup{int}(C_i)$ as in \eqref{eq:heegaardsplittingwithbundlem}. Then
    \[
        I(Y,L,w)=0 \quad \Leftrightarrow \quad \left(\mathscr{D}_1\cap \mathscr{A}_2\right) \cup \left(\mathscr{A}_1\cap \mathscr{D}_2\right) \cup \left( \mathscr{D}^0_1\cap \mathscr{D}^1_2\right) \cup \left( \mathscr{D}^1_1\cap \mathscr{D}^0_2\right) \neq \emptyset.
    \]
\end{prop}

\begin{proof}
    By Theorem \ref{prop:detection}, if $I(Y,L,w)=0$ there is an embedded $2$-sphere $S'\subset Y$, transverse to $L$, such that either (i) $S'\cap L$ is a point or (ii) $S'\cap L=\emptyset$ and $[S']\cdot w$ is odd. In case (i), the proof of Proposition \ref{prop:instantonheegaarddetectionnobundle} implies $\left(\mathscr{D}_1\cap \mathscr{A}_2\right) \cup \left(\mathscr{A}_1\cap \mathscr{D}_2\right)\neq \emptyset$, so suppose case (ii) holds. 
    
    Suppose $S'$ intersects the glued 3-pointed sphere $(S,P)$ transversely, and that the number of components in $S'\cap S$ is minimal among all such spheres that are disjoint from $L$ and have odd pairing with $[w]$. The same argument as before shows that $S\cap S'$ has no components that bound a disk in $S\smallsetminus P$. If $S\cap S'$ is nonempty, it then has a component in $S\smallsetminus P$ which bounds a disk $D\subset S$ with $D\cap P$ one point, with $D$ disjoint from all other components of $S\cap S'$. Perform $2$-surgery on $S'$ along $D$. Each of the resulting $2$-spheres intersects $L$ in one point, reverting to case (i).

    Thus we continue on with case (ii), assuming $S\cap S'$ is empty. As before, let $(M',T')$ be obtained from $(Y,L)$ by cutting along $S$ and removing regular neighborhoods of special arcs; it has a Heegaard splitting as in \eqref{eq:heegaardsplittingafterremovalofspecialarcs}. Let $w'=w\cap M'$. Since $S'$ is disjoint from $S$ and $L$, it lies in the interior of $M'\smallsetminus T'$. By Lemma \ref{lemma:haken} \eqref{lemma:haken1} applied to $S'$ in the Heegaard splitting of $(M',T')$, there exists a sphere $S''$ in $M'\smallsetminus T'$ such that $S''$ intersects $\Sigma\smallsetminus\mathcal{P}$ in at most one curve. However, we require also that $[S'']\cdot [w]$ is odd. The proof of Lemma \ref{lemma:haken} from \cite[Theorem 1.3]{hayashi-shimokawa} shows that $S''$ is obtained from a sequence of isotopies and $2$-surgeries on disks, where at each surgery at least one of the resulting $2$-spheres does not bound a ball, and choosing one of these spheres at each stage eventually yields $S''$. Observe that at each of these surgeries, one of the resulting $2$-spheres has odd pairing with $[w]$, and in the scheme of the proof we make the minor additional requirement that $S''$ is obtained by always taking this resulting $2$-sphere. The proof carries through without any other changes, and $[S'']\cdot [w]$ is odd, as desired. 

    Thus $S''\cap \Sigma$ is an essential simple closed curve in $\Sigma\smallsetminus\mathcal{P}$. This curve divides $S''$ into a pair of disks, one in $C_1$ and the other in $C_2$, exactly one of which has odd pairing with $[w_i]$, implying that one of $\mathscr{D}^0_1\cap \mathscr{D}^1_2$ or $\mathscr{D}^1_1\cap \mathscr{D}^0_2$ is nonempty.

    The converse statement is immediate, as each of the four possible intersections of disk sets leads to a sphere in $Y$ of type (i) or (ii).
\end{proof}

Give a triple $(C,T,w)$ where $(C,T)$ is a standard $(g,n)$ pair of either type I or II and a compact $1$-manifold $w\subset \text{int}(C)$ satisfying $w\cap T = \partial w$, we consider
\[
	 \mathfrak{X}(C,T,w) \subset M_{g,n}.
\]
By Proposition \ref{unperturbed-3-man-Lag}, this is an embedded Lagrangian submanifold of $M_{g,n}$. Furthermore, we have
\begin{equation}\label{eq:lagrangianidentification}
	\mathfrak{X}(C,T,w) \cong \begin{cases} \left( S^3 \right)^{g} \times \left(S^2 \right)^{(n-3)/2} \quad & \textup{(type I)} \\[2mm]
		S^1 \times \left( S^3 \right)^{g-1} \times \left(S^2 \right)^{(n-1)/2} \quad  & \textup{(type II)}  \end{cases}
\end{equation}
To explain this, first consider the case in which $(C,T)$ is type I. Recall that elements of $M_{g,n}$ may be viewed as conjugacy classes of tuples $A_i,B_i,C_j$ (where $1\leq i\leq g$ and $1\leq j\leq n$) such that 
\[
	[A_1,B_1]\cdots [A_g,B_g] C_1\cdots C_n = 1.
\]
Recall that $C_j$ is the holonomy around a small loop encircling the point $p_{j}$. In our presentation of $(C,T)$, we may arrange that each $B_i$ corresponds to the holonomy around a loop which is a cocore of an attached $1$-handle, and that for $2\leq k \leq (n-1)/2$ the points $p_{2k}$ and $p_{2k+1}$ are connected by a $\partial_+$-parallel arc. Then $\mathfrak{X}(C,T,w)$ is the conjugacy classes of tuples satisfying
\begin{gather*}
	B_i = (-1)^{\varepsilon_i} \;\; (1\leq i\leq g),\\[2mm]
	 C_1C_2C_3 = (-1)^{\varepsilon_{123}},\\[2mm]
	 C_{2k} C_{2k+1} = (-1)^{\varepsilon'_{k}} \;\; (2\leq k \leq (n-1)/2)
\end{gather*}
where $\varepsilon_i$ is the intersection number of $w$ with a cocore disk with boundary $B_i$, $\varepsilon_{123}$ is the number of endpoints of $w$ on $\{p_1,p_2,p_3\}\times [0,1]$ and $\varepsilon'_{k}$ is the number of endpoints of $w$ on the $\partial_+$-parallel arc connecting $p_{2k}$ and $p_{2k+1}$. We may conjugate so that 
\[
  C_1=   \left[\begin{array}{cc} i & 0 \\ 0 & -i \end{array}\right], \quad C_2 =    \left[\begin{array}{cc} 0 & -1 \\ 1 & 0 \end{array}\right], \quad C_3 =  (-1)^{\varepsilon_{123}}\left[\begin{array}{cc} 0 & i \\ i & 0
  \end{array}\right].
\]
Then there are free variables $A_i\in SU(2)\cong S^3$ for $1\leq i \leq g$ and $C_{2k}\in \Gamma(1/2)\cong S^2$ for $2\leq k \leq (n-1)/2$, with no relations, yielding the first identification in \eqref{eq:lagrangianidentification}.

In the type II case, similar reasoning shows that $\mathfrak{X}(C,T,w)$ is defined by
\begin{gather*}
	B_i = (-1)^{\varepsilon_i} \;\; (2\leq i\leq g),\\[2mm]
	 [A_1,B_1]C_1 = (-1)^{\varepsilon_{0}}, \quad B_1\in \Gamma(1/2)\\[2mm] 
	 C_{2k} C_{2k+1} = (-1)^{\varepsilon'_{k}} \;\; (1\leq k \leq (n-1)/2)
\end{gather*}
where $\varepsilon_0$ is the number of endpoints of $w$ on the special arc $t_0$ and the vertical arc $\{p_1\}\times [0,1]$. Here we have chosen $B_1$ to be the holonomy along the cocore of the $1$-handle which has a special arc in its core. Up to conjugation there is a circle of solutions to the system $[A_1,B_1]C_1=\pm 1$, $B_1\in \Gamma(1/2)$; the remaining free variables $B_i\in SU(3)\cong S^3$ for $2\leq i\leq g$ and $C_{2k}\in \Gamma(1/2)\cong S^2$ for $1\leq k\leq (n-1)/2$, and we obtain the second identification in \eqref{eq:lagrangianidentification}.

\subsection{Detection of instanton homology via mapping classes of $\Sigma_{g,n}$} \label{sec:detectionmappingclasses}

Next fix one standard $(g,n)$ pair $(C,T)$ and let $\mathscr{D}=\mathscr{D}(C,T)$ be the corresponding disk set. Let $(C_1,T_1)$ and $(C_2,T_2)$ be copies of the standard pair $(C,T)$. Orient $C_2$ oppositely that of $C_1$. Let $w_i\subset \text{int}(C_i)$ be (possibly different) $1$-manifolds with $w_i\cap T_i = \partial w_i$. Given $\phi\in \text{Mod}(\Sigma_{g,n})$, form
\[
    (Y_\phi, L_{\phi}, w_1\sqcup w_2) := (C_1, T_1, w_1) \cup_{\Sigma_{g,n} \sqcup \Sigma_{0,3}} (C_2, T_2, w_2)
\]
where the copies of $\Sigma_{g,n}$ are glued via $\phi$ and the copies of $\Sigma_{0,3}$ are glued by the identity. Recall that $d(\cdot , \cdot)$ is the simplicial distance function on the curve complex $\mathcal{C}(\Sigma_{g,n})$.

\begin{prop} \label{prop:nonvanishingformappingclasswithdistancebound}
If $\phi\in \textup{Mod}(\Sigma_{g,n})$ satisfies $d(\mathscr{D}, \phi(\mathscr{D}))>1$, then for all choices of the $1$-manifolds $w_1,w_2$ we have $I(Y_\phi,L_\phi,w_1\sqcup w_2)\neq 0$.
\end{prop}

\begin{proof}
In addition to $\mathscr{D}=\mathscr{D}(C,T)$, we have  $\mathscr{A}=\mathscr{A}(C,T)$ and  $\mathscr{D}_i^\varepsilon = \mathscr{D}^\varepsilon(C,T,w_i)$ for $\varepsilon\in \{0,1\}$ and $i\in \{1,2\}$. By Proposition \ref{prop:instantonheegaarddetection}, the non-vanishing of $I(Y_\phi, L_\phi, w_1\sqcup w_2)$ is equivalent to
\begin{equation}\label{eq:vanishingrequiredforproof}
	 \left(\phi(\mathscr{D})\cap \mathscr{A}\right) \cup \left(\phi(\mathscr{A})\cap \mathscr{D}\right) \cup \left( \phi(\mathscr{D}^0_1)\cap \mathscr{D}^1_2\right) \cup \left( \phi(\mathscr{D}^1_1)\cap \mathscr{D}^0_2\right) = \emptyset.
\end{equation}
Assume $d(\mathscr{D}, \phi(\mathscr{D}))>1$. Then $\mathscr{D}\cap \phi(\mathscr{D})$ is empty. As $\mathscr{D}_i^\varepsilon \subset \mathscr{D}$ we also have that $ \phi(\mathscr{D}^0_1)\cap \mathscr{D}^1_2$ and $\phi(\mathscr{D}^1_1)\cap \mathscr{D}^0_2$ are empty. It remains to argue that $ \phi(\mathscr{D})\cap \mathscr{A} $ and $ \phi(\mathscr{A})\cap \mathscr{D} $ are empty.

We first claim that for any $[x]\in \mathscr{A}$, there exists $[y]\in \mathscr{D}$ such that $d(x,y)=1$, i.e. $x$ and $y$ are not isotopic and can be made disjoint. Let $D$ be a disk in $C$ with $\partial D = x$ such that $D$ intersects $T$ transversely in one point. Let $D'$ be a disk in $C\smallsetminus T$ with $\partial D'=y$, where $[y]\in \mathscr{D}$. Let us assume, having fixed $[x]$, that $D$ and $D'$ are chosen in a way such that they are transverse and $D\cap D'$ has the minimal possible number of components. We will show that $d(x,y)=1$. 

If $D\cap D'$ contains any circles, then it contains a circle that bounds a disk $D''\subset D'$. Surgery of $D$ along $D''$ produces a disk $D_0$ and a sphere $S_0$. One of $D_0$ or $S_0$ intersects $T$ in one point, while the other misses $T$. If $S_0$ were to intersect $T$ in one point, the associated Lagrangian would be empty, contradicting the computations of the previous section. Thus $D_0$ intersects $T$ in one point. As $D_0$ has boundary $x$ and intersects $D'$ in fewer components than $D$, this gives a contradiction to the minimality assumption placed on $D, D'$. Thus $D\cap D'$ does not contain circles. 

Suppose $D\cap D'$ contains an arc. Then there exists such an arc which bounds a disk $D''\subset D$ which misses $T$. Surgery of $D'$ along $D''$ produces two disks $D'_1$ and $D'_2$ in $C\smallsetminus T$. Write $y_i = \partial D'_i$. If $y_1$ bounds a disk in $\Sigma\smallsetminus \mathcal{P}$, then $D'_2$ has boundary isotopic to $x$ and $D'_2$ intersects $D$ in fewer components than $D'$, a contradiction. Similarly if $y_2$ bounds a disk in $\Sigma\smallsetminus \mathcal{P}$. Next, suppose that $y_1$ bounds a disk in $\Sigma$ containing one marked point. Then this disk together with $D'_1$ forms a sphere intersecting the tangle in one point, leading to a contradiction as above. Similarly for $y_2$. Thus each of $y_1$ and $y_2$ are essential and non-peripheral, hence, given $y_i = \partial D_i'$, we have $[y_i]\in\mathscr{D}$. On the other hand, each of $D'_1$ and $D'_2$ intersect $D$ in few components than $D'$, a contradiction. Thus $D\cap D'=\emptyset$. It follows that $x$ and $y$ are disjoint, establishing $d(x,y)=1$.

Next, let $[x]\in \mathscr{D}$ and $[x']\in\mathscr{A}$ be such that $d(\mathscr{D},\phi(\mathscr{A}))=d(x,\phi(x'))$. By the established claim above, there exists $[y]\in\mathscr{D}$ such that $d(\phi(x'),\phi(y))=1$. Then
    \[
        d(\mathscr{D},\phi(\mathscr{A})) = d(x,\phi(x')) \geq d(x,\phi(y)) - d(\phi(x'),\phi(y)) \geq d(\mathscr{D},\phi(\mathscr{D})) - 1 > 0,
    \]
    where we have used our assumption $d(\mathscr{D},\phi(\mathscr{D}))>1$. Thus $\mathscr{D}\cap \phi(\mathscr{A})$ is empty. A similar argument shows that $\mathscr{A}\cap \phi(\mathscr{D})$ is empty. Having established \eqref{eq:vanishingrequiredforproof}, the proof is complete.
\end{proof}
 
If $3g + n \leq 4$, then $(g,n)\in \{(0,1), (0,3), (1,1)\}$. Standard pairs are not defined in the case $(g,n)=(0,1)$, so Proposition \ref{prop:nonvanishingformappingclasswithdistancebound} does not apply. If $(g,n)$ is equal to either $(0,3)$ or $(1,1)$, then it is easy to see that the relevant disk sets $\mathscr{D}$ are always empty, and so the hypothesis $d(\mathscr{D},\phi(\mathscr{D})) = \min (\emptyset ) = +\infty >1$ of Proposition \ref{prop:nonvanishingformappingclasswithdistancebound} is automatically satisfied.

\begin{prop}\label{prop:vanishingforidentitymappingclass}
	If $3g+n>4$, then there are $w_1,w_2$ such that $ I(Y_\textup{id},L_\textup{id},w_1\sqcup w_2)=0$.
\end{prop}

\begin{proof}
	The assumption $3g+n>4$ guarantees that $(C,T)$ has either a $1$-handle (without a special arc) or a $\partial_+$-parallel arc. In the first case, let $D\in \mathscr{D}$ be a cocore of the $1$-handle and let $w_1$ be a circle which runs once along the core of the $1$-handle. In the second case, let $D$ be a properly embedded disk which together with $\partial_+ C$ encloses a trivial ball neighborhood of a $\partial_+$-parallel arc, and let $w_1$ be an arc which connects the $\partial_+$-parallel arc to any other arc in $T$. In either case set $w_2=\emptyset$. Then the double of $D$ inside $Y_{\text{id}}=C\cup C$ is a $2$-sphere which is disjoint from $L_\text{id}=T\cup T$ and intersects $w_1\sqcup w_2$ transversely in one point. The result then follows from Theorem \ref{thm:detectionintro}.	
\end{proof}

\begin{prop}\label{prop:nonsymplecticisotopicformappingclassesviadistance}
	Suppose $3g+n>4$ and $\phi\in \textup{Mod}(\Sigma_{g,n})$ satisfies $d(\mathscr{D}, \phi(\mathscr{D}))>1$. Then the symplectomorphism of $M_{g,n}$ induced by $\phi$ is not Hamiltonian isotopic to the identity.
\end{prop}

\begin{proof}
Let $(C,T)$ be a standard $(g,n)$ pair and let $w_1,w_2$ be as in Proposition \ref{prop:vanishingforidentitymappingclass}. Denote by $\Lambda_i = \mathfrak{X}(C,T,w_i)$ the associated Lagrangians in $M_{g,n}$. By Theorem \ref{thm:lagrangianineq}, $\mathfrak{n}(\Lambda_1,\Lambda_2)=0$; indeed, $\Lambda_1\cap \Lambda_2=\emptyset$. On the other hand, by Theorem \ref{thm:lagrangianineq} and Proposition \ref{prop:nonvanishingformappingclasswithdistancebound}, we have $\mathfrak{n}(\phi(\Lambda_1),\Lambda_2)>0$. As $\mathfrak{n}(\Lambda_1,\Lambda_2)$ is an invariant of $\Lambda_1,\Lambda_2$ up to Hamiltonian isotopies, the result follows.
\end{proof}

\subsection{The case of pseudo-Anosov mapping classes}\label{subsec:pseudoanasov}

In order to leverage Proposition \ref{prop:nonsymplecticisotopicformappingclassesviadistance}, some control over $d(\mathscr{D},\phi(\mathscr{D}))$ is needed. This is done here for the case of pseudo-Anosov mapping classes.

Given the marked surface $\Sigma_{g,n}=(\Sigma,\mathcal{P})$, denote by $\Sigma^\circ$ the surface with boundary obtained from $\Sigma$ by removing disk neighborhoods of the points in $\mathcal{P}$. (This slighly diverges from earlier notation, where $\Sigma^\circ$ was the punctured surface.) A diffeomorphism $\phi:\Sigma^\circ\to \Sigma^\circ$ is  {\emph{pseudo-Anosov}} if it is the identity on the boundary, and there exists a pair of measured foliations $\mathscr{F}^u$ and $\mathscr{F}^s$ on $\Sigma^\circ$ such that $\phi(\mathscr{F}^u) = \lambda\mathscr{F}^u $ and $\phi(\mathscr{F}^s) = \lambda^{-1}\mathscr{F}^s $ for some constant $\lambda>1$. The measured foliations $\mathscr{F}^u$ and $\mathscr{F}^s$ are transverse on the interior of $\Sigma^\circ$ and each component of $\partial \Sigma^\circ$ is a cycle of leaves of $\mathscr{F}^u$ (resp. $\mathscr{F}^s$) and contains a singularity of $\mathscr{F}^u$ (resp. $\mathscr{F}^s$). See \mbox{\cite[\S 11]{flp}}. Let $(C,T)$ be a standard $(g,n)$ pair of either type I or II, equipped with an identification $\partial_+(C,T)\cong \Sigma_{g,n}$, and denote by $\mathscr{D}\subset \mathcal{C}(\Sigma_{g,n})^{(0)}$ the corresponding disk set.

\begin{prop}\label{prop:highdistance}
    Suppose $\phi:(\Sigma,\mathcal{P})\to (\Sigma,\mathcal{P})$ restricts to a pseudo-Anosov diffeomorphism of $\Sigma^\circ = \Sigma\smallsetminus\textup{nbhd}(\mathcal{P})$. Then there exists a diffeomorphism $\psi$ of $(\Sigma,\mathcal{P})$ such that 
    \[
        \limsup_{m\to \infty}\;\; d(\mathscr{D}, \; \psi \phi^m \psi^{-1}(\mathscr{D} )) =\infty.
    \]
\end{prop}

The strategy used to prove this proposition is an adaptation of one from Ichihara--Saito \cite{ichihara-saito}, which itself relies on ideas from \cite{hempel, abrams-schleimer, mms} as further explained below. Our setup is similar but slightly different than that of Ichihara--Saito. For example, in their work, the number of points $|\mathcal{P}|$ on $\Sigma$ is even, and the subset $\mathscr{D}$ of $\mathcal{C}(\Sigma_{g,n})^{(0)}$ is instead the curves bounded by disks in a standard handlebody with some $\partial_+$-parallel arcs removed.

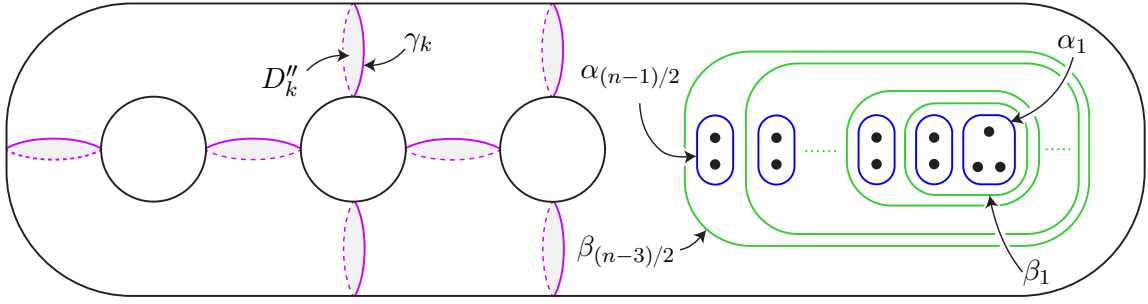
\begin{figure}[t]
\centering
\labellist
  \pinlabel {$\gamma_k$} [r] at 233 147
  \pinlabel {$D_k''$} [r] at 163 127
  \pinlabel {$\alpha_{(n-1)/2}$} [r] at 362 129
  \pinlabel {$\beta_{(n-3)/2}$} [r] at 359 38
  \pinlabel {$\alpha_1$} [r] at 574 146
  \pinlabel {$\beta_1$} [r] at 555 27
\endlabellist
\includegraphics[scale=0.72]{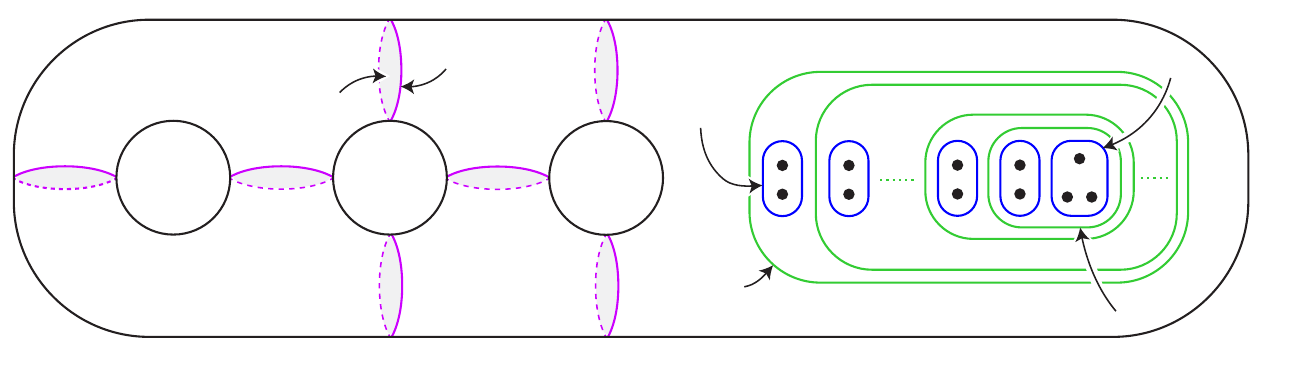}
\caption{\small The collection of curves $\mathscr{C}$ on $\Sigma_{g,n}$ associated to a standard pair $(C,T)$ of type I.}
\label{fig:compbody1}
\end{figure}

To set the stage, we give a more explicit presentation of a standard pair $(C,T)$. We consider the type I case. Consider a standard genus $g$ handlebody $H$ with boundary $\Sigma$ and $n$ points $\mathcal{P}=\{p_1,\ldots,p_{n}\}\subset \Sigma$. Consider a set of simple closed curves in $\Sigma^\circ=\Sigma\smallsetminus \text{nbhd}(\mathcal{P})$ denoted
\[
    \mathscr{C}=\{\alpha_i , \beta_j , \gamma_k \; \mid \; 1\leq i \leq (n-1)/2, 1\leq j \leq (n-3)/2, 1\leq k\leq 3g-2\},
\]
whose elements are  pairwise disjoint and non-isotopic, and satisfy the following:
\begin{enumerate}[label=(\roman*)]
    \item  $\alpha_i$ bounds a disk $\Delta_i\subset \Sigma$ with $\{\delta\in \mathscr{C} \, \mid \, \delta\subset \text{int} (\Delta_i) \} = \emptyset$; \label{list:type1specificpresentation.1}
    \item  the disk $\Delta_1$ satisfies $\Delta_1\cap \mathcal{P}=\{p_1,p_2,p_3\}$; \label{list:type1specificpresentation.2}
    \item for $2\leq i \leq (n-1)/2$, the disk $\Delta_i$ satisfies $\Delta_i\cap \mathcal{P}=\{p_{2i},p_{2i+1}\}$; \label{list:type1specificpresentation.3}
      \item  $\beta_j$ bounds a disk $\Delta_j'\subset \Sigma$ with $\{\delta\in \mathscr{C} \, \mid \, \delta\subset \text{int} \Delta_j' \} = \{\alpha_1,\ldots,\alpha_{j+1}, \beta_1,\ldots,\beta_{j-1}\}$; \label{list:type1specificpresentation.4}
    \item $\gamma_k$ bounds a disk $D_k''$ in $H$; \label{list:type1specificpresentation.5}
    \item $\text{int}(\Sigma^\circ)\smallsetminus \bigcup_{\delta\in \mathscr{C}} \delta = Q\sqcup \bigsqcup_{\ell} F_\ell$ where $Q$ is a $4$-holed sphere and each $F_\ell$ is a pair of pants. \label{list:type1specificpresentation.6}
\end{enumerate}
More specifically, we may choose $\mathscr{C}$ to be the collection of curves depicted in Figure \ref{fig:compbody1}. Note $Q$ and $F_\ell$ are open, and one should take their closures to include their boundary curves. Note $\overline{Q} = \Delta_1\cap \Sigma^\circ$. Call a pair of pants $F_\ell$ {\emph{exterior}} if its closure intersects $\partial \Sigma^\circ$ and {\emph{interior}} otherwise. The exterior pairs of pants are given by $\Delta_i\cap \Sigma^\circ$ for $i\geq 2$.

By slightly pushing the disks $\Delta_i$ and $\Delta_j'$ into $H$ we obtain properly emebedded disks $D_i$ and $D_j'$ in $H$. This is done so that the disks $D_i,D_j',D_k''$ are pairwise disjoint. 

Denote by $B_i\subset H$ the ball bounded by $D_i\cup \Delta_i$. Let $B'_1$ be an open ball whose closure is contained in $B_1$. Let $C=H\smallsetminus B'_1$ so that $C$ is a compression body with $\partial_+C = \Sigma$ and $\partial_- C= S^2$. For $i\geq 2$, take an arc between $p_{2i}$ and $p_{2i+1}$ and push it into $B_i$ to obtain a $\partial_+$-parallel arc. Take three properly embedded arcs in $B_1\cap C$ each intersecting both $\Sigma$ and $S^2$. Then the union of all these arcs gives a trivial tangle $T\subset C$, and $(C,T)$ is a standard $(g,n)$ pair of type I. 

A similar setup can be specified for a standard pair of type II, as follows. Let $(C,T)$ be a standard $(g-1,n+2)$ pair of type I as detailed above. Form $C'$ by attaching a $1$-handle to small disk neighborhoods of $p_2$ and $p_3$ which are contained in $\Delta_1$, and form $T'$ be connecting the vertical arcs of $p_2$ and $p_3$ by the core of the attached $1$-handle. Then $(C',T')$ is a standard $(g,n)$ pair of type II, with induced curve set $\mathscr{C}$ and disks as described for $(C,T)$. Note that in the type II case, $Q$ is instead a torus with one puncture and one open disk removed.

We review some terminology from \cite{mms}. A {\emph{seam}} (resp. {\emph{wave}}) on a pair of pants $F$ is an essential properly embedded arc in $F$ with endpoints on two separate components (resp. one component) of $\partial F$. Up to isotopy rel boundary, $F$ has three distinct seams and three distinct waves. If $F$ is a pair of pants inside a surface $S$, then a curve $\gamma\subset S$ has a {\emph{seam}} (resp. {\emph{wave}}) in $F$ if it intersects $F$ minimally in its isotopy class and some component of $\gamma\cap F $ is a seam (resp. wave).

\begin{lemma}\label{lemma:wave}
    If $[\delta]\in \mathscr{D}$ is not isotopic to a curve in $\mathscr{C}$, it has a wave in an interior pair of pants.
\end{lemma}

\begin{proof}
We prove the statement in the case that $(C,T)$ is type I. The type II case is similar. Let $S\subset C$ be the union of the disjoint disks $D_i,D_j',D_k''$ whose boundaries are the curves in $\mathscr{C}$. Write $C^\circ$ for the complement of a regular neighborhood of $T$ in $C$ such that $C^\circ \cap \Sigma = \Sigma^\circ$ and $S\subset C^\circ$. Let $[\delta]\in \mathscr{D}$. Then $\delta=\partial D\subset \Sigma^\circ$ where $D\subset C^\circ$ is a properly embedded disk. Suppose $D$ is chosen such that it is transverse to $S$ and it has the minimal number of components in $D\cap S$ among all such disks with $\partial D$ in the isotopy class of $\delta\subset \Sigma^\circ$. 

    The closures of the connected components of $C^\circ \smallsetminus S$ are of three types:
    \begin{itemize}
        \item[(i)] a ball, whose boundary is the union of an interior pair of pants and three disks from $S$;
        \item[(ii)] a thickened annulus $A_i$ $(i\geq 2)$, with $\partial A_i$ the union of $D_i\subset S$ and $\Delta_i\cap \Sigma^\circ$;
        \item[(iii)] a thickened pair of pants, whose boundary contains $\overline{Q}=\Delta_1\cap \Sigma^\circ$.  
    \end{itemize}
    
    Suppose $D\cap S$ is empty. Then $D$ is contained in one of the three types of components above. Since $\delta=\partial D$ is essential and non-boundary parallel in $\Sigma^\circ$, it is easy to see that $\delta$ must be isotopic in $\Sigma^\circ$ to a curve in $\mathscr{C}$ in each case.

    Suppose $D\cap S$ is nonempty. If $D\cap S\subset D$ contains a circle, choose an innermost circle $c$, intersecting $S$ in one of its defining disks $D'$. Then $c$ bounds a disk $D''\subset D$ which is properly embedded in one of the above components $C'$ of (i)--(iii). In each case, an isotopy of $D$ through properly embedded disks in $C$ is obtained by pushing $D''$ through $D'$ to lie outside of $C'$, removing $c$ from $D\cap S$ and contradicting the minimality assumption. Thus $D\cap S$ has no circles.

    Thus $D\cap S$ is a nonempty union of arcs in $D$. Let $a\subset D\cap S$ be an outermost arc. Then there is an arc $b\subset \partial D$ such that $a\cap b=\partial a=\partial b$ and $a\cup b$ bounds a disk $D''\subset D$ with $\text{int}(D'')$ disjoint from $S$. Moreover, $a$ is contained in one of the disks $D'$ forming $S$, with $\partial D'=\varepsilon \in \mathscr{C}$. In particular, $\partial b \subset \varepsilon$. The disk $D''$ is properly embedded in one of the components $C'$ in (i)--(iii) above. In case (i), the arc $b$ cannot be boundary parallel in the pair of pants to $\varepsilon$, for otherwise an isotopy of $D$ through properly embedded disks in $C$ may be induced by pushing $D''$ through $D'$, contradicting the minimality assumption. Thus $b$ is a wave in the pair of pants, as $b$ is an essential arc and $\partial b\subset \varepsilon$ is contained in a single boundary component of the pair of pants. Finally, $D''$ lying in a component $C'$ from either case (ii) or (iii) leads to a contradiction on the minimality assumption, as in these cases $D''$ can be properly pushed out of $C'$. 
\end{proof}

\begin{figure}[t]
\centering
\includegraphics[scale=0.71]{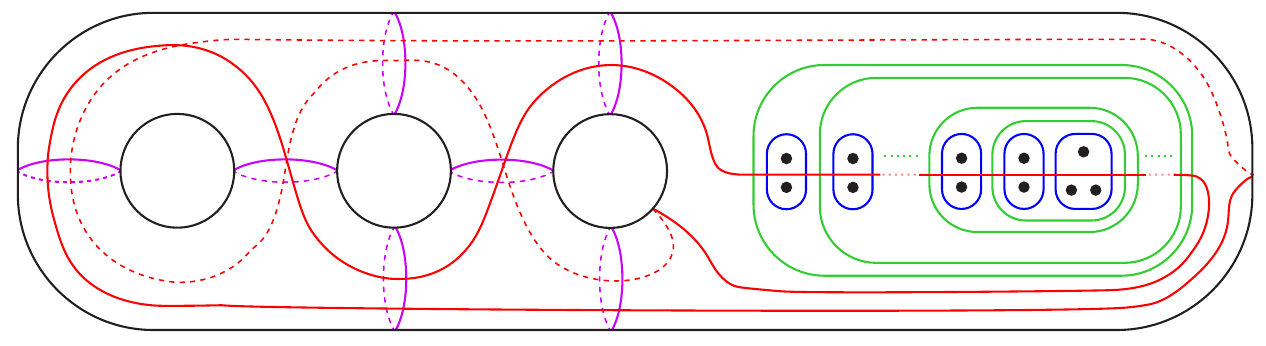}
\caption{\small A simple closed closed curve $\lambda$ which has every seam in every interior pair of pants.}
\label{fig:compbody1}
\end{figure}

\begin{proof}[Proof of Proposition \ref{prop:highdistance}]
    The argument is an adaptation of the one from \cite[\S 3.3]{ichihara-saito}.  Denote by $\overline{\mathscr{D}}$ the closure of $\mathscr{D}$ as viewed as a subset of $\mathcal{P}\mathcal{M}\mathcal{L}(\Sigma^\circ)$, the space of projective measured laminations of $\Sigma^\circ$. Suppose $\eta:\Sigma^\circ\to \Sigma^\circ$ is a pesudo-Anosov diffeomorphism such that its stable and unstable foliations $\mathcal{F}^s_\eta$ and $\mathcal{F}^u_\eta$ satisfy $[\mathcal{F}^s_\eta]\not\in\overline{\mathscr{D}}$ and $[\mathcal{F}^u_\eta]\not\in\overline{\mathscr{D}}$. Then 
    \[
        \limsup_{m\to \infty}\;\; d(\mathscr{D}, \; \eta^m(\mathscr{D} )) =\infty.
    \]
    This follows from the same argument given for Claim 2 in \cite{ichihara-saito}, which is a straightforward extension of the result in the case without tangles present, the latter of which is due to Hempel \cite{hempel} and explained by Abrams--Schleimer \cite[Theorem 2.4]{abrams-schleimer}.

    Choose a simple closed curve $\lambda\subset \Sigma^\circ$ which has every seam in every interior pair of pants in $\Sigma^\circ$. An example is shown in Figure \ref{fig:compbody1}. Denote by $\tau:\Sigma^\circ\to \Sigma^\circ$ the Dehn twist along $\lambda$. For a positive integer $N$, the pseudo-Anosov diffeomorphism
    \[
        \phi_N = \tau^N \circ \phi \circ \tau^{-N}
    \]
    has stable and unstable laminations given by $\tau^N(\mathcal{F}^s)$ and $\tau^{N}(\mathcal{F}^u)$ where $\mathcal{F}^s$ and $\mathcal{F}^u$ are the stable and unstable laminations of $\phi$. To complete the proof it suffices to show that $\tau^N(\mathcal{F}^s)\not\in \overline{\mathscr{D}}$ and $\tau^N(\mathcal{F}^u)\not\in \overline{\mathscr{D}}$ for sufficiently large $N$. The argument is a mild modification of the one for Claim 1 in \cite{ichihara-saito}, which uses ideas from  \cite{mms}.

    First, as $\mathcal{F}^u$ and $\mathcal{F}^s$ are the stable and unstable laminations of a pseudo-Anosov map $\phi$, they are filling, in the sense that they intersect every essential non-boundary parallel curve in $\Sigma^\circ$. From this, the computations of \cite[\S 6]{flp} imply that applying $N$ Dehn twists along $\lambda$ and taking $N\to \infty$, the resulting laminations $\tau^N(\mathcal{F}^s)$ and $\tau^N(\mathcal{F}^u)$ converge to $[\lambda]$ as elements in $\mathcal{P}\mathcal{M}\mathcal{L}(\Sigma^\circ)$. Thus to prove $\tau^N(\mathcal{F}^s)\not\in \overline{\mathscr{D}}$ and $\tau^N(\mathcal{F}^u)\not\in \overline{\mathscr{D}}$ for large $N$, it suffices to show that $[\lambda]\not\in \overline{\mathscr{D}}$. 

    To this end, consider a sequence $[\delta_1], [\delta_2],\ldots $ in $\mathscr{D}$. As $\lambda$ is not isotopic to any curve in $\mathscr{C}$, we may pass to a subsequence such that each $\delta_i$ is not isotopic to a curve in $\mathscr{C}$. By Lemma \ref{lemma:wave}, each $\delta_i$ has a wave in an interior pair of pants. By passing again to a subsequence, we may assume that all $\delta_i$ have the same wave $w$ in the same interior pair of pants $F\subset \Sigma^\circ$. As $\lambda$ has every seam in every interior pair of pants, so too do the $\delta_i$ for large enough $i$. But the wave $w$ intersects one of the seams in $F$ non-trivially, implying that such $\delta_i$ have non-trivial self-intersection, a contradiction.
\end{proof}

\subsection{An application to the symplectic mapping class group of $M_{g,n}$}\label{sec:mappingclassproof}

\begin{proof}[Proof of Theorem \ref{thm:mappingclassgroupaction}]

	Let $K$ be the intersection of $\text{ker}(\rho_{g,n})$ with the subgroup $\text{Mod}(\Sigma_{g,n})\subset\Gamma_{g,n}$. Then $K$ is a normal subgroup of $\text{Mod}(\Sigma_{g,n})$. By \cite[\S 11.5]{ivanov}, a finite normal subgroup of $\text{Mod}(\Sigma_{g,n})$ lies in the center of $\text{Mod}(\Sigma_{g,n})$, and by \cite[\S 8.15]{ivanov} the center of $\text{Mod}(\Sigma_{g,n})$ is trivial given the assumption that $3g+n>4$ and $n$ is odd. Therefore if $K$ is non-trivial, it must be infinite. Then, by Theorem 1 and Corollary 7.13 of \cite{ivanov}, $K$ contains an element which corresponds to a pseudo-Anosov diffeomorphism $\phi:\Sigma^\circ\to \Sigma^\circ$. Note that in the cited reference, ``pseudo-Anosov'' does not require $\phi$ to be the identity on the boundary, but some power of $\phi$ satisfies this property. Therefore Proposition \ref{prop:highdistance} applies to yield $\psi\in \text{Mod}(\Sigma_{g,n})$ such that
	    \[
        \limsup_{m\to \infty}\;\; d(\mathscr{D}, \; \psi \phi^m \psi^{-1}(\mathscr{D} )) =\infty.
    \]
Proposition \ref{prop:nonsymplecticisotopicformappingclassesviadistance} then implies $\rho_{g,n}(\psi \phi^m \psi^{-1})\neq 1$ for some $m\geq 1$. But since
\[
	\rho_{g,n}(\psi \phi^m \psi^{-1}) = \rho_{g,n}(\psi ) \rho_{g,n}(\phi)^m \rho_{g,n}(\psi )^{-1}
\]
this implies $\rho_{g,n}(\phi)\neq 1$. Thus $K$ is the trivial subgroup of $\text{Mod}(\Sigma_{g,n})$.

It remains to show that intersection of $\text{ker}(\rho_{g,n})$ with the subgroup $H^1(\Sigma^\circ;\Z/2)\subset \Gamma_{g,n}$ is also trivial. Let $\zeta\in H^1(\Sigma^\circ;\Z/2)$ be nonzero. Let $(C,T)$ be a standard pair of type I or II with a choice of $1$-manifold $w$. Choose an embedded $1$-manifold with boundary $w_\zeta \subset \Sigma$ with $\partial w_\zeta = w_\zeta \cap \mathcal{P}$ such that $\zeta(z)=[w_\zeta]\cdot z$ for all $z\in H_1(\Sigma^\circ;\Z/2)$. Then
\[
	\zeta\left( \mathfrak{X}(C,T,w) \right) = \mathfrak{X}(C,T,w \sqcup w_\zeta) 
\]
where $w_\zeta$ is viewed in a copy of $\Sigma$ which is slightly pushed into the interior of $C$. This follows from the description of the $H^1(\Sigma^\circ;\Z/2)$ action on $M_{g,n}$ given in \S \ref{subsec:monotonecaselag}. Choose $w=\emptyset$ and write $\mathfrak{X}(C,T)=\mathfrak{X}(C,T,w)$. 

Similar to the proof of Proposition \ref{prop:vanishingforidentitymappingclass}, using $3g+n>4$ and $\zeta\neq 0$, we may choose $(C,T)$ such that either $w_\zeta$ has odd intersection with the cocore of a $1$-handle, or has an odd number of boundary points on some $\partial_+$-parallel arc of $T$. Write $(Y_{\text{id}},L_{\text{id}})$ for the double of $(C,T)$. The existence of a sphere in $Y_{\text{id}}\smallsetminus L_{\text{id}}$ which has odd intersection with $w_\zeta$ implies $I(Y_{\text{id}},L_{\text{id}},w_\zeta)=0$, or equivalently
\begin{equation}\label{eq:lagrangiandisplacementbyelementofh1}
	\zeta\left( \mathfrak{X}(C,T) \right) \cap \mathfrak{X}(C,T) = \emptyset.
\end{equation}

On the other hand, we claim that $I(Y_{\text{id}},L_{\text{id}}) \neq 0$. First observe that
\begin{equation}\label{eq:doubleidentification}
	(Y_{\text{id}},L_{\text{id}}) \cong \begin{cases} ( \#^{g}S^1\times S^2  \# S^1\times S^2, U \sqcup L) \quad & \textup{(type I)} \\[2mm]
		( \#^{g-1}S^1\times S^2  \# S^1\times S^2\# S^1\times S^2, U' \sqcup L' \sqcup L'') \quad  & \textup{(type II)}  \end{cases}
\end{equation}
In the type I case, $U$ is an $(n-3)/2$ component unlink and $L$ is three parallel copies of $S^1$ in the last $S^1\times S^2$ summand. In the type II case, $U'$ is an $(n-1)/2$ component unlink, $L'$ is a copy of $S^1$ in the last $S^1\times S^2$ summand, and $L''$ is a knot in the last two summands $ S^1\times S^2\# S^1\times S^2$ which is in the homotopy class $aba^{-1}b^{-1}$ in a presentation of $\pi_1( S^1\times S^2\# S^1\times S^2)$, based at a point in the separating $2$-sphere, where $a$ and $b$ are the two $S^1$ factors. These descriptions are easily obtained by taking the doubles of the standard pairs in Figure \ref{fig:standardpieces}, for example. For type I, the link $L$ is formed by the gluing the two triples of vertical arcs in each copy of $(C,T)$. For type II, $L'$ is the gluing of the two vertical arcs and $L''$ that of the two special arcs.

To show $I(Y_{\text{id}},L_{\text{id}}) \neq 0$, by Theorem \ref{thm:detectionintro} it suffices to show that there are no $2$-spheres in $Y_{\text{id}}$ that intersect $L_{\text{id}}$ in a single point. In the type I case, as such a 2-sphere has odd homological pairing with $L_{\text{id}}$, it has odd pairing with each of the three parallel components of $L$, and thus intersects $L_{\text{id}}$ in at least $3$ points, yielding the desired claim. 

Suppose then that $(C,T)$ is type II. Let $S\subset Y$ be a $2$-sphere that intersects $L_\text{id}$ in a single point. Let $S'$ be the $2$-sphere which separates the summand $ \#^{g-1}S^1\times S^2$ from $S^1\times S^2\# S^1\times S^2$ in \eqref{eq:doubleidentification}. Suppose $S$ is chosen among all $2$-spheres which intersect $L_\text{id}$ in one point, with $S\cap S'$ transverse, so that the number of components in $S\cap S'$ is minimal. If $S\cap S'$ is nonempty, choose a component of $S\cap S'$ which bounds a disk in $S'$. Doing surgery of $S$ along this disk yields two new spheres, one of which intersects $L_\text{id}$ in one point and has fewer components of intersection with $S'$ than did $S$, a contradiction. Thus we may assume $S\cap S'$ is empty. 

As $U'$ and $L''$ are nullhomologous, $S$ must intersect the component $L'$. Together with the conclusion of the previous paragraph, the problem is reduced to showing that there are no $2$-spheres in $(S^1\times S^2\# S^1\times S^2, L'\sqcup L'')$ that intersect $L'$ once and miss $L''$. Given such a $2$-sphere $S$, the complement of a regular neighborhood of $S\cup L'$ implies that $L''$ lies in a copy of $S^1\times S^2\smallsetminus B^3$. Since $L''$ is nullhomologous, it must be nullhomotopic in $S^1\times S^2\smallsetminus B^3$. This contradicts the fact that $L''$ is in the non-trivial homotopy class $aba^{-1}b^{-1}$ in $\pi_1( S^1\times S^2\# S^1\times S^2)$.

Having shown $I(Y_{\text{id}},L_{\text{id}}) \neq 0$, we conclude from Theorem \ref{thm:lagrangianineq} that 
\[
		\mathfrak{n}(\mathfrak{X}(C,T),\mathfrak{X}(C,T)) > 0.
\]
Together with \eqref{eq:lagrangiandisplacementbyelementofh1}, we obtain that $\mathfrak{X}(C,T)$ and $\zeta(\mathfrak{X}(C,T))$ are not Hamiltonian isotopic, and thus the symplectomorphism on $M_{g,n}$ induced by $\zeta$ is not Hamiltonian isotopic to the identity. This completes the argument that the intersection of $\text{ker}(\rho_{g,n})$ with the subgroup $H^1(\Sigma^\circ;\Z/2)\subset \Gamma_{g,n}$ is the trivial group, and completes the proof of the theorem.
\end{proof}

%!TEX root = main.tex

\section{The Lagrangian sphere complex of the del Pezzo surface $\text{Bl}_5\mathbb {CP}^2$}\label{sec:lagspherecomp}

In this section, we study Lagrangian spheres in the symplectic manifold $M_{0,5}$. The main theme of this section is a connection between the intersection theory of Lagrangian spheres in $M_{0,5}$, modulo symplectic isotopy, with the intersection theory of essential simple closed curves in $\Sigma_{0,5}$, considered up to isotopy. As pointed out in the introduction, $M_{0,5}$ is symplectomorphic to the monotone del Pezzo surface $\textup{Bl}_5\mathbb {CP}^2$, the projective plane blown up at $5$ generic points. Specifically, a diffeomorphism between $M_{0,5}$ and $\textup{Bl}_5\mathbb {CP}^2$ follows from \cite{kirk-klassen}, as explained in \cite{kirk}, and by \cite[Theorem 1.2]{mcduff}, any two symplectic forms $\omega_1$ and $\omega_2$ on $\textup{Bl}_5\mathbb {CP}^2$ satisfying $[\omega_i] = c_1(\textup{Bl}_5\mathbb {CP}^2)$ are symplectomorphic. The cohomology class $c_1(\textup{Bl}_5\mathbb {CP}^2)$ is given by
\begin{equation}\label{symplectic-cohom-class}
  [\omega]=3H-E_1+\cdots-E_5, 
\end{equation}
where $H$ is the hyperplane class and $E_i$ are the exceptional classes of the blowups.

Let $\gamma$ be an essential, non-peripheral simple closed curve in $\Sigma_{0,5}$.  Attaching a $2$-handle to a thickened $\Sigma_{0,5}$ along $\gamma$ produces a $3$-manifold diffeomorphic to a $3$-ball with two disjoint open $3$-balls removed. The $5$ marked points determine a tangle in this manifold whose three boundary components contain five, three, and two marked points, respectively. Cap off the latter boundary sphere by gluing in a ball with a $\partial_+$-parallel arc, to obtain the pair $(C_\gamma,T_\gamma)$. Then $C_\gamma$ is the compression body  $S^2\times [0,1]$ and $T_\gamma$ is a trivial tangle with one $\partial_+$-parallel boundary component and three vertical boundary components. In particular, $(C_\gamma,T_\gamma)$ is a standard $(0,5)$ pair of type I. Thus to the curve $\gamma$ in $\Sigma_{0,5}$ we have an associated Lagrangian
\[
  \Lambda_\gamma^+ : =\mathfrak{X}(C_\gamma,T_\gamma) \subset M_{g,n}.  
\]
Let $w_\gamma$ be an embedded compact 1-manifold in $C_\gamma$ with $\partial w_\gamma=w_\gamma \cap T_\gamma$ the union of a point on the $\partial_+$-parallel boundary component of $T_\gamma$ and a point on one of the vertical boundary components. From this we obtain another Lagrangian associated to the curve $\gamma$:
\[
 \Lambda_\gamma^- :=\mathfrak{X}(C_\gamma,T_\gamma,w_\gamma) \subset M_{g,n}.
\]
From \eqref{eq:lagrangianidentification}, we see that $\Lambda^\pm_\gamma$ are Lagrangian spheres. Explicitly, applying a diffeomorphism of $\Sigma_{0,5}$, we may assume that $\gamma$ is as in Figure \ref{gamma}. With respect to a standard presentation of $\pi_1(S^2\smallsetminus \mathcal P)$ with generators $c_1,\ldots,c_5$ and relation $c_1c_2c_3c_4c_5=1$, the moduli space $M_{0,5}$ consists of the conjugacy classes of tuples $(C_1,\dots,C_5)$ of traceless elements in $SU(2)$ with $C_1C_2 C_3 C_4C_5=1$, and $\Lambda_\gamma^{\pm}$ corresponds to the subspace defined by the conditions
\begin{equation}\label{eq:lagrangiansphereconcrete}
	C_1 C_2 = \pm 1 , \qquad C_3 C_4 C_5 = \pm 1.
\end{equation}
Note that although the choice of $w_\gamma$ is not unique, $\Lambda_\gamma^-$ is independent of this choice. Any given $[C_1,\ldots, C_5]\in \Lambda^{\pm}_\gamma$ has a unique representative with 
\[
	C_3 = \mp\left[\begin{array}{cc} i & 0 \\ 0 & -i \end{array}\right], \quad C_4 =  \left[\begin{array}{cc} 0 & 1 \\ -1 & 0 \end{array}\right], \quad  C_5 =\left[\begin{array}{cc} 0  & i  \\ i & 0  \end{array}\right],  
\]
and where $C_1=\mp C_2$ is any traceless element. Note that $\Lambda_\gamma^+$ and $\Lambda_\gamma^-$ are disjoint.

A {\emph{signed curve}}  in $\Sigma_{0,5}$ is a pair $(\gamma,\epsilon)$ of an essential simple closed curve $\gamma$ in $\Sigma_{0,5}$ and a choice of sign $\epsilon\in \{-,+\}$. Signed curves $(\gamma_1,\epsilon_1)$ and $(\gamma_2,\epsilon_2)$ are {\emph{isotopic}} if $\gamma_1$ and $\gamma_2$ are isotopic and $\epsilon_1=\epsilon_2$. The above construction associates to a signed curve $(\gamma,\epsilon)$ in $\Sigma_{0,5}$ a Lagrangian sphere $\Lambda_\gamma^\epsilon$ in $\smash{M_{0,5}=\textup{Bl}_5\mathbb {CP}^2}$ which clearly only depends on the isotopy class of $(\gamma,\epsilon)$.

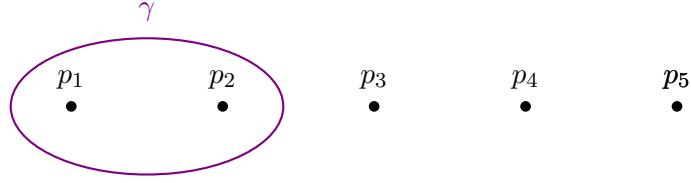
\begin{figure}
\begin{center}
\begin{tikzpicture}
  \foreach \x in {0,2,4,6,8} {
    \fill (\x,0) circle (2pt);
  }
  \draw[violet, thick] (1,0) ellipse (1.8 and 0.9);
\node[above] at (0,0.1) {$p_1$};
\node[above] at (2,0.1) {$p_2$};
\node[above] at (4,0.1) {$p_3$};
\node[above] at (6,0.1) {$p_4$};
\node[above] at (8,0.1) {$p_5$};
\node[above] at (8,0.1) {$p_5$};
\node[above,violet] at (1,1) {$\gamma$};
\end{tikzpicture}
\end{center}
\caption{After identifying $S^2$ with the plane together with the point at infinity, and placing the marked points at the locations $p_i$ shown above, any essential curve may be mapped to the curve $\gamma$ by a suitable diffeomorphism of $\Sigma_{0,5}$.}
\label{gamma}
\end{figure}

\begin{theorem}\label{displace-characterization}
	Two Lagrangians $\Lambda_{\gamma_1}^{\epsilon_1}$ and $\Lambda_{\gamma_2}^{\epsilon_2}$ associated to non-isotopic $(\gamma_1,\epsilon_1)$ and $(\gamma_2,\epsilon_2)$ 
	can be displaced from each other by a Hamiltonian isotopy if and only if the simple closed curves $\gamma_1$ and $\gamma_2$ can be made disjoint from each other by 
	isotopy.
\end{theorem}

In preparation for the proof we have the following. Recall the notation from \S \ref{sec:heegaard}. 

\begin{lemma} \label{lemma:hamiltoniandisplacementviacurvesets}
For the compression body with trivial tangle $(C_\gamma, T_\gamma)$, we have:
	\begin{enumerate}[label=(\roman*)]
		\item $\mathscr{D}(C_\gamma,T_\gamma) = \{ \gamma \}$. \label{item:lemmadiskann1}
		\item Every $\gamma' \in  \mathscr{A}(C_\gamma, T_\gamma)$ may be displaced from $\gamma$. \label{item:lemmadiskann2}
	\end{enumerate}
\end{lemma}

\begin{proof}
	Recalling $\Sigma_{0,5} = (S^2, \mathcal{P})$ where $\mathcal{P}=\{p_1,\ldots,p_5\}$, note
	\[
	  \pi_1(S^2\smallsetminus \mathcal P)\cong F_4,\hspace{1cm}\pi_1(C_\gamma\smallsetminus T_\gamma)\cong F_3,
	\]
	where $F_n$ is the free group with $n$ generators. Let $\gamma'$ be an essential simple closed curve that is nullhomotopic in $C_\gamma\smallsetminus T_{\gamma}$ and pick 
	orientations for both of $\gamma$, $\gamma'$. Then $\gamma'$ determines a conjugacy class in $\pi_1(S^2\smallsetminus \mathcal{P})$ such that the inclusion of 
	$S^2\smallsetminus \mathcal{P}$ into $C_\gamma\smallsetminus T_\gamma$ maps this conjugacy class to the identity class in 
	 $\pi_1(C_\gamma\smallsetminus T_\gamma)$. Since the fundamental group of $C_\gamma\smallsetminus T_\gamma$ is the quotient of the free group $\pi_1(S^2\smallsetminus \mathcal{P})$ 
	 by the normal closure $\langle\!\langle \gamma \rangle\!\rangle$ determined by $\gamma$, the conjugacy class of $\gamma'$ belongs to 
	 $\langle\!\langle \gamma \rangle\!\rangle$. Therefore, the identity map of $\pi_1(S^2\smallsetminus \mathcal{P})$ induces a surjection 
	 \[
	   \pi_1(S^2\smallsetminus \mathcal{P})/\langle\!\langle \gamma' \rangle\!\rangle \to \pi_1(S^2\smallsetminus \mathcal{P})/\langle\!\langle \gamma \rangle\!\rangle.
	 \]
	 The domain and the codomain of this map are isomorphic to $F_3$, and the Hopfian property of free groups implies that this map is an isomorphism, hence 
	 $\langle\!\langle \gamma \rangle\!\rangle =\langle\!\langle \gamma' \rangle\!\rangle $. Using  \cite{Magnus}, we conclude that $\gamma'$ is conjugate to $\gamma$ or 
	 $\gamma^{-1}$. 
	 In particular, the underlying simple closed curves are isotopic to each other. It follows that $\mathscr{D}(C_\gamma,T_\gamma)=\{\gamma\}$, proving \ref{item:lemmadiskann1}.
	 
	 	 Next, let $\gamma'\in \mathscr{A}(C_\gamma,T_\gamma)$. Let $\gamma_i$ be a small loop around the marked point $p_i$.
	 By definition of $\mathscr{A}(C_\gamma,T_\gamma)$, there is a disk $D\subset  C_\gamma$ with $\partial D=\gamma'$ that intersects $T_\gamma$ transversely 
	 in one point, and we assume $p_i$ is one of the endpoints of the component of $T_\gamma$ intersecting $D$. Note that in the case of Figure \ref{gamma}, the point $p_i$ must be one of $p_3,p_4,p_5$. Cut out a small disk in $D$ which intersects $T_\gamma$ and replace it with a tube that runs down to $S^2$, yielding an annulus in $C_\gamma\smallsetminus T_\gamma$ with boundary components $\gamma'$ and $\gamma_i$. Do surgery on this annulus along a disk cobounding an arc on $S^2\smallsetminus \mathcal{P}$ and an arc on the annulus, each running from $\gamma'$ to $\gamma_i$. The result is a possibly immersed disk whose boundary is a band sum $\gamma'\#\gamma_i$. As this curve is essential in $S^2\smallsetminus \mathcal{P}$ and is null-homotopic in $C_\gamma\smallsetminus T_\gamma$, item \ref{item:lemmadiskann1} implies $\gamma'\#\gamma_i$ is isotopic to $\gamma$. From this, \ref{item:lemmadiskann2} follows.
\end{proof}

\begin{proof}[Proof of Theorem \ref{displace-characterization}]
	Suppose Lagrangians $\Lambda_{\gamma_1}^{\epsilon_1}$ and $\Lambda_{\gamma_2}^{\epsilon_2}$ for non-isotopic $(\gamma_1,\epsilon_1)$ and $(\gamma_2,\epsilon_2)$ 
	can be displaced by Hamiltonian isotopy. Let $(Y,L,w)$ be the triple obtained from gluing $(C_{\gamma_1},T_{\gamma_1},w_1)$ and 
	$(C_{\gamma_2},T_{\gamma_2},w_2)$ along their common boundaries, where $w_i=\emptyset$ if $\epsilon_i=+$ and $w_i=w_{\gamma_i}$ if $\epsilon_i=-$.
	Theorem \ref{thm:lagrangianineq} implies that $I(Y,L,w)=0$, and then by Proposition \ref{prop:instantonheegaarddetection} we have 
	\[
	  \left(\mathscr{D}_1\cap \mathscr{A}_2\right) \cup \left(\mathscr{A}_1\cap \mathscr{D}_2\right) \cup \left( \mathscr{D}^0_1\cap \mathscr{D}^1_2\right) \cup 
	  \left( \mathscr{D}^1_1\cap \mathscr{D}^0_2\right) \neq \emptyset
	\]
	where $\mathscr{D}_i=\mathscr{D}(C_i,T_i)$, $\mathscr{D}^\varepsilon_i=\mathscr{D}^\varepsilon(C_i,T_i,w_i)$, and $\mathscr{A}_i=\mathscr{A}(C_i,T_i)$. If
	$\mathscr{D}_1\cap \mathscr{A}_2$ is non-empty, then Lemma \ref{lemma:hamiltoniandisplacementviacurvesets} implies that $\gamma_1$ can be made disjoint from 
	$\gamma_2$ by an isotopy. We obtain a similar conclusion if $\mathscr{A}_1\cap \mathscr{D}_2$ is non-empty. If $\mathscr{D}^0_1\cap \mathscr{D}^1_2$ or 
	$\mathscr{D}^1_1\cap \mathscr{D}^0_2$ is non-empty, then Lemma \ref{lemma:hamiltoniandisplacementviacurvesets} implies $\gamma_1$ is isotopic to $\gamma_2$. In particular, $\gamma_1$ can be made disjoint from $\gamma_2$ 
	by an isotopy. 
	
	In the reverse direction, if  $(\gamma_1,\epsilon_1)$ and $(\gamma_2,\epsilon_2)$ are non-isotopic and $\gamma_1$ and $\gamma_2$ can be made disjoint from each other, then
	either $\gamma_1$ and $\gamma_2$ are isotopic to each other and $\epsilon_1\neq \epsilon_2$ or $\gamma_1$ and $\gamma_2$ are non-isotopic to each other. 
	In the former case, either $\mathscr{D}^0_1\cap \mathscr{D}^1_2$ or $\mathscr{D}^1_1\cap \mathscr{D}^0_2$ contains $\gamma_1=\gamma_2$, and hence there is 
	an embedded sphere $S$ in $Y$ that is disjoint from $L$ and $S\cdot w$ is odd. In particular, the character variety 
	$\mathfrak{X}(Y,L,w)=\Lambda_{\gamma_1}^{\epsilon_1}\cap \Lambda_{\gamma_2}^{\epsilon_2}$ is empty. 
	In the latter case that $\gamma_1$ and $\gamma_2$ are non-isotopic and they can be made disjoint from each other by isotopy, either $\mathscr{D}_1\cap \mathscr{A}_2$
	or $\mathscr{A}_1\cap \mathscr{D}_2$ is non-empty. In particular, there is an embedded sphere $S$ in $Y$ that intersects $L$ transversely in exactly one point. 
	This again implies $\mathfrak{X}(Y,L,w)=\Lambda_{\gamma_1}^{\epsilon_1}\cap \Lambda_{\gamma_2}^{\epsilon_2}$ is empty. 
\end{proof}

The symplectic mapping class group of the del Pezzo surface $\textup{Bl}_5\mathbb {CP}^2$ has been studied extensively. In particular, the isomorphism type of this group can be recovered from the results of \cite{evans-stein}. The basis $\{H, E_1,\cdots,E_5\}$ of $H_2(M_{0,5};\Z)$ consists of orthogonal elements with respect to the intersection form. Furthermore, $H\cdot H=1$ and $E_i\cdot E_i=-1$. Any symplectomorphism $\phi:M_{0,5}\to M_{0,5}$ determines an action on $H_2(M_{0,5};\Z)$ that preserves the intersection form and acts identically on the cohomology class $[\omega]$ given in \eqref{symplectic-cohom-class}.
In particular, any such $\phi$ determines an automorphism of the negative definite lattice $[\omega]^{\perp}\subset H_2(M_{0,5};\Z)$. This lattice is isomorphic to negative definite root lattice of type $D_5$, as exhibited by the basis in Figure \ref{five-lagrangians-homology}.
	\begin{figure}
\begin{center}
\begin{tikzpicture}[scale=1.2]
\fill (0,0) circle (2pt);
\fill (1.8,0) circle (2pt);
\fill (3.6,0) circle (2pt);
\fill (5,1) circle (2pt);
\fill (5,-1) circle (2pt);
\draw (0,0) -- (1.8,0);
\draw (1.8,0) -- (3.6,0);
\draw (3.6,0) -- (5,1);
\draw (3.6,0) -- (5,-1);
\node[above] at (0,0.1) {$E_1-E_2$};
\node[above] at (1.8,0.1) {$E_2-E_3$};
\node[above] at (4.54,-0.22) {$E_3-E_4$};
\node[above] at (5,1.1) {$E_4-E_5$};
\node[above] at (5,-1.66) {$H-E_1-E_2-E_3$};
\end{tikzpicture}
\end{center}
\caption{ A basis for the lattice $[\omega]^{\perp}$ realizing the $D_5$ Dynkin diagram }
\label{five-lagrangians-homology}
\end{figure}
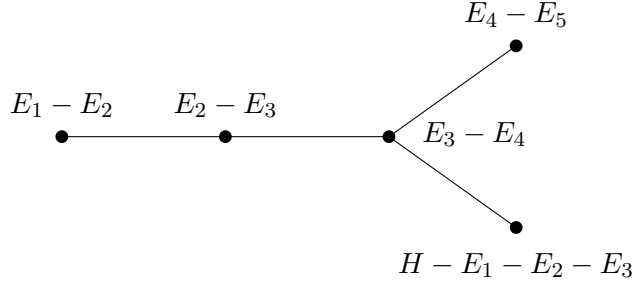

The group $W(D_5)$ of isometries of the lattice $[\omega]^\perp$ is isomorphic to $S_5\ltimes H^1(\Sigma^{\circ};\Z/2)$. In particular, we have a homomorphism $\pi_0{\rm Symp}(M_{0,5})\to W(D_5)$, which is surjective, as follows from \cite{seidel-4d} or the proof of Proposition \ref{rho-0-5-iso} below. Recall that the kernel of this homomorphism, denoted by $\pi_0{\rm Symp}_h(M_{g,n})$, is the group of symplectomorphisms acting trivially on $H_2(M_{0,5};\Z)$. This group is isomorphic to the pure mapping class group $ {\rm PMod}(\Sigma_{0,5})$, see \cite{evans-stein}.

\begin{prop}\label{rho-0-5-iso}
	The map $\rho_{0,5}: \Gamma_{0,5} \to \pi_0{\rm Symp}(M_{0,5})$ is an isomorphism.
\end{prop}
\begin{proof}
	In the previous section, we showed that $\rho_{0,5}: \Gamma_{0,5}\to\pi_0{\rm Symp}(M_{g,n}) $ is injective, 
	and by Proposition \ref{prop:puremappingclassgroup}, this action maps the subgroup ${\rm PMod}(\Sigma_{0,5})$ of 
	$\Gamma_{0,5}$ into $\pi_0{\rm Symp}_h(M_{0,5})\cong {\rm PMod}(\Sigma_{0,5})$. 
	This implies that the restriction of $\rho_{0,5}$ to ${\rm PMod}(\Sigma_{0,5})$  is an isomorphism because ${\rm PMod}(\Sigma_{0,5})$ is co-Hopfian 
	\cite[Corollary 3]{BM:Artin-coHopfian}. 
	Thus $\rho_{0,5}$ induces an injective map 
	\[
	  \Gamma_{0,5}/{\rm PMod}(\Sigma_{0,5}) \to \pi_0{\rm Symp}(M_{0,5})/\pi_0{\rm Symp}_h(M_{0,5}).
	\]
	Since the domain and the codomian of this map are both isomorphic to the finite group $S_5\ltimes (\Z/2)^4$, this map has to be an isomorphism. 
	In particular, $\rho_{0,5}$ is an isomorphism. 
\end{proof}

\begin{remark}
	Given an essential non-peripheral simple closed curve $\gamma$, let $\tau_\gamma:\Sigma_{0,5} \to \Sigma_{0,5}$ be the diffeomorphism given by half twist along $\gamma$. 
	Then the symplectomorphism $\rho_{0,5}(\tau_\gamma)$ is given by {\it symplectic Dehn twist} along the Lagrangian sphere $\Lambda_\gamma^+$.
	This observation goes back to Seidel and it is discussed in detail in \cite[Theorem 3.11]{WW:ex-tri-Dehn-twist}. 
\end{remark}
\begin{prop}\label{transitive-action-Lag}
	$\Gamma_{0,5}$ acts transitively on symplectic isotopy classes of Lagrangian spheres in $M_{0,5}$. 
\end{prop}
\begin{proof}
	For any Lagrangian sphere $L$ in $M_{0,5}$, the homology class $\zeta$ of $L$ has self-intersection $-2$ and pairs trivially with \eqref{symplectic-cohom-class}.
	In particular, $\zeta$ belongs to the lattice of $D_5$ described above.
	The action of the Weyl group $W(D_5)$ on the lattice is transitive on the points with self-intersection $-2$. This implies that starting with a fixed 
	$(\gamma_0,\epsilon_0)$, there is a symplectomorphism $\phi:M_{0,5} \to M_{0,5}$ such that $\phi(\Lambda_{\gamma}^{\epsilon})$ is in the same homology class as $\zeta$. 
	Now the result follows from \cite[Corollary 1.2]{borman-li-wu} which states that the action of $\pi_0{\rm Symp}_h(M_{0,5})$ on symplectic isotopy classes of Lagrangian 
	spheres in the same homology class is transitive.
\end{proof}

The homological $D_5$ configuration of Figure \ref{five-lagrangians-homology} may be realized geometrically as follows. First, suppose $\gamma_1$ and $\gamma_2$ are essential non-peripheral simple closed curves in $\Sigma_{0,5}$ with minimal geometric intersection number $2$. After a diffeomorphism, we may assume $\gamma_1$ and $\gamma_2$ are as in Figure \ref{five-lagrangians}. Choose signs $\epsilon_1$ and $\epsilon_2$. Then, using the description of $\Lambda_{\gamma_1}^{\epsilon_1}$ given in \eqref{eq:lagrangiansphereconcrete}, and the analogous description of $\Lambda_{\gamma_2}^{\epsilon_2}$, we obtain that $\Lambda_{\gamma_1}^{\epsilon_1} \cap \Lambda_{\gamma_2}^{\epsilon_2} $ intersects transversely in a single point represented by 
\[
	C_1= -\epsilon_1 C_2 =  \epsilon_1 \epsilon_2 C_3 = \left[\begin{array}{cc} i & 0 \\ 0 & -i \end{array}\right], \quad C_4 =  -\epsilon_2  \left[\begin{array}{cc} 0 & 1 \\ -1 & 0 \end{array}\right], \quad C_5 =  \left[\begin{array}{cc} 0 & i \\ i & 0 \end{array}\right].
\]
It follows that the Lagrangian spheres associated to the signed curves in Figure \ref{five-lagrangians} form a $D_5$ configuration. That is, after an automorphism of the lattice, the homology classes of the associated Lagrangian spheres align with the vertices in Figure \ref{five-lagrangians-homology}, and two Lagrangian spheres intersect transversely in a single point if their vertices are connected by an edge, and are otherwise disjoint.

	\begin{figure}
\begin{center}
\begin{tikzpicture}
  \foreach \x in {0,2,4,6,8} {
    \fill (\x,0) circle (2pt);
  }
  \draw[magenta, thick] (1,0) ellipse (1.4 and 0.9);
    \draw[green!65!black, thick] (3,0) ellipse (1.4 and 0.9);
    \draw[red, thick] (5,0) ellipse (1.4 and 0.9);
    \draw[violet, thick] (7,0) ellipse (1.68 and 1.43);
    \draw[blue, thick] (7,0) ellipse (1.4 and 0.65);    
\node[magenta] at (1,1.2) {$(\gamma_1,+)$};
\node[green!65!black] at (3,1.2) {$(\gamma_2,+)$};
\node[red] at (5,1.2) {$(\gamma_3,+)$};
\node[violet] at (7,1.74) {$(\gamma_4,+)$};
\node[blue] at (7,0.95) {$(\gamma_5,-)$};
\end{tikzpicture}
\end{center}
\caption{A family of Lagrangians spheres inducing the graph $D_5$}
\label{five-lagrangians}
\end{figure}
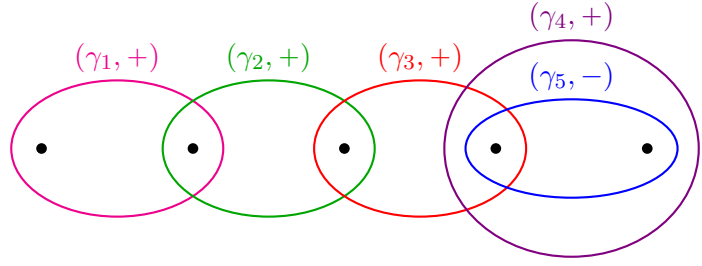

Recall from the introduction the definitions of the signed curve complex $\widetilde{\mathcal{C}}(\Sigma_{0,5})$ and the Lagrangian sphere complex $\mathcal L(M_{0,5})$. The association of the Lagrangian sphere $\Lambda_{\gamma}^\epsilon$ to a signed curve $(\gamma,\epsilon)$ defines a map on the vertices of these simplicial complexes:
\[
	F: \widetilde{\mathcal{C}}(\Sigma_{0,5})^{(0)} \to  \mathcal L(M_{0,5})^{(0)}
\]
There is a natural simplicial action of $\pi_0\text{Symp}(M_{0,5})$ on $\mathcal L(M_{0,5})$. The natural action of ${\rm Mod}(\Sigma_{0,5})$ on $\widetilde{\mathcal{C}}(\Sigma_{0,5})$ extends to a simplicial $\Gamma_{0,5}$-action, where $(\phi,\zeta) \in \Gamma_{0,5}={\rm Mod}(\Sigma_{0,5})  \ltimes H^1(\Sigma^\circ ;\mathbb Z/2)$ maps a vertex $(\gamma,\epsilon)$ to $(\varphi_*(\gamma), \epsilon\cdot \zeta(\gamma))$, where  $\zeta(\gamma)$ is the pairing of the cohomology class $\zeta$ and the homology class of $\gamma$. These actions on the set of vertices of $\widetilde{\mathcal{C}}(\Sigma_{0,5})$ and $\mathcal L(M_{0,5})$ are intertwined via $\rho_{0,5}$ and the map $F$. The following implies Theorem \ref{thm:lagrangianspheres}.

\begin{theorem}\label{thm:lagrangianspheres-gen}
	There is a $\Gamma_{0,5}$-equivariant simplicial isomorphism $\widetilde {\mathcal C}(\Sigma_{0,5}) \cong \mathcal L(M_{0,5})$.
\end{theorem}
\begin{proof}
	First we show that $F$ is an isomorphism. Since $F$ is $\Gamma_{0,5}$-equivariant and the action of $\Gamma_{0,5}$
	on the set of vertices of $\mathcal L(M_{0,5})$ is transitive by Proposition \ref{transitive-action-Lag}, the map $F$ is surjective. 
	To show injectivity, let $(\gamma,\epsilon)$ and  $(\gamma',\epsilon')$ be mapped to two Lagrangians in the same symplectic isotopy class. 
	Since $\Lambda_{\gamma}^{\epsilon}$ and $\Lambda_{\gamma}^{-\epsilon}$ can be displaced, we may displace $\Lambda_{\gamma}^{-\epsilon}$ and 
	$\Lambda_{\gamma'}^{\epsilon'}$. Therefore, Theorem \ref{displace-characterization} implies that $\gamma$ and $\gamma'$ 
	can be made disjoint from each other by an isotopy. 
	Another application of Theorem \ref{displace-characterization} shows that if $\gamma$ and $\gamma'$ are not isotopic to each other, then $\Lambda_{\gamma}^{\epsilon}$ and 
	$\Lambda_{\gamma'}^{\epsilon'}$ can be displaced from each other. 
	But this is impossible because these two Lagrangians are isotopic to each other and they have self-intersection
	$-2$. Thus $\gamma$ and $\gamma'$ are isotopic. Now if $\epsilon\neq \epsilon'$, then $\Lambda_{\gamma}^{\epsilon}$ and 
	$\Lambda_{\gamma'}^{\epsilon'}$ can be made disjoint from each other, which is another contradiction. We conclude that  $(\gamma,\epsilon)$ and $(\gamma',\epsilon')$
	represent the same vertices in $\widetilde {\mathcal C}(\Sigma_{0,5})$. Thus $F$ is an isomorphism.
	
	To show that $F$ can be extended to an isomorphism of simplicial complexes, we need to show that distinct vertices 
	$(\gamma_0,\epsilon_0), \ldots, (\gamma_n,\epsilon_n)$ span an 
	$n$-simplex in $\widetilde {\mathcal C}(\Sigma_{0,5})$ if and only if the Lagrangian spheres $\Lambda_{\gamma_0}^{\epsilon_0}$, 
	$\cdots$, $\Lambda_{\gamma_n}^{\epsilon_n}$ span an $n$-simplex in $\mathcal L(M_{0,5})$. For $n=1$, this is proved in Theorem \ref{displace-characterization}. 
	For $n>1$, we observe that $\smash{\widetilde {\mathcal C}(\Sigma_{0,5})}$ and $\mathcal L(M_{0,5})$ are both flag simplicial complexes: a set of vertices in 
	these complexes span a simplex if and only if every pair of such vertices are connected by an edge. In the case of $\mathcal L(M_{0,5})$, this property is already 
	established in the proof of Theorem \ref{displace-characterization}. 
	Indeed, the proof shows that if two Lagrangian spheres $\Lambda_{\gamma}^\epsilon$ and $\Lambda_{\gamma'}^{\epsilon'}$ can be displaced by 
	a symplectic isotopy, then they are already disjoint without applying a non-trivial isotopy. This property for $\smash{\widetilde {\mathcal C}(\Sigma_{0,5})}$
	follows from standard properties of simple closed curves on a punctured surface. 
\end{proof}

Finally, Corollary \ref{cor:lagrangianspheres} follows from Theorem \ref{thm:lagrangianspheres} and the following.

\begin{lemma}
	There is a homotopy equivalence $\smash{\widetilde {\mathcal C}(\Sigma_{g,n})} \simeq \smash{ {\mathcal C}(\Sigma_{g,n})}$.
\end{lemma}

\begin{proof}
	For a simplex $\Delta$ in $\smash{ {\mathcal C}(\Sigma_{g,n})}$ spanned by $k+1$ vertices $\gamma_0,\ldots,\gamma_k$ we denote by $\Delta_\star$ the simplex in $\smash{\widetilde {\mathcal C}(\Sigma_{g,n})}$ spanned by the $2k+2$ vertices $(\gamma_0,\pm),\ldots,(\gamma_k,\pm)$. The geometric realization of $\Delta_\star$ is the subset of the real vector space with basis the vertices of $\Delta_\star$ given by linear combinations
	\[
		\sum t_i^{\epsilon} (\gamma_i,\epsilon)
	\]
	where the sum is over $0\leq i\leq k$ and $\epsilon\in \{+,-\}$, and $\sum t_i^\epsilon = 1$,  $t_i^\epsilon\geq 0$. A deformation retraction $f:[0,1]\times\Delta_\star\to \Delta_\star$ onto the simplex spanned by the vertices $(\gamma_0,+),\ldots,(\gamma_k,+)$ is defined by
	\[
		f( s, \sum t_i^{\epsilon} (\gamma_i,\epsilon) ) =   \sum ( t_i^+ + s t_i^-) (\gamma_i,+) + (1-s) \sum t_i^- (\gamma_i,-) 
	\]
	Noting $(\Delta \cap \Delta')_\star =  \Delta_\star \cap \Delta'_\star$, we see that these deformation retractions combine to define a deformation retraction of $\smash{\widetilde {\mathcal C}(\Sigma_{g,n})}$ onto the subcomplex spanned by vertices with sign $+$, and the latter may be identified with the ordinary curve complex $\smash{ {\mathcal C}(\Sigma_{g,n})}$.
\end{proof}

For an alternative and more symmetric viewpoint, note that there is a simplicial involution
\[
	\sigma :\widetilde {\mathcal C}(\Sigma_{g,n})\to \widetilde {\mathcal C}(\Sigma_{g,n})
\]
induced by the map on vertices $\sigma(\gamma,\epsilon)=(\gamma,-\epsilon)$. The fixed point set of $\sigma$, while not a subcomplex of $\widetilde {\mathcal C}(\Sigma_{g,n})$, is homeomorphic to the ordinary curve complex $\smash{ {\mathcal C}(\Sigma_{g,n})}$. An argument similar to the one above shows that $\smash{\widetilde {\mathcal C}(\Sigma_{g,n})}$ deformation retracts onto the fixed point set of $\sigma$.

%!TEX root = main.tex

\section{Intersections of Lagrangians in $M_{g,n}$ and $SU(2)$-simple knots}\label{sec:intersectionsandsu2simple}

A knot $K\subset S^3$ is called {\emph{non-degenerate $SU(2)$-simple}} if every element of the traceless $SU(2)$ character variety $\mathfrak{X}(S^3,K)$ is non-degenerate and has binary dihedral image.  In \S \ref{subsec:stabheegaardint}, given a Heegaard decomposition of $(S^3,K)$ along $\Sigma_{g,n}$, we will associate Lagrangians $\Lambda_1$ and $\Lambda_2$ in $M_{g,n+1}$, and in the case that $K$ is non-degenerate $SU(2)$-simple, we compute the minimal intersection number of these Lagrangians, see Theorem \ref{thm:minimalintsu2simple}. In \S \ref{2-birdge-int-number}, we apply these methods to the case of $M_{0,5}$ to prove Theorem \ref{thm:preciseintersectionresult}. In the course of the proof of this theorem, we require a relationship between two different versions of singular instanton homology, which is established in \S \ref{con-sum-red-I-natural}.

\subsection{Connected sums and $I^\natural(Y,L)$} \label{con-sum-red-I-natural}
Following \cite{km-unknot}, for a given (not necessarily admissible) triple $(Y,L,w)$ and a choice of basepoint $p$ on the link $L$, define the reduced singular instanton homology group
\[
	I^\natural(Y,L) = I(Y, L\# H, w\sqcup w_H)
\]
where, as before, $H$ is a Hopf link and $w_H$ is an arc connecting the two components of $H$; the connected sum is taken at the chosen basepoint $p$ of $L$ with one of the components of $H$. This instanton homology group by definition is equal to $I(Y,L^\pi,w^\pi)$, introduced in \S \ref{sec:detection}, with $\pi=\{p\}$. The constructions of this paper are related to $I^\natural(Y,L)$ via the following.

\begin{prop}\label{prop:connectedsumtheoremforinatural}
	Let $(Y,L,w)$ be an admissible link, with a chosen basepoint on $L$. Then
	\[
		I( (Y,L,w) \# S^1\times \Sigma_{0,3} ) \cong I^\natural(Y,L,w) \oplus I^\natural(Y,L,w)
	\]
	 as $\Z/2$-graded abelian groups, where the connected sum is taken with respect to the chosen basepoint on $L$ and any basepoint on $S^1\times \{p_1,p_2,p_3\}$. 
	The same claim holds if $Y$ is an integer homology sphere, $L$ is a link with non-zero determinant and $w$ is empty.
\end{prop}

\begin{proof}
	For the proof of the claim in the case that $(Y,L,w)$ is admissible or a knot in an integer homology sphere, we apply the connected sum theorem \cite[Theorem 8.8]{ds1}. We refer to that reference for the necessary background on $\mathcal{S}$-complexes. The same argument applied to the construction of \cite{ds3} treats the case that $L$ is a non-zero determinant in an integer homology sphere. To apply the connected sum theorem, we must describe the $\mathcal{S}$-complex $\widetilde C(S^1\times \Sigma_{0,3} )$. This may be done using some computations from \cite{km-relations}, which we now recall.

		The admissible link $S^1\times \Sigma_{0,3} $ has two irreducible critical points for the unperturbed Chern--Simons functional. Indeed, its traceless $SU(2)$ character variety consists of two copies of $M_{0,3}$ (each a point), distinguished by whether the holonomy around the $S^1$-factor is $\pm 1$. The Floer gradings of these generators differ by $2\pmod{4}$, and thus after choosing an absolute $\Z/4$-grading, the irreducible singular instanton chain complex may be written
	\[
		C(S^1\times \Sigma_{0,3} ) = \Z_{(0)} \oplus \Z_{(2)}
	\]
	where subscripts indicate $\Z/4$ gradings. The differential $d$ of this complex vanishes for degree reasons. To determine the $\mathcal{S}$-complex $\widetilde C(S^1\times \Sigma_{0,3} )$, it remains to determine the $v$-map $v:C(S^1\times \Sigma_{0,3} )\to C(S^1\times \Sigma_{0,3} )$, which is degree $2\pmod{4}$. A priori, $v$ depends on a choice of basepoint in $S^1\times \{p_1,p_2,p_3\}$. Up to a nonzero constant, it is equal to one of the operators $\delta_i$ in \cite{km-relations}, where $i\in \{1,2,3\}$ records which component of $S^1\times \{p_1,p_2,p_3\}$ contains the basepoint.

 Following \cite[\S 7.4]{km-relations} (with trivial local coefficient systems, setting the variables $\tau_i=1$), in addition to $\delta_1,\delta_2,\delta_3$, 
	there are operators $\alpha,\varepsilon$ acting on irreducible instanton homology $I(S^1\times \Sigma_{0,3} )= C(S^1\times \Sigma_{0,3} )$ that satisfy the relations
	\[
		\epsilon_1 \delta_1 + \epsilon_2\delta_2  + \epsilon_2\delta_3 + 2 \alpha + 2\varepsilon =0 
	\]
	for any $\epsilon_i\in \{-1,1\}$ with $\epsilon_1\epsilon_2\epsilon_3=1$. Taking the differences of two such relations gives $\delta_i + \delta_j=0$ for $i\neq j$, which leads to $\delta_i=0$ for $i\in \{1,2,3\}$. In particular, $v=0$.  
	
	Thus the $\mathcal{S}$-complex $\widetilde C(S^1\times \Sigma_{0,3} )$ splits as a direct sum of two copies of the $\mathcal{S}$-complex for the pair $(H,w_H)$ of the Hopf link, with the copies differing in grading by $2\pmod{4}$. The connected sum theorem applies just as in \cite[\S 8.2]{ds1} to yield the desired result.
\end{proof}

\subsection{Lagrangians from stabilized Heegaard splittings}\label{subsec:stabheegaardint}

Two embedded Lagrangians $\Lambda_1$ and $\Lambda_2$ in a symplectic manifold $M$ have {\emph{clean intersection}} if $\Lambda_1\cap \Lambda_2$ is a smooth submanifold of $M$, and for all $p\in \Lambda_1\cap \Lambda_2$,
\[
	T_p(\Lambda_1\cap \Lambda_2) = T_p\Lambda_1\cap T_p\Lambda_2.
\]

\begin{lemma}\label{lemma:pozniak}
	Suppose $\Lambda_1$ and $\Lambda_2$ are embedded Lagrangians in a symplectic manifold that have clean intersection, and choose a Morse function $f:\Lambda_1\cap \Lambda_2\to \R$. Then
	\[
		\mathfrak{n}(\Lambda_1,\Lambda_2) \leq |\textup{Crit}(f)|.
	\]
\end{lemma}

\begin{proof}
	Given the Morse function $f$, Pozniak shows in \cite[Theorem 3.4.11]{pozniak} how to perturb one of the Lagrangians by a small Hamiltonian isotopy so that the resulting intersection is transverse and in bijection with the critical points of $f$.
\end{proof}

Suppose $L$ is a link in a closed $3$-manifold $Y$, and $w\subset Y$ is a $1$-manifold satisfying $\partial w  = w\cap L$. We do not assume $(Y,L,w)$ is admissible. Choose a Heegaard decomposition 
\begin{equation}\label{eq:heegaarddecomplinkins3nondeg}
	(Y,L,w) = (\overline{C}_1,\overline{T}_1,w_1)\cup_{\Sigma_{g,n}} (\overline{C}_2,\overline{T}_2,w_2)
\end{equation}
where $w=w_1\cup w_2$; we may assume $w\cap \Sigma_{g,n}=\emptyset$, for homological reasons. Note $n$ must be even. Choose a basepoint on $L$ which lies on $\Sigma_{g,n}$, and form the connected sum of $(Y,L)$ with $S^1\times \Sigma_{0,3}$ with respect to this point and a basepoint on $S^1\times \{p_1\}$. We have an induced decomposition
\begin{equation}\label{eq:heegaarddecomplinkins3nondeg2}
	(Y,L,w)\# S^1\times \Sigma_{0,3} = (C_1, T_1,w_1) \cup_{\Sigma_{g,n+1} \sqcup  \Sigma_{0,3} } (C_2,T_2,w_2)
\end{equation}
where $(C_i,T_i)$ is formed by performing a boundary connect sum of $(\overline{C}_i,\overline{T}_i)$ with $I\times  \Sigma_{0,3}$, and these are glued in the obvious way. The compression bodies $(C_i,T_i,w_i)$ induce, in the familiar way, embedded Lagrangians $\Lambda_i = \mathfrak{X}(C_i,T_i,w_i)$ in $M_{g,n+1}$. 

As a variation of the traceless character variety $\mathfrak{X}(Y,L,w)$, define the representation variety
\[
	\mathfrak{R}(Y,L,w) = \{ \rho: \pi_1(Y\smallsetminus (L\cup w))\to SU(2) \; \mid \; \rho(\mu)=\left[ \begin{array}{cc} i & 0 \\ 0  & -i \end{array} \right],\; \rho( \nu ) = -1 \},
\]
where $\mu$ is a distinguished meridian around the basepoint of $L$, and $\nu$ ranges over all meridians of $w$. There is a circle action on $\mathfrak{R}(Y,L,w)$  by conjugating with respect to diagonal matrices of $SU(2)$, and the quotient is $\mathfrak{X}(Y,L,w)$. We may also view the elements of $\mathfrak{R}(Y,L,w)$ as singular flat connections ${\bf A}$ on $(Y,L,w)$ with holonomy approaching $\text{diag}(i,-i)$ around meridians of the preferred basepoint of $L$, modulo gauge transformations that preserve this condition; see \cite{ds1}.

The dimension of the Zariski tangent space of $[{\bf A}]\in\mathfrak{R}(Y,L,w)$ is 
\begin{equation}\label{eq:zariskicondition}
	\dim H^1_{\bf A} + 1- \dim H^0_{\bf A}.
\end{equation}
Note $\dim H^0_{\bf A} = \dim \textup{Stab}_{S^1}({\bf A})$ is either $0$ or $1$, depending on whether ${\bf A}$ is irredicible or not. We say $\mathfrak{R}(Y,L,w)$ is {\emph{Morse-Bott non-degenerate}} if it is a smooth manifold and the Zariski tangent space at each point has the same dimension as the component of $\mathfrak{R}(Y,L,w)$ in which the point lies. Equivalently, the relevant Chern--Simons functional on the space of singular connections for $(Y,L,w)$, framed at the basepoint of $L$, is Morse--Bott non-degenerate.

\begin{lemma}\label{lemma:morsebottclean}
	Suppose $\mathfrak{R}(Y,L,w)$  is Morse-Bott non-degenerate. Then for any Heegaard decomposition of $(Y,L,w)$ as in \eqref{eq:heegaarddecomplinkins3nondeg}, the induced Lagrangians $\Lambda_1,\Lambda_2\subset M_{g,n+1}$ have clean intersection. 
\end{lemma}

\begin{proof}
	As explained in \eqref{identification-Lag-crit}, $\Lambda_1\cap \Lambda_2$ can be identified with $\mathfrak{X}^\#=\mathfrak {X}((Y,L,w)\# S^1\times \Sigma_{0,3})$  modulo
	a free involution. 
	The Zariski tangent space of an element $[{\bf A}^\#]$ in $\mathfrak{X}^\#$ is equal to $H^1_{\bf A^\#}$. We have
	\[
		H^1_{\bf A^\#}= T_{\bf A^\# } \Lambda_1 \cap  T_{\bf A^\# } \Lambda_2,
	\]
	as follows from the Mayer-Vietoris exact sequence for the decomposition \eqref{eq:heegaarddecomplinkins3nondeg2} in cohomology with coefficients associated to the adjoint flat connections. From this we see that $\Lambda_1$ and $\Lambda_2$ intersect cleanly if and only if $\mathfrak{X}^\#$ is Morse-Bott non-degenerate.
	
	It remains to argue that $\mathfrak{X}^\#$ is Morse-Bott non-degenerate. First, there is a diffeomorphism
	\[
		\mathfrak{R}(Y,L,w)\times \{-1,+1\} \to \mathfrak{X}^\#
	\]
	which, from the viewpoint of flat connections, sends $([{\bf A}],\pm 1)$ to $[{\bf A}_{\pm}^\#]$, where ${\bf A}_{\pm}^\#$ is a singular flat connection on $(Y,L,w)\# S^1\times \Sigma_{0,3}$ obtained by gluing the unique singular flat connection ${\bf B}$ on $S^1\times \Sigma_{0,3}$ with holonomy $\pm 1$ around $S^1$ to the flat connection ${\bf A}$ on $(Y,L,w)$ using the framing at the basepoint of $L$. Thus $\mathfrak{X}^\#$ is a smooth manifold. To verify the Zariski tangent space condition, apply the Mayer-Vietoris sequence with local coefficients to the decomposition of $(Y,L,w)\# S^1\times \Sigma_{0,3} $ along the connected sum sphere $(S^2, 2 \text{ pts})$ to obtain the induced exact sequence 
	\[
		0 \to H^0_{\bf A} \to H^0_\theta \to H^1_{{\bf A}_{\pm}^\#} \to H^1_{{\bf A}}\to H^1_\theta
	\]
	where we have used the irreducibility and the non-degeneracy of ${\bf B}$, giving $H^i_{\bf B} =0$ for $i\in \{0,1\}$, and the irreducibility of ${\bf A}_{\pm}^\#$. Furthermore, $\theta$ is the unique singular flat connection on $(S^2, 2\text{ pts})$, and satisfies $H^0_\theta=\R$, $H^1_\theta=0$. It follows that $\dim H^1_{{\bf A}_{\pm}^\#}$ agrees with \eqref{eq:zariskicondition}.
\end{proof}

For a triple $(S^3,L,w)$, we say $w$ is non-trivial if $w$ intersects one of the components of $L$ in an odd number of points. Otherwise, we say $w$ is trivial. 
\begin{lemma}\label{lemma:detnonzeroinequality}
	Suppose $(S^3,L,w)$ has $\mathfrak{X}^\textup{irr}(S^3,L,w)$ non-degenerate. If $w$ is trivial and $\det(L)\neq 0$, then for the induced Lagrangians $\Lambda_1, \Lambda_2\subset M_{g,n+1}$ we have
	\[
		\mathfrak{n}(\Lambda_1, \Lambda_2) \leq 2^{|L|-1} + 2|\mathfrak{X}^\textup{irr}(S^3,L,w),
	\]
	where $|L|$ is the number of components of $L$. If $w$ is non-trivial , then we have 
			\[
		\mathfrak{n}(\Lambda_1, \Lambda_2) \leq 2|\mathfrak{X}^\textup{irr}(S^3,L,w)|.
	\]
\end{lemma}

\begin{proof}
If $w$ is trivial and $\det(L)\neq 0$, the character variety $\mathfrak{X}(S^3,L,w) $ has $2^{|L|-1}$ reducible classes, all non-degenerate \cite[Prop. 3.1]{ds3}. Together with the non-degeneracy assumption on $\mathfrak{X}^\textup{irr}(S^3,L,w)$, we obtain that $\mathfrak{R}(S^3,L,w)$ is Morse-Bott non-degenerate for any choice of basepoint on $L$, and is the disjoint union of $2^{|L|-1}$ points and $|\mathfrak{X}^\textup{irr}(S^3,L,w)|$ circles. The result then follows from Lemmas \ref{lemma:pozniak} and \ref{lemma:morsebottclean}, by choosing $f$ so that it is a standard Morse function with two critical points on each circle. The argument in the case that $w$ is non-trivial is the same, except that there are no reducibles in $\mathfrak{X}^\textup{irr}(S^3,L,w)$.
\end{proof}

\begin{theorem}\label{thm:minimalintsu2simple}
	If $K\subset S^3$ is a non-degenerate $SU(2)$-simple knot, then $\Lambda_1,\Lambda_2\subset M_{g,n+1}$ satisfy
	\[
		\mathfrak{n}(\Lambda_1, \Lambda_2) = \det(K).
	\]
\end{theorem}

\begin{proof}
	By Theorem \ref{thm:lagrangianineq}, the decomposition \eqref{eq:heegaarddecomplinkins3nondeg2}, and Proposition \ref{prop:connectedsumtheoremforinatural}, we have 
	\begin{equation}\label{eq:lagrankineqintro}
		\mathfrak{n}(\Lambda_1,\Lambda_2) \geq \frac{1}{2}\, \textup{rank}_\Z \; I((S^3,K)\# S^1\times \Sigma_{0,3} ) = \textup{rank}_\Z \; I^\natural(K) 
	\end{equation}
	Furthermore, by \cite{km-alexander}, $\textup{rank}_\Z \; I^\natural(K)\geq\det(K)$. The reverse inequality $\mathfrak{n}(\Lambda_1, \Lambda_2) \leq  \det(K)$ is implied by Lemma \ref{lemma:detnonzeroinequality} and the fact that for any knot $K$ the number of irreducible binary dihedral $SU(2)$ representations, up to conjugacy, is $(\det(K)-1)/2$; see \cite{klassen}.
\end{proof}

The basic examples of non-degenerate $SU(2)$-simple knots are two-bridge knots. For more details, including constructions of non-two-bridge examples, see \cite{zentner-simple}.

\subsection{Intersection numbers in $\Sigma_{0,5}$ and two-bridge links}\label{2-birdge-int-number}

We now turn to the proof of Theorem \ref{thm:preciseintersectionresult}. Let $(\gamma,\epsilon)$ and $(\gamma',\epsilon')$ be non-isotopic signed curves in $\Sigma_{0,5}$. Suppose that there is a disk $D\subset \Sigma_{0,5}$ containing exactly two of the marked points and which is disjoint from $\gamma$ and $\gamma'$. Referring to Figure \ref{five-lagrangians}, after a diffeomorphism of $\Sigma_{0,5}$ we may assume that $\gamma=\gamma_1$ and that the disk $D$ is the compact disk interior to $\gamma_5$. 

The isotopy class of $\gamma'\subset \Sigma_{0,5}\setminus D$ may be understood as follows. By collapsing $D$ to a point, we may view $\gamma'$ as a curve in the $4$-punctured sphere $\Sigma_{0,4}$. Isotopy classes of essential simple closed curves in $\Sigma_{0,4}$ are in correspondence with $\Q \cup \{\infty\}$, see \cite[\S 2.2.5]{farb-margalit}. In our setting, $\gamma'$ corresponds to the reduced fraction $p/q$ where, upon orienting $\gamma_1,\gamma_2, \gamma'$, the minimal oriented geometric intersection number of $\gamma'$ with $\gamma_1$ is $2p$, and that of $\gamma'$ with $\gamma_2$ is $2q$.

In this case, the closed pair $(Y,L) = ({C}_{\gamma},{T}_{\gamma}) \cup_{\Sigma_{0,5}\sqcup \Sigma_{0,3}} ({C}_{\gamma'},{T}_{\gamma'})$ may be identifed as
\begin{equation}\label{eq:connectedsumtwobridgeid}
	  (Y,L) = (S^3, L_{p/q}) \# S^1\times \Sigma_{0,3}
\end{equation}
where $L_{p/q}$ is a two-bridge link of type $p/q$, and the connected sum is of links. The particular components that are joined in the connected sum will not be important. 

To explain the identification \eqref{eq:connectedsumtwobridgeid}, view $\gamma$ and $\gamma'$ in $\Sigma_{0,4}$ by collapsing $D$ to a point. Let $(\overline{C}_\gamma,\overline{T}_\gamma)$ be a $4$-ball $\overline{C}_\gamma$ with tangle $\overline{T}_\gamma$ consisting of two $\partial_+$-parallel arcs such that $\partial(\overline{C}_\gamma,\overline{T}_\gamma)=\Sigma_{0,4}$ and the two pairs of points in $\Sigma_{0,4}$ which are divided by $\gamma$ are respectively connected by the arcs of $\overline{T}_\gamma$. Define $(\overline{C}_{\gamma'},\overline{T}_{\gamma'})$ similarly. Then gluing these $4$-balls, each containing a trivial $2$-tangle, in the obvious way gives a standard description of the two-bridge link $L_{p,q}$ via rational tangles:
\[
	(S^3, L_{p/q}) = (\overline{C}_{\gamma},\overline{T}_{\gamma}) \cup_{\Sigma_{0,4}} (\overline{C}_{\gamma'},\overline{T}_{\gamma'})
\]
The effect of remembering that one of the four points is the collapsing of the disk $D$, and replacing $(\overline{C}_\gamma,\overline{T}_\gamma)$ by the original compression body $(C_\gamma,T_\gamma)$ and similarly for $\gamma'$, we obtain the connected sum description of $(Y,L)$ given above. With this background we proceed to the proof.

\begin{proof}[Proof of  Theorem \ref{thm:preciseintersectionresult}]
If $p/q=0$, then $\gamma$ and $\gamma'$ are isotopic, and this is covered by Theorem \ref{displace-characterization}. If $p/q=\infty$, then after a diffeomorphism, $\gamma$ and $\gamma'$ are $\gamma_1$ and $\gamma_2$ in Figure \ref{five-lagrangians}, a case already treated above. Henceforth assume $p/q \in \Q\smallsetminus \{0\}$. Note $\det(L_{p/q})=|p|$ and $i(\gamma,\gamma')=2|p|$.

The Lagrangians $\Lambda_\gamma^\epsilon$, $\Lambda^{\epsilon'}_{\gamma'}$ in $M_{0,5}$ are exactly the sort considered in the previous section. If $p$ is odd, so that $L_{p/q}$ is a knot, then Theorem \ref{thm:minimalintsu2simple} directly applies to yield the desired result.

To handle the case in which $p$ is even and $w$ is possibly non-trivial, first note that by Theorem \ref{thm:lagrangianineq}, the decomposition \eqref{eq:connectedsumtwobridgeid}, and Proposition \ref{prop:connectedsumtheoremforinatural}, we have
		\begin{equation}\label{eq:lagrankineqintro}
		 \mathfrak{n}(\Lambda_\gamma^\epsilon,\Lambda^{\epsilon'}_{\gamma'})  \geq \frac{1}{2}\, \textup{dim} \; I(Y,L,w;\Z/2) = \textup{dim} \; I^\natural(L_{p/q},w;\Z/2).
	\end{equation}
Using \cite[Cor. 8.17]{ds1}, which holds for any choice of $w$, we have an isomorphism
	 \[
	 	I^\#(L_{p/q},w; \Z/2 ) \cong  I^\natural(L_{p/q}, w;\Z/2) \oplus I^\natural(L_{p/q}, w_0;\Z/2).
	 \] 
	 Then, by Corollary 1.5 of \cite{gong}, since $L_{p/q}$ is non-split alternating, $I^\#(L_{p/q},w; \Z/2 )$ does not depend on $w$, and hence neither does $I^\natural(L_{p/q},w; \Z/2 ) $. By the proof of Corollary 1.6 in \cite{km-unknot},
	\[
		\textup{dim} I^\natural(L_{p/q},w; \Z/2) = |p|,
	\]
	 since $L_{p/q}$ is non-split alternating and $\det(L_{p/q})=|p|$. Thus $\mathfrak{n}(\Lambda_\gamma^\epsilon,\Lambda^{\epsilon'}_{\gamma'}) \geq |p|$.
	 
	It remains to establish the reverse inequality $\mathfrak{n}(\Lambda_\gamma^\epsilon,\Lambda^{\epsilon'}_{\gamma'}) \leq |p|$. For this, we use
	 \[
	|\mathfrak{X}^{\textup{irr}}(L_{p/q},w)| = \frac{1}{2}(\det(L_{p/q}) - 2^{|L_{p/q}|-1}), 
	\]
	and that these character varieties are non-degenerate. This is standard in the case that $w$ is trivial, and one proof, just as in the case for knots, follows by passing to adjoint $SO(3)$ representations and pulling back to the branched double covers, which are lens spaces with cyclic fundamental group. The case of more general $w$ may be treated similarly, or we may alternatively appeal to \cite[Theorem 1.8]{gong}. The desired inequality then follows from Lemma \ref{lemma:detnonzeroinequality}.
\end{proof}

%!TEX root = main.tex

\section{Lagrangians in the odd character variety $M_g^{\text{odd}}$ }\label{sec:genus}

We turn to adapt some of the above methods to the case of the odd character varieties $M_g^{\text{odd}}$. In \S \ref{subsec:oddchdetection}, in parallel with results from \S \ref{sec:mappingclassgroupthm}, we study certain Lagrangians in $M_{g}^{\text{odd}}$ that come from compression bodies, and describe when two such Lagrangians can be displaced via Heegaard splitting data. As mentioned in the introduction, in the genus $2$ case, the symplectic manifold $M_2^{\text{odd}}$ may be identified with the intersection of two generic quadrics $Q_1$ and $Q_2$ in $\mathbb{CP}^2$ equipped with a monotone symplectic form. In \S \ref{subsec:genus2} we study this case, and prove Theorem \ref{thm:lagrangianspheresinquadricintersection}. Finally, in \S \ref{subsec:su2cyclic} we study the relationship between intersection numbers of the Lagrangians in $M_g^{\text{odd}}$ and $SU(2)$-cyclic $3$-manifolds, and derive an explicit result in the genus $2$ case.

\subsection{A general detection result for Lagrangians}\label{subsec:oddchdetection}

Assume $n$ is odd. Suppose $C$ is a compression body obtained from $(g-1)$-many $1$-handles to a thickened torus $\Sigma_1\times [0,1]$ along $\Sigma_1\times \{1\}$. Let $w = w'\sqcup w'' \subset C$ where $w'$ is the union of $n$ disjoint arcs $\{\text{pt}\}\times [0,1]$ (located away from the 1-handle attaching regions) and $w''$ is a closed $1$-manifold in the interior of $C$. Identifying $\partial_+(C,w)$ with $\Sigma_{g,n}$ we obtain a Lagrangian
\begin{equation}\label{eq:lagrangianoddcharactervariety}
	(S^3)^{g-1} \cong \mathfrak{X}(C,w) \subset M_g^{\text{odd}} := M_{g,n}(1,\ldots,1),
\end{equation}
where the marked points on $\Sigma_{g,n}$ are labelled with the conjugacy class of $-1\in SU(2)$. In particular, the holonomy of a connection in $\mathfrak{X}(C,w) $ around any meridian of $w$ is $-1$.  These Lagrangians were studied in \cite{dsodd}. In all of these constructions, the particular choice of $n$ is immaterial, apart from it being odd; the reader may prefer to assume $n=1$ throughout.

Recall that the sign-twisted curve complex ${\mathcal{C}}^\text{tw}(\Sigma_{g,n})$ has vertices equivalence classes $[\gamma,\epsilon]$ where $\gamma$ is a simple closed curve in $\Sigma_{g}\smallsetminus \mathcal{P}$ which is essential in $\Sigma_g$, and $(\gamma,\epsilon)\sim (\gamma',\epsilon')$ if and only if $\gamma$ and $\gamma'$ are isotopic through an isotopy which passes through $\mathcal{P}$ tranversely $m$ times, where the parity of $m$ agrees with the sign $\epsilon\epsilon'$. There is a simplicial involution induced by negation:
\[
	\sigma: {\mathcal{C}}^\text{tw}(\Sigma_{g,n}) \to {\mathcal{C}}^\text{tw}(\Sigma_{g,n}), \qquad \sigma \left( [\gamma,\epsilon ] \right) = [\gamma, - \epsilon].
\]
Also denote by $\mathcal{N}(\Sigma_g)$ the subcomplex of ${\mathcal{C}}^\text{tw}(\Sigma_{g,n}) $ spanned by non-separating curves of $\Sigma_g$. Thus
\[
	\mathcal{N}(\Sigma_g)^{(0)} = \left\{ [\gamma,\epsilon] \; \mid \; \Sigma_g\smallsetminus \gamma \text{ connected} \right\} \subset  {\mathcal{C}}^\text{tw}(\Sigma_{g,n})^{(0)} 
\]
Consider a compression body pair $(C,w)$ with $\partial_+(C,w)=\Sigma_{g,n}$ as above. For a properly embedded submanifold $A\subset C$ write $\partial_{\pm} A = A \cap \partial_\pm C$. Then $\partial A = \partial_- A \sqcup \partial_+ A$. Define
\begin{align*}
	\mathscr{N}(C,w) &= \{ [\gamma, \epsilon] \; \mid \gamma = \partial_+ D, \;  D \text{ a disk in } C, \; D\cdot w \equiv \tfrac{1}{2}(1-\epsilon\cdot 1) \, (\text{mod } 2) \} \subset \mathcal{N}(\Sigma_g)^{(0)}  \\[3mm]
	\underline{\mathscr{D}}(C) &= \{ [\gamma] \; \mid \; \gamma = \partial_+ D, \; D \text{ a disk in } C, \; \gamma \text{ non-separating }   \}  \subset \mathcal{C}(\Sigma_g)^{(0)} \\[3mm]
	\underline{\mathscr{A}}(C) &= \{ [\gamma] \; \mid \; \gamma = \partial_+ A, \; A \text{ an annulus in } C,  \; \partial_- A  \text{ essential in } \partial_- C \}  \subset \mathcal{C}(\Sigma_g)^{(0)}
\end{align*}
Note that there is natural $2$-to-$1$ map from ${\mathcal{C}}^\text{tw}(\Sigma_{g,n})^{(0)}$ to $\mathcal{C}(\Sigma_g)^{(0)}$ which forgets signs. Under this map, $\mathscr{N}(C,w) $ is sent bijectively to $\underline{\mathscr{D}}(C)$.

For $(C_1,w_1)$, $(C_2,w_2)$ two such compression body pairs as above, form the closed pair
\begin{equation*}  
	(Y,w) = (C_1,w_1) \cup_{\Sigma_{g,n} \sqcup \Sigma_{1,1} } (C_2,w_2)
\end{equation*}
Use the abbreviations $\mathscr{N}_i = \mathscr{N}(C_i,w_i)$, $\underline{\mathscr{D}}_i = \underline{\mathscr{D}}(C_i)$, and $\underline{\mathscr{A}}_i= \underline{\mathscr{A}}(C_i)$. The following detection result is an analogue of Proposition \ref{prop:instantonheegaarddetection}.

\begin{prop}
	Suppose $(Y,w)$ is formed by gluing two compression bodies as above. Then
	\[
		I(Y,w) = 0  \quad \Leftrightarrow \quad \left(\mathscr{N}_1\cap \sigma\left(\mathscr{N}_2\right) \right) \cup \left( \underline{\mathscr{D}}_1 \cap \underline{\mathscr{A}}_2\right)  \cup \left( \underline{\mathscr{A}}_1 \cap \underline{\mathscr{D}}_2\right) \neq \emptyset.
	\]
\end{prop}

\begin{proof}
First, suppose $I(Y,w)=0$. By Theorem \ref{thm:detectionintro}, there is a $2$-sphere $S\subset Y$ which has $[S]\cdot [w]$ odd. Denote by $F\subset Y$ the torus obtained by gluing the tori $\partial_- C_i \cong \Sigma_1$. Suppose $S\subset Y$ is a $2$-sphere with $[S]\cdot [w]$ odd for which $F$ is transverse to $S$ and the number of components in $S\cap F$ is minimal. If some component of $S\cap F$ bounds a disk in $F$, surgery on $S$ along this disk yields two $2$-spheres, both of which have less components of intersection with $F$ than $S$, and one of which has odd intersection with $w$, a contradiction. Thus no components in $S\cap F$ bound a disk in $F$.  

Thus $S\cap F$, if non-empty, consists of a collection of parallel disjoint copies of some essential simple closed curve $\gamma'\subset F$. Cutting $S$ along $S\cap F$ gives a union of planar surfaces, and among these at least one component  is a disk. We may view this disk $D$ inside
\[
	 Y^\circ = C_1  \cup_{\Sigma_{g,n} } C_2
\] 
which, from another viewpoint, is $Y$ cut along $F$. From $D$ we obtain, using \cite[Lemma 1.1]{casson-gordon}, a disk $D' \subset Y^\circ$ which intersects $\Sigma_g$ in a single essential simple closed curve $\gamma$. In particular, $D'$ is either divided into a disk in $C_1$ and an annulus $A$ in $C_2$ with $\partial_+ A = \gamma$ and $\partial_- A =\gamma'$, implying
\[
	[\gamma] \in \underline{\mathscr{D}}_1 \cap \underline{\mathscr{A}}_2,
\]	
or $D'$ is divided by an annulus in $C_1$ and a disk in $C_2$, implying $[\gamma] \in \underline{\mathscr{A}}_1 \cap \underline{\mathscr{D}}_2$. Note that $\gamma$ is non-separating because every curve in $\underline{\mathscr{A}}_i$ is automatically non-separating. Indeed, an annulus $A\subset C_i$ having $\partial_- A\subset \partial_- C_i$ an essential closed curve guarantees that the simple closed curve $\partial_+ A\subset \partial_+ C_i$ is homologically non-trivial, hence non-separating.

If on the other hand $S\cap F$ is empty, then $S$ may be viewed as a $2$-sphere in $Y^\circ$. Applying \cite[Lemma 1.1]{casson-gordon}, we obtain a $2$-sphere $S'$ in $Y^\circ$ which intersects $\Sigma_g$ in a single essential simple closed curve $\gamma$ in $\Sigma_g$. The curve $\gamma$ divides $S'$ into two disks, one in $C_1$ and another in $C_2$. The condition that the union of these disks has odd intersection with $w$ implies that $\gamma$ is non-separating and $ [\gamma,\epsilon]\in  \mathscr{N}_1\cap \sigma\left(\mathscr{N}_2\right)$ for some sign $\epsilon$.

Now consider the converse. We explain how each of the conditions gives rise to a $2$-sphere $S\subset Y$ with $[S]\cdot [w]$ odd, which implies $I(Y,w)=0$.  First, given $[\gamma,\epsilon]\in \mathscr{N}_1\cap \sigma\left(\mathscr{N}_2\right)$, gluing the relevant disks $D_i\subset C_i$ with $\partial D_i = \gamma$ gives the desired $2$-sphere. Next, given $[\gamma]\in \underline{\mathscr{D}}_1 \cap \underline{\mathscr{A}}_2$, glue the relevant disk and annulus to obtain an embedded disk $D$ in $Y$ such that $\partial D$ is an essential simple closed curve on $F=\Sigma_1$ away from the marked point. Take $D'$ to be a parallel pushoff of $D$ whose boundary is a small parallel pushoff of $\partial D$ on $F$. In particular, $D\cap D'=\emptyset$. Then $\partial D$ and $\partial D'$ cobound two annuli on $F$. Choose the annulus $A$ that contains the marked point of $F$. Then gluing $D\cup A \cup D'$ gives the desired $2$-sphere. The case $ \underline{\mathscr{D}}_2 \cap \underline{\mathscr{A}}_1\neq \emptyset$ is similar.
\end{proof}

\begin{corollary}\label{cor:lagdisploddchar}
	Given two compression body pairs $(C_1,w_2)$ and $(C_2,w_2)$ as above, the displacibility of the associated Lagrangians $\Lambda_i = \mathfrak{X}(C_i,w_i)\subset M_{g}^\textup{odd}$ is determined by:
	\[
		\mathfrak{n}(\Lambda_1,\Lambda_2) = 0  \quad \Leftrightarrow \quad \left(\mathscr{N}_1\cap \sigma\left(\mathscr{N}_2\right) \right) \cup \left( \underline{\mathscr{D}}_1 \cap \underline{\mathscr{A}}_2\right)  \cup \left( \underline{\mathscr{A}}_1 \cap \underline{\mathscr{D}}_2\right) \neq \emptyset.
	\]
\end{corollary}

\subsection{Lagrangians in the intersection of two quadrics in $\mathbb{C}\mathbb{P}^5$}\label{subsec:genus2}

For $g=2$, the construction \eqref{eq:lagrangianoddcharactervariety} gives Lagrangian spheres in $M_2^{\text{odd}}$. These may be described alternatively as follows. Start with $\Sigma_{2,n} = (\Sigma_2,\mathcal{P})$ and choose a non-separating simple closed curve $\gamma\subset \Sigma_{2}\smallsetminus \mathcal{P}$ which is essential on $\Sigma_2$, the underlying closed surface. Attach a two-handle to $\Sigma_2\times [0,1]$ to obtain a compression body $C_\gamma$. Choose a circle $w'_-$ in the interior of $C_\gamma$, disjoint from $\mathcal{P}\times [0,1]$, which intersects the core of the two-handle transversely in one point. Let $w'_+=\emptyset$. For a given sign $\epsilon\in \{-,+\}$ define $w_\epsilon = w'_\epsilon \cup \mathcal{P}\times [0,1]$. We obtain a Lagrangian $3$-sphere
\begin{equation}\label{eq:lagrangiangenus2desc}
	  \Lambda_\gamma^\epsilon	:= \mathfrak{X}(C_\gamma,w_\epsilon ) \subset M_2^{\text{odd}}
\end{equation}
This Lagrangian may be described as the flat $SU(2)$ connections in $M_2^{\text{odd}}$ whose holonomy around $\gamma$ is $\epsilon \cdot \text{id}$. If $\gamma$ is isotoped to a curve $\gamma'$ by an isotopy that crosses a single marked point in $\mathcal{P}$ transversely, then the holonomy of a given connection in $\Lambda_\gamma^\epsilon$ along $\gamma'$ is $-\epsilon \cdot \text{id}$.

\begin{lemma} \label{lemma:quintic}
For the compression body pair $(C_\gamma, w_\epsilon)$, we have:
	\begin{enumerate}[label=(\roman*)]
		\item $\mathscr{N}(C_\gamma,w_{\epsilon}) = \{ [\gamma,\epsilon] \}$. \label{lemma:quintic1}
		\item $\underline{\mathscr{D}}(C_\gamma) = \{\gamma\}$. \label{lemma:quintic2}
		\item Every $\gamma'\subset \Sigma_g$ which bounds a disk $D'\subset C_\gamma$ may be displaced from $\gamma$. \label{lemma:quintic2.5}
		\item Every $\gamma' \in \underline{\mathscr{A}}(C_\gamma)$ may be displaced from $\gamma$ and is not isotopic to $\gamma$. \label{lemma:quintic3}
	\end{enumerate}
\end{lemma}

\begin{proof}
	Similar to the proof of Lemma \ref{lemma:hamiltoniandisplacementviacurvesets}, we begin with the observation
	\[
		\pi_1(\Sigma_2 \smallsetminus \mathcal{P} ) \cong F_{3+n} , \qquad \pi_1(C_\gamma \smallsetminus \mathcal{P}  \times [0,1]) \cong F_{2+n}
	\]
	where the latter fundamental group is the quotient of the former by the normal closure of $\gamma$. Let $[\gamma', \epsilon']\in \mathscr{N}(C_\gamma,w_{\epsilon}) $. If $\gamma'$ bounds a disk in $C_\gamma \smallsetminus \mathcal{P}  \times [0,1]$ then, since $\gamma'$ is an essential simple closed curve which is non-separating, the argument of Lemma \ref{lemma:hamiltoniandisplacementviacurvesets} \ref{item:lemmadiskann1} implies that $\gamma'$ is isotopic to $\gamma$, and $\epsilon=\epsilon'$ since the isotopy does not cross any marked points. In the more general case, we have that $\gamma'$ bounds a disk $D\subset C_\gamma$ that intersects $\mathcal{P}\times [0,1]$ transversely in $m$ points where the parity of $m$ agrees with $\epsilon \epsilon'$. The argument of Lemma \ref{lemma:hamiltoniandisplacementviacurvesets} \ref{item:lemmadiskann2} shows that $\gamma'$ may be related to $\gamma$ by performing $m$ consecutive band sums, at each step involving a small loop around one of the marked points. In particular, $\gamma'$ is isotopic to $\gamma$ as a curve on $\Sigma_2$ through an isotopy that passes through $\mathcal{P}$ transversely $m$ times. Thus $[\gamma,\epsilon]=[\gamma',\epsilon']$. This proves \ref{lemma:quintic1}, which also implies \ref{lemma:quintic2}.

	Next, we consider \ref{lemma:quintic2.5}. If $\gamma'$ is non-separating, this follows from the previous paragraph, so assume $\gamma'$ is separating. Isotope $\gamma'$ to have minimal intersection with $\gamma$ and suppose $\gamma\cap \gamma'\neq \emptyset$. Let $D\subset C_\gamma$ be a disk with $\partial D=  \gamma$ which intersects $D'$ transversely. If there is a circle component in $D\cap D'$, then an innermost circle component in $D\cap D'\subset D$ bounds a disk in $D$ with ints interior disjoint from $D'$. Doing surgery on $D'$ along such a disk produces a disk $D''$ with $\partial D'' = \gamma'$ that intersects $D$ in a smaller number of components. For this reason, we may assume $D\cap D'$ has no circle components. Thus there exists an outermost arc in $D\cap D'\subset D$ which bounds a disk in $D$. Do surgery on $D'$ along this disk to obtain two disks $D_1$ and $D_2$ in $C_\gamma$ and write $\gamma_i  = \partial D_i$. Thus $\gamma'$ is a band sum of two disjoint curves $\gamma_1$ and $\gamma_2$. By the minimality assumption on $\gamma\cap \gamma'$, each of $\gamma_1$ and $\gamma_2$ is essential.  Since two disjoint separating essential simple closed curves on $\Sigma_2$ are isotopic, and necessarily have null-homotopic band sum, one of $\gamma_1$ or $\gamma_2$ must be non-separating. In fact, the triviality of the homology class of $\gamma'$ implies that both $\gamma_1$ and $\gamma_2$ are non-separating. By \ref{lemma:quintic2}, $\gamma_1$, $\gamma_2$ are isotopic to $\gamma$. Thus after an isotopy, we may assume that $\gamma'$ is the band sum of $\gamma$ and another curve disjoint from $\gamma$. Now if we push $\gamma$ in this band sum in the direction of the band, we obtain a curve isotopic to $\gamma'$ and disjoint from $\gamma$.
	
	Finally, we prove \ref{lemma:quintic3}. Isotope $\gamma'$ to have minimal intersection with $\gamma$ and suppose $\gamma\cap \gamma'\neq\emptyset$. Let $A\subset C_\gamma$ be an annulus with boundary components $\gamma' = \partial_+ A$ and $\partial_-A$. Our assumption on $\underline{\mathscr{A}}(C_\gamma)$ implies that the homology class $[\partial_-A]=[\gamma']$ in $C_\gamma$ is non-trivial. Since $\gamma$ represents the trivial homology class in $C_\gamma$, the closed curve $\gamma' $ is not isotopic to $\gamma$. Let $D\subset C_\gamma$ be a disk which intersects $A$ transversely and $\partial D = \gamma$. As in the previous paragraph, if $A\cap C$ contains a circle component, we may do a surgery on $A$ along a disk $D'\subset D$ that $D'\cap A=\partial D'$. If $D'\cap A$ is essential in $A$, then this surgery produces a disc with boundary $\gamma'$. In particular, $\gamma'$ is null-homologous, which is a contradiction as we agued above. Thus this surgery produces another annulus with the same boundary components as $A$ that intersects the disk $D$ in smaller number of components. Iterating this surgery process, we may assume that $A\cap D$ has no circle components. 
	
	There is an outermost arc in $A\cap D\subset D$ which bounds a disk in $D$. Surgery on $A$ along this disk gives a disk $D_1$ and an annulus $A_2$. Write $\gamma_1=\partial D_1$, $\gamma_2 = \partial_+ A_2$. By \ref{lemma:quintic2}, $\gamma_1$ is either isotopic to $\gamma$ or separating. In the former case and after applying an isotopy, we conclude that $\gamma'$ is a band sum of $\gamma$, $\gamma_2$, and as in the proof of \ref{lemma:quintic2.5} we may displace $\gamma'$ from $\gamma$. In the latter case, the minimality assumption on $\gamma\cap \gamma'$ implies that $\gamma_1$ is an essential separating curve.  By \ref{lemma:quintic2.5} and after applying an isotopy,
we may assume that $\gamma_1$ is disjoint from $\gamma$. The space $C_\gamma\smallsetminus D_1$ contains two connected components and their closures in $C_\gamma$ can be identified with a solid torus and a thickened torus $I\times \Sigma_1$. Since $A_2$ intersects $\partial_-C_\gamma$ non-trivially, $A_2$ is a subset of the thickened torus component of $C_\gamma\smallsetminus D_1$. The homology class of $\gamma$ in $C_\gamma$ is trivial, and hence $\gamma$ is contained in the solid torus component of $C_\gamma\smallsetminus D_1$. In particular, $\gamma$ and $\gamma_2$ are in different connected components of $\Sigma_2\smallsetminus \gamma_1$, and the curve $\gamma'$ obtained by a band sum of $\gamma_1$ and $\gamma_2$ is disjoint from $\gamma$. This proves \ref{lemma:quintic3}.	
\end{proof} 

Recall $\mathcal{N}(\Sigma_2)$ is the subcomplex of $ {\mathcal{C}}(\Sigma_{g,n})^\text{tw}$ spanned by vertices $[\gamma,\epsilon]$ with $\gamma$ non-separating. The construction \eqref{eq:lagrangiangenus2desc} induces a well-defined map
\begin{equation}\label{vertex-map}
	  \mathcal{N}(\Sigma_2)^{(0)} \to \mathcal{L}(M_{2}^{\text{odd}})^{(0)}
\end{equation}
by associating the Lagrangian sphere $\Lambda_\gamma^\epsilon$ to the equivalence class $[\gamma,\epsilon]$. We now prove that this map induces an injection of simplicial complexes, which implies Theorem \ref{thm:lagrangianspheresinquadricintersection}.

\begin{proof}[Proof of Theorem \ref{thm:lagrangianspheresinquadricintersection}]
First we prove an analogue of Theorem \ref{displace-characterization} in the present setup by showing that for $[\gamma_1,\epsilon_1] \neq [\gamma_2,\epsilon_2]$, the Lagrangians $\Lambda_1 = \Lambda_{\gamma_1}^{\epsilon_1}$ and $\Lambda_2= \Lambda_{\gamma_2}^{\epsilon_2}$ can be displaced by a Hamiltonian isotopy in $M_g^{\text{odd}}$ if and only if $\gamma_1$ and $\gamma_2$ can be made disjoint from each other by an isotopy. By Corollary \ref{cor:lagdisploddchar}, $\Lambda_1$ and $\Lambda_2$ are displaceable by a Hamiltonian isotopy if and only if
\[
\left(\mathscr{N}_1\cap \sigma\left(\mathscr{N}_2\right) \right) \cup \left( \underline{\mathscr{D}}_1 \cap \underline{\mathscr{A}}_2\right)  \cup \left( \underline{\mathscr{A}}_1 \cap \underline{\mathscr{D}}_2\right) \neq \emptyset.
\]
By Lemma \ref{lemma:quintic} \ref{lemma:quintic1}, we have $\mathscr{N}_i=\{[\gamma_i,\epsilon_i]\}$. Thus if $\mathscr{N}_1\cap \sigma\left(\mathscr{N}_2\right) \neq\emptyset$, then $\gamma_1$ and $\gamma_2$ are isotopic, and hence are displaceable on $\Sigma_2$.  Lemma \ref{lemma:quintic}  \ref{lemma:quintic2} implies $\underline{\mathscr{D}}_1 = \{\gamma_1\}$ and Lemma \ref{lemma:quintic}  \ref{lemma:quintic3} implies any curve in $\underline{\mathscr{A}}_2$ is displaceable from $\gamma_2$. Thus if $\underline{\mathscr{D}}_1\cap \underline{\mathscr{A}}_2  \neq\emptyset$, then $\gamma_1$ and $\gamma_2$ are displaceable. The case $\underline{\mathscr{D}}_2\cap \underline{\mathscr{A}}_1  \neq\emptyset$ is similar. For the converse, suppose $\gamma_1$ and $\gamma_2$ are disjoint. If they are isotopic, then $\epsilon_1=-\epsilon_2$ and $\mathscr{N}_1\cap \sigma(\mathscr{N}_2)\neq \emptyset$. If $\gamma_1$ and $\gamma_2$ are not isotopic, we see that both $\underline{\mathscr{D}}_1\cap \underline{\mathscr{A}}_2$ and $\underline{\mathscr{D}}_2\cap\underline{\mathscr{A}}_1$ are non-empty by Lemma \ref{lemma:quintic}  \ref{lemma:quintic2} and \ref{lemma:quintic3}.

Now the rest of the proof follows a similar plan as in the proof of Theorem \ref{thm:lagrangianspheres-gen}. To show that \eqref{vertex-map} is injective, let $[\gamma_1,\epsilon_1]$, $[\gamma_2,\epsilon_2]$ be vertices in $\mathcal{N}(\Sigma_2)^{(0)}$ that induce Hamiltonian isotopic Lagrangians $ \Lambda_{\gamma_1}^{\epsilon_1}$ and $ \Lambda_{\gamma_2}^{\epsilon_2}$. Thus, $ \Lambda_{\gamma_1}^{\epsilon_1}$ and $ \Lambda_{\gamma_2}^{-\epsilon_2}$ are displaceable because $\Lambda_{\gamma_2}^{\epsilon_2}$ and $\Lambda_{\gamma_2}^{-\epsilon_2}$  can be displaced by a Hamiltonian isotopy. The previous paragraph implies that $\gamma_1$ and $\gamma_2$ can be made disjoint from each other by an isotopy. If $\gamma_1$ and $\gamma_2$ are not isotopic to each other, then by another application of the previous paragraph, the Hamiltonian isotopic Lagrangians $ \Lambda_{\gamma_1}^{\epsilon_1}$ and $ \Lambda_{\gamma_2}^{\epsilon_2}$ are displaceable. This is a contradiction because it can be checked easily using Corollary \ref{cor:lagdisploddchar} and Lemma \ref{lemma:quintic} that a Lagrangian sphere of the form $\Lambda_{\gamma}^{\epsilon}$ in $M_g^{\text{odd}}$ cannot be displaced from itself. Thus, $\gamma_1$ and $\gamma_2$ are isotopic. This implies that $[\gamma_1,\epsilon_1] =[\gamma_2,\epsilon_2]$, otherwise $ \Lambda_{\gamma_1}^{\epsilon_1}$ and $ \Lambda_{\gamma_2}^{\epsilon_2}$ are displaceable which leads again to a contradiction. Finally we may also check that the map \eqref{vertex-map} preserves simplices by observing as in the Proof of Theorem \ref{thm:lagrangianspheres-gen} that $ \mathcal{N}(\Sigma_2)$ and $\mathcal{L}(M_{2}^{\text{odd}})$ are both flag simplicial complexes and then using again the result of the previous paragraph.
\end{proof}

\subsection{Intersections of Lagrangians in $M_g^{\text{odd}}$ and $SU(2)$-cyclic $3$-manifolds}\label{subsec:su2cyclic}

Given a closed $3$-manifold $Y$ with $1$-manifold $w\subset Y$, consider a Heegaard decomposition
\[
	(Y,w) = (\overline{C}_1,\overline{w}_1) \cup_{\Sigma_g}  (\overline{C}_2,\overline{w}_2)
\]
where $\overline{C}_1$ and $\overline{C}_2$ are handlebodies of genus $g$ with boundary $\Sigma_g$, and $\overline{w}_i$ is a closed $1$-manifold in the interior of $\overline{C}_i$. Choose a basepoint in $Y$ that lies on $\Sigma_g$ and form the connected sum of $(Y,w)$ with $S^1\times \Sigma_{1,1}$. We have an induced decomposition
\[
	(Y,w)\# S^1\times \Sigma_{1,1} = (C_1,w_1) \cup_{\Sigma_{g+1,1} \sqcup \Sigma_{1,1} } (C_2,w_2)
\]
where $C_i$ is formed by a boundary sum of $\overline{C}_i$ and $I\times \Sigma_{1,1}$ and $w_i$ is the union of $\overline{w}_i$ and $I\times \{p\}$ where $p$ is the marked point on $\Sigma_{1,1}$. We have induced Lagrangians 
\[
	\Lambda_i = \mathfrak{X}(C_i,w_i)
\subset M_{g+1}^{\text{odd}}
\]
 Call $(Y,w)$ {\emph{non-degenerate $SU(2)$-cyclic}} if $b_1(Y)=0$ and $\mathfrak{X}(Y,w)$ consists entirely of representations with abelian image, all of which are non-degenerate.

\begin{lemma}\label{lemma:lagrangianssu2cyclic}
Consider Lagrangians $\Lambda_1$ and $\Lambda_2$ in $M_g^{\textup{odd}}$ associated to a Heegaard decomposition of $(Y,w)$ as above. If $b_1(Y)=0$, then we have the inequality
\[
	\mathfrak{n}(\Lambda_1,\Lambda_2) \geq | H_1(Y;\Z) |.
\]
If moreover $(Y,w)$ is non-degenerate $SU(2)$-cyclic, then equality holds.
\end{lemma}

\begin{proof}
	For the first statement, an application of Theorem \ref{thm:lagrangianineq} adapted to the present setting gives 
		\begin{equation*} 
		\mathfrak{n}(\Lambda_1,\Lambda_2) \geq \frac{1}{2}\, \textup{rank}_\Z \; I((Y,w)\# S^1\times \Sigma_{1,1}) = \textup{rank}_\Z \; I^\#(Y,w) \geq |H_1(Y;\Z)|,
	\end{equation*}
	where the middle equality is the definition of framed instanton homology $I^\#(Y,w)$, and the last inequality follows from the Euler characteristic computation $\chi( I^\#(Y,w))=|H_1(Y;\Z)|$ in the case $b_1(Y)=0$, see \cite[Corollary 1.4]{scaduto-thesis}.
	
	For the second statement, consider the representation variety
	\[
	\mathcal{R}(Y,w) = \{ \rho: \pi_1(Y\smallsetminus w ) \to SU(2) \; \mid \; \rho(\nu) = -1\},
\]
which has an $SO(3)$-action induced by conjugation, and the quotient is $\mathfrak{X}(Y,w)$. There is a natural double covering from $\mathfrak{X}((Y,w)\# S^1\times \Sigma_{1,1})$ to $\mathcal{R}(Y,w) $. In particular, $\Lambda_1\cap \Lambda_2$ is identified with $\mathcal{R}(Y,w)$. Moreover, by a similar argument as given for Lemma \ref{lemma:morsebottclean}, $\Lambda_1\cap \Lambda_2$ is a clean intersection if and only if $\mathcal{R}(Y,w) $ is Morse-Bott non-degenerate.

Assume $(Y,w)$ is non-degenerate $SU(2)$-cyclic. Then $\mathcal{R}(Y,w) $ is automatically Morse-Bott non-degenerate. If $[w]=0$ in $H_1(Y;\Z/2)$, then $\mathcal{R}(Y,w)$ consists of $|H_1(Y;\Z/2)|$ points and $ (|H_1(Y;\Z)|-|H_1(Y;\Z/2)|)/2$ many $2$-spheres. If $[w]\neq 0$ in $H_1(Y;\Z/2)$, then $\mathcal{R}(Y,w)$ consists of $|H_1(Y;\Z)|/2$ many $2$-spheres. See \cite[Proposition 2.7]{mme}. In either case, by choosing a Morse function which on each two-sphere has exactly $2$ critical points, we obtain from Lemma \ref{lemma:pozniak} that  $\mathfrak{n}(\Lambda_1,\Lambda_2) \leq | H_1(Y;\Z) |$, completing the proof.
\end{proof}

\begin{theorem}\label{thm:int-odd-g=2}
	Suppose $[\gamma_1,\epsilon_1]$ and $[\gamma_2,\epsilon_2]$ are distinct vertices in $\mathcal{N}(\Sigma_2)$. In particular, $\gamma_1$ and $\gamma_2$ are non-separating essential simple closed curves in $\Sigma_2$. Suppose moreover that there is a separating essential simple closed curve in $\Sigma_2$ which is disjoint from $\gamma_1$ and $\gamma_2$. Then
	\[
		\mathfrak{n}(\Lambda_{\gamma_1}^{\epsilon_1}, \Lambda_{\gamma_2}^{\epsilon_2}) = i(\gamma_1,\gamma_2).
	\]
\end{theorem}

\begin{proof}
The proof is similar to that of Theorem \ref{thm:preciseintersectionresult}. First, using the separating curve we may write $\Sigma_2$ as the connected sum of two genus $1$ surfaces $\Sigma_1'$ and $\Sigma_1''$. If $\gamma_1$ and $\gamma_2$ are separated by this curve then they are disjoint, and the result follows from Theorem \ref{thm:lagrangianspheresinquadricintersection}. So we may assume both $\gamma_1$ and $\gamma_2$ are on the side of the $\Sigma_1'$-summand. Using a diffeomorphism we may write $\Sigma_1=S^1\times S^1$ and, viewing $\gamma_1$ and $\gamma_2$ as curves in $\Sigma_1$, we may assume $\gamma_1$ and $\gamma_2$ are in the homology classes $\pm (0,1)$ and $\pm(p,q)$ in $H_1(\Sigma_1;\Z)\cong \Z\oplus \Z$.

In gluing the two compression bodies defining the Lagrangians, we then have
\[
	(C_{\gamma_1},w_{\epsilon_1}) \cup_{\Sigma_{2,n} \sqcup \Sigma_{1,1}} (C_{\gamma_2},w_{\epsilon_2}) \cong (  L(p,q), w ) \# S^1\times \Sigma_{1,1}
\] 
where $L(p,q)$ is a lens space, $w\subset L(p,q)$ is some closed $1$-manifold. The result then follows from Lemma \ref{lemma:lagrangianssu2cyclic} and the fact that $(L(p,q),w)$ is non-degenerate $SU(2)$-cyclic.
\end{proof}


\Addresses

\end{document}